\documentclass[11pt,leqno]{amsart}
\usepackage{enumerate}
\usepackage{tikz}

\usepackage{amssymb}
\usepackage[english]{babel}
\usepackage{amssymb,amsthm,amsmath,eucal,mathrsfs}
\usepackage{bm}

\usepackage{amsmath}
\usepackage{amsfonts}
\usepackage{amsmath}
\usepackage{amssymb}
\usepackage{graphicx}
\usepackage{lscape}
\usepackage{amstext}
\usepackage{amsthm}
\usepackage{color}
\usepackage{float}
\usepackage{mathrsfs}
\usepackage{epsfig}
\usepackage{url}
\usepackage{fancyhdr}
\usepackage{pspicture}
\usepackage{graphicx}

\usepackage{cite}

\usepackage{amsmath,verbatim}

\usepackage{amsthm}

\usepackage{amssymb}

\usepackage{amsfonts}

\usepackage{dsfont}

\usepackage{hyperref}

\newtheorem{theorem}{Theorem}[section]

\newtheorem{lemma}[theorem]{Lemma}

\newtheorem{corollary}[theorem]{Corollary}

\newtheorem{assumption}{Assumption}

\theoremstyle{definition}

\newtheorem{remark}[theorem]{Remark}

\numberwithin{equation}{section}

\newcommand{\dist}{\mathrm{dist}}      

\newcommand{\diam}{\mathrm{diam}}      

\newcommand\restr[2]{{
  \left.\kern-\nulldelimiterspace 
  #1 
  \vphantom{\big|} 
  \right|_{#2} 
  }}

\usepackage{eucal}

\begin{document}
\title[The Calderón problem revisited]{The Calderón problem revisited:\\
Reconstruction with resonant perturbations}

\author[Ghandriche and Sini]{Ahcene Ghandriche  $^*$ and Mourad Sini$^{\ddag}$}
\thanks{$^*$ Nanjing Center for Applied Mathematics, Nanjing, 211135, People’s Republic of China. Email: gh.hsen@njcam.org.cn}
\thanks{$^{\ddag}$ RICAM, Austrian Academy of Sciences, Altenbergerstrasse 69, A-4040, Linz, Austria. Email: mourad.sini@oeaw.ac.at. This author is partially supported by the Austrian Science Fund (FWF): P 30756-NBL}


\allowdisplaybreaks

\begin{abstract}

The original Calderón problem consists in recovering the potential (or the conductivity) from the knowledge of the related Neumann to Dirichlet map (or Dirichlet to Neumann map). Here, we first perturb the medium by injecting small-scaled and highly heterogeneous particles. Such particles can be bubbles or droplets in acoustics or nanoparticles in electromagnetism. They are distributed, periodically for instance, in the whole domain where we want to do reconstruction. Under critical scales between the size and contrast, these particles resonate at specific frequencies that can be well computed. 
Using incident frequencies that are close to such resonances, we show that
\bigskip

\begin{enumerate}
\item the corresponding Neumann to Dirichlet map of the composite converges to the one of the homogenized medium. In addition, the equivalent coefficient, which consist in the sum of the original potential and the effective coefficient, is negative valued with a controlable amplitude.
\medskip

\item as the equivalent coefficient is negative valued,  then we can linearize the corresponding Neumann to Dirichlet map using the effective coefficient's amplitude.
\medskip

\item from the linearized Neumann to Dirichlet map, we reconstruct the original potential using explicit complex geometrical optics solutions (CGOs).

\end{enumerate}

\end{abstract}


\maketitle

\tableofcontents
\addtocontents{toc}{~\hfill\textbf{Page}\par}
\addtocontents{toc}{\setcounter{tocdepth}{2}}
\section{Introduction and statement of the results}
\subsection{General introduction}

The original Calderón problem stated in the acoustic framework reads as follows. Let $n:=c^{-1}$ be the index of refraction, where $c$ stands for the speed of acoustic sound. In turn, this speed of sound is given by $c:=\sqrt{\dfrac{k}{\rho}}$ where $\rho$ is the mass density and $k$ is the bulk modulus. In the time-harmonic regime, the propagation of the acoustic waves is modeled by:
\begin{align}\label{EquaKg-introdution}
\begin{cases}  
\left( \Delta  + \omega^{2} \, n^2(\cdot) \right) \, p^{f} = 0 \quad \text{in} \quad  \Omega,  \\ 
\qquad \quad \qquad \, \partial_{\nu} p^{f} \; =  f \;\;\;  \text{on} \quad \partial \Omega. 
\end{cases}
\end{align}
where $p^f$ is the acoustic pressure generated by the applied source $f$. The Neumann to Dirichlet (NtD) operator $\Lambda_c$ corresponds to any $f \in \mathbb{H}^{-\frac{1}{2}}(\partial \Omega)$, the trace on $\partial \Omega$ of the induced pressure $p^{f}$, i.e. $\Lambda_{c}(f) := p^{f}_{|_{\partial \Omega}}$.
The Calderón problem consists in recovering the sound speed $c$ from the knowledge of the NtD map $\Lambda_c$. According to the model $(\ref{EquaKg-introdution})$, the mass density $\rho$ is assumed to be a constant, while the bulk modulus $k$ is variable in a smooth domain $ \Omega$. We assume $k$ to be a $\mathbb{W}^{1, \infty}(\Omega)$ \footnote{This condition can be replace by an $\mathbb{L}^{\infty}$-regularity.} and positive function and $\Omega$ of class $C^2$. In addition, we assume that $(\ref{EquaKg-introdution})$ has a unique solution, i.e. $\omega^2$ is not an eigenvalue of $-n^{-2}\Delta$ with zero Neumann boundary condition on $\partial \Omega$. 
\bigskip

The Calderón problem was the object of an intensive study since the early 80's. The reader can see the following references for more information \cite{colton2019inverse, Isakov-book, Ramm-book} and \cite{Uhlmann-Review}. A model of particular interest is the EIT (Electrical Impedance Tomography) problem, also called Calderón's problem, which consists in identifying the conductivity $\gamma$ using Cauchy data $(u_{|_{\partial \Omega}}, \gamma \nabla u \cdot \nu_{|_{\partial \Omega}})$
 of the solution of equation $\nabla \cdot \gamma \nabla u=0$, in $\Omega \subset \mathbb{R}^3$, where $\nu$ is the outward unit normal vector to $\partial \Omega$. The uniqueness question of this problem is reduced to the construction of the so-called complex geometrical optics solutions (in short CGO's), see \cite{JSGU}, where $\gamma$ is a positive 
$C^2$-smooth function. The regularity of $\gamma$ is reduced to $C^{\frac{3}{2}+\epsilon}, \epsilon >0$,
 in \cite{Brown-1996}, then to $\mathbb{W}^{\frac{3}{2}, \infty}$  in \cite{paivarinta2003complex} and to $\mathbb{W}^{\frac{3}{2}, p}, p>6$ 
 in \cite{Brown-Torres-2003}. Finally, in \cite{Caro-Rogers-2016} and \cite{Haberman-Tataru} this condition is reduced to $\mathbb{W}^{1, \infty}$
 and then to $\mathbb{W}^{1, 3}$
 in \cite{Haberman-2015}. The corresponding Calderón problem in the 2D-setting was solved in \cite{nachman1996global}. In  \cite{bukhgeim2008recovering} the author shows, for the Schrödinger equation given by $\Delta u + \mathfrak{q} \, u =0$, in $\Omega \subset \mathbb{R}^{2}$, the uniqueness of a reconstruction of the potential $\mathfrak{q}(\cdot) \in \mathbb{L}^{p}\left( \Omega \right) ,  p > 2,$ from the Cauchy data, i.e. $\left( u_{|_{\partial \Omega}} ; \partial_{\nu} u_{|_{\partial \Omega}} \right)$, see \cite[Theorem 3.5]{bukhgeim2008recovering}. In \cite{nachman1996global}, we find a justification of the uniquely determination, from the knowledge of Dirichlet-to-Neumann map, of the coefficient $\gamma$ of the elliptic equation\footnote{The substitution $\tilde{u} = \sqrt{\gamma} \, u$ in $\nabla \cdot \left( \gamma \, \nabla u \right) = 0$ yields $\Delta \tilde{u} + \mathfrak{q} \, \tilde{u} = 0$, with $\mathfrak{q} = - \dfrac{1}{\sqrt{\gamma}}\, \Delta \left( \sqrt{\gamma} \right)$. 
 } $\nabla \cdot \left( \gamma \, \nabla u \right) = 0$ in a two dimensional domain. In \cite{Chanillo}, it is proved that if for all cubes $\bm{Q} \subset \mathbb{R}^{n}$ the condition on the smallness of $
\underset{\bm{Q}}{Sup} \; \left\vert \bm{Q} \right\vert^{\frac{2 \, p \, - \, n}{n \, p}} \; \left\Vert \mathfrak{q} \right\Vert_{\mathbb{L}^{p}(\bm{Q})}, \quad p > \frac{(n-1)}{2}, $
is satisfied, or $\mathfrak{q} \in \mathbb{L}^{p}$ with $ p > \frac{n}{2}$, then the Dirichlet-to-Neumann map determines the potential $\mathfrak{q}$. Let us also cite  \cite{MPS-2018, Caro-Garcia} regarding Dirac-type singular potentials. For more details, and without being exhaustive we  refer the readers to the following works \cite{Alessandrini-1988, Alessandrini-1990, paivarinta2003complex, brown1996global, Brown1997UniquenessIT, JSGU, KV, nachman1988, salo2008calderon} and the references therein. Let us mention, however, that apart from few works, like \cite{nachman1988}, where we find a reconstruction algorithm, most of these works are devoted to unique identifiability questions or stability estimates.
\bigskip

In this work, we propose a different approach for solving constructively this problem. The motivation of this approach comes from the engineering literature, see for instance \cite{B-B-C, Qui-al-2009, A-H-Z-T-D-H-S-L-F-R2011, C-F-Q2009, L-C2015, Mc-Oz-2014, Vasilis-Ntziachristos} and many more, where it is suggested to inject small-scaled contrast agents into the region of interest to create the contrasts that are missing to generate clearer images. Such contrast agents could be injected in isolation, as single (or isolated), Dimers, or in general as designed Polymers. They can also be injected as a cluster 'distributed' in the region of interest. Based on these ideas, we proposed in our recent works to use resonant contrast agents for solving inverse problems appearing in some imaging modalities, as ultrasound, optics or photo-acoustic imaging modalities, \cite{ghandriche2022mathematical, AhceneMouradMaxwell, ghandriche2022simultaneous, AlexBubbles, SW-2022, SSW-2023}. In those works, we use the measurements created after injecting single contrast agents (acoustic bubbles or nano-particles) as follows:
\begin{enumerate}
\item In the time-harmonic regime, we recover the induced resonances (as the Minnaert or plasmonic ones) from which, we could recover the wave speeds (or related coefficients), see \cite{ghandriche2022mathematical, AhceneMouradMaxwell, AlexBubbles}.
\item[] 
\item In the time-domain regime, we recover the internal values of the travel time function. From the Eikonal equation, we extract the values of the speed, see \cite{ghandriche2022simultaneous, SW-2022, SSW-2023}. 
\end{enumerate}

In those works, we use contrast agents injected in isolation. This means, for each single injected agent we collect the generated measurements. However, it is of importance to emphasize that we measure only on one single point. In terms of dimensionality, this is advantageous. 
\medskip

In the current work, we inject all the contrast agents at once and then collect the measurements for multiple incident waves. In short, we collect the NtD mapping after injecting the collection of contrast agents all at once. With such measurements, we propose an approach to perform the reconstruction of the index of refraction $n^{2}(\cdot)$.  This approach is divided into two steps:
\begin{enumerate}
\item[] 
\item In the first step, we show that the NtD map generated by the coefficient $n^{2}(\cdot)$ and the collection of contrast agents converges to the one generated by a sum of $n^{2}(\cdot)$ and an effective coefficient. This effective coefficient is negative-valued and one can tune its amplitude. The negativity of the effective coefficient, which is key, is due to the resonant character of the injected contrast agents. Therefore, we can tune these injected agents so that the sum of $n^{2}(\cdot)$ and the effective coefficients is negative valued with a controllable amplitude. 
\item[]
\item From the effective NtD map, we reconstruct the coefficient $n^{2}(\cdot)$. To do so, we show that, due to the negativity of the effective coefficients, mentioned above, we can linearize the  effective NtD map. Finally, from this linearized map, we derive an explicit formula to recover $n^{2}(\cdot)$ in terms of (explicit) CGO-solutions. 
\item[] 
\end{enumerate}

To go further into details, let us take as contrast agents droplets, which are bubbles filled in with water, having the following properties. They are modelled as $D_j, j=1, ..., M,$ of the form $D_j=z_j + a\; B$ with $B$ as a smooth domain containing the origin and maximum radius as unity, such that $D \, \equiv \, \underset{j=1}{\overset{M}{\cup}} D_{j}$. Their mass density $\rho_j$ are equal and estimated as $\rho_{j} \, = \, \rho_{0}$, for $1 \leq j \leq M$, with $\rho_{0}$ is a constant independent on the parameter $a$, while their bulk modulus are very small and of order 
\begin{equation}\label{ScaleBulk}
k_{j} = k_{0} \; a^2,
\end{equation}
with $k_0$ being as fixed constant independent of $a$.
The maximum radius $a$ of this droplet is of  order micrometer, therefore we take $$a \ll 1. $$ 
\bigskip
We introduce the Newtonian operator $N_{D_j}(\cdot) \, : \, \mathbb{L}^{2}(D_j) \, \rightarrow \, \mathbb{L}^2(D_j)$, with the image in $\mathbb{H}^2(D_j)$, given by the expression 
\begin{equation}\label{DefNPO}
N_{D_j}(f)(x) := \int_{D_j} \Phi_{0}(x;y) \; f(y) \; dy ,\; x \in D_{j},
\end{equation}
where $\Phi_{0}(\cdot;\cdot)$ is the fundamental solution of the free-space Laplacian operator, i.e., 
\begin{equation}\label{LREq1110}
\underset{x}{\Delta} \Phi_{0}(x;y) \, = \, - \, \underset{y}{\boldsymbol{\delta}}(x), \quad \text{with} \;\; x, y \, \in \, \mathbb{R}^{3},     
\end{equation}
given by 
\begin{equation}\label{DefPhi0Equa1100}
    \Phi_{0}(x;y) \, := \, \frac{1}{4 \, \pi \, \left\vert x \, - \, y \right\vert}, \quad x \neq y.
\end{equation}
 This operator is self-adjoint and compact, therefore it enjoys a positive sequence of eigenvalues $\left\{ \lambda^{D_j}_{n} \right\}_{n \in \mathbb{N}}$ and they scale as $\lambda^{D_j}_{n} =a^2\; \lambda^{B}_{n}$, where $\left\{ \lambda_{n}^{B} \right\}_{n \in \mathbb{N}}$ is the sequence of eigenvalues associated to the operator $N_{B}(\cdot)$, see \cite{rozenblum2016isoperimetric, Anderson}. We fix any $n_0 \in \mathbb{N}$; $j \in \{1; \cdots ; M \}$ and we consider the eigenvalue $\lambda^{D_{j}}_{n_0}$. The incident frequency $\omega$ that we use, in this acoustic model, is taken of the form:
 \begin{equation}\label{cn0}
\frac{\omega^2}{\omega_0^2} \, = \, 1 \, - \, \frac{ c_{n_{0}} \; a^h}{k_{0}},
 \end{equation}
where the parameter $h$ is positive, and $c_{n_{0}} \in \mathbb{R}$ satisfies $\;\; c_{n_{0}} \, < \, 0,\;\;$
is a parameter which is independent of $a$. The quantity $\omega_0$ is defined as 
\begin{equation*}
\omega_0:=\sqrt{\frac{k_j}{\lambda^{D_j}_{n_0}}}=\sqrt{\frac{k_0}{\lambda^{B}_{n_0}}},
\end{equation*}
 where the last equality is a consequence of the eigenvalues scales and $(\ref{ScaleBulk})$.
\bigskip
\newline
Next, we make the following necessary assumption about the $D_{m}$'s distribution to derive the first main result of this work, i.e. \textbf{Theorem \ref{principal-Thm}}.
\begin{assumption}\label{AssDmDist}
The droplets are distributed periodically inside $\Omega$. More precisely, let $\Omega$ be a bounded domain of unit volume, containing the droplets $D_{j}$, with $j=1, \cdots, M$. We divide $\Omega$ as  
\begin{equation}\label{Decoupage-Omega}
\Omega \, \equiv \, \Omega_{cube} \cup \Omega_{r} \quad \text{with} \quad 
\Omega_{cube} \, \equiv \,  \overset{M}{\underset{j=1}{\cup}}
\Omega_{j} \;\; \text{and} \;\; \Omega_{r} \, \equiv \, \underset{j-1}{\overset{\aleph}{\cup}} \Omega_{j}^{\star}, \;\;\;\; M = M(a), \; \aleph = \aleph(a) \in \mathbb{N}, 
\end{equation}
where $\Omega_{j}$'s are cubes located strictly within the interior of the domain $\Omega$, i.e. they do not intersect with $\partial \Omega$ $\left( \Omega_{cube} \subsetneq \Omega \right)$. Each subdomain $\Omega_{j}$ contains one $D_j$ such that $z_{j} \in D_{j} \subset \Omega_{j}$ and $\left\vert \Omega_{j} \right\vert \, = \, a^{1-h}$, for $j=1, \cdots, M$ and $0 \leq h < 1$, while the $\Omega_j^{\star}$'s do not contain any, see \textbf{Figure \ref{Fig1}} for a schematic representation.  Therefore, by denoting $\Omega_{0}$ as a reference subdomain the distribution of $\Omega_{j}$'s is constructed by appropriate translations of $\Omega_{0}$. In a concise manner, the distribution of the droplets can be written as 
\begin{equation}\label{D-Omega_cube}
    D \, \equiv \, \Omega_{cube} \, \cap \, d \, \left(\mathbb{Z}^{3} \, + \, \left(z \, + \, \frac{a}{d} \, B \right) \right),
\end{equation}
where $d$ is the minimal distance given by $(\ref{dmin})$, $z$ is a point contained in unit cell domain, and $B$ is a Lipschitz domain in $\mathbb{R}^{3}$, such that $\diam(B) \, \sim \, 1$. Besides, we assume that $\Omega_{cube}$ is away from the boundary $\partial \Omega$, such that
\begin{equation}\label{dist-Omega_int-to-Omega}
    \dist\left(\partial \Omega_{cube}; \partial \Omega \right) \, \sim \, \kappa(a) \, \sim \, a^{\frac{(1-h)}{3}}, \quad \text{with} \;\; a \ll 1 \;\; \text{and} \;\; 0 < h < 1. 
\end{equation}
Hence, from $(\ref{D-Omega_cube})$ and $(\ref{dist-Omega_int-to-Omega})$, the droplets $D$ are away from the boundary $\partial \Omega$, such that
\begin{equation}\label{distto}
    \dist\left(D; \partial \Omega \right) \, \sim \, \kappa(a) \, \sim \, a^{\frac{(1-h)}{3}}, \quad \text{with} \;\; a \ll 1 \;\; \text{and} \;\; 0 < h < 1.
\end{equation}
\end{assumption}
\begin{figure}[H]
\begin{center}
 \begin{tikzpicture}[scale = 0.75]
\draw[black, ultra thick] (0,0) ellipse (7 and 3.5);
\draw[gray, thick] (-7,0) -- (7,0);
\draw[gray, thick] (-6.7,1) -- (6.7,1);
\draw[gray, thick] (-6.7,-1) -- (6.7,-1);
\draw[gray, thick] (-5.8,2) -- (5.8,2);
\draw[gray, thick] (-5.8,-2) -- (5.8,-2);
\draw[gray, thick] (-3.7,3) -- (3.7,3);
\draw[gray, thick] (-3.7,-3) -- (3.7,-3);
\draw[gray, thick] (0,-3.5) -- (0,3.5);
\draw[gray, thick] (1,-3.5) -- (1,3.5);
\draw[gray, thick] (-1,-3.5) -- (-1,3.5);
\draw[gray, thick] (2,-3.4) -- (2,3.4);
\draw[gray, thick] (-2,-3.4) -- (-2,3.4);
\draw[gray, thick] (3,-3.2) -- (3,3.2);
\draw[gray, thick] (-3,-3.2) -- (-3,3.2);
\draw[gray, thick] (4,-2.9) -- (4,2.9);
\draw[gray, thick] (-4,-2.9) -- (-4,2.9);
\draw[gray, thick] (5,-2.5) -- (5,2.5);
\draw[gray, thick] (-5,-2.5) -- (-5,2.5);
\draw[gray, thick] (6,-1.8) -- (6,1.8);
\draw[gray, thick] (-6,-1.8) -- (-6,1.8);
\filldraw[green] (0.5,0.5) circle (0.35);
\filldraw[green] (0.5,1.5) circle (0.35);
\filldraw[green] (0.5,2.5) circle (0.35);
\filldraw[green] (0.5,-0.5) circle (0.35);
\filldraw[green] (0.5,-1.5) circle (0.35);
\filldraw[green] (0.5,-2.5) circle (0.35);
\filldraw[green] (1.5,0.5) circle (0.35);
\filldraw[green] (1.5,1.5) circle (0.35);
\filldraw[green] (1.5,2.5) circle (0.35);
\filldraw[green] (1.5,-0.5) circle (0.35);
\filldraw[green] (1.5,-1.5) circle (0.35);
\filldraw[green] (1.5,-2.5) circle (0.35);
\filldraw[green] (2.5,0.5) circle (0.35);
\filldraw[green] (2.5,1.5) circle (0.35);
\filldraw[green] (2.5,2.5) circle (0.35);
\filldraw[green] (2.5,-0.5) circle (0.35);
\filldraw[green] (2.5,-1.5) circle (0.35);
\filldraw[green] (2.5,-2.5) circle (0.35);
\filldraw[green] (-0.5,0.5) circle (0.35);
\filldraw[green] (-0.5,1.5) circle (0.35);
\filldraw[green] (-0.5,2.5) circle (0.35);
\filldraw[green] (-0.5,-0.5) circle (0.35);
\filldraw[green] (-0.5,-1.5) circle (0.35);
\filldraw[green] (-0.5,-2.5) circle (0.35);
\filldraw[green] (-1.5,0.5) circle (0.35);
\filldraw[green] (-1.5,1.5) circle (0.35);
\filldraw[green] (-1.5,2.5) circle (0.35);
\filldraw[green] (-1.5,-0.5) circle (0.35);
\filldraw[green] (-1.5,-1.5) circle (0.35);
\filldraw[green] (-1.5,-2.5) circle (0.35);
\filldraw[green] (-2.5,0.5) circle (0.35);
\filldraw[green] (-2.5,1.5) circle (0.35);
\filldraw[green] (-2.5,2.5) circle (0.35);
\filldraw[green] (-2.5,-0.5) circle (0.35);
\filldraw[green] (-2.5,-1.5) circle (0.35);
\filldraw[green] (-2.5,-2.5) circle (0.35);
\filldraw[green] (-3.5,0.5) circle (0.35);
\filldraw[green] (-3.5,1.5) circle (0.35);
\filldraw[green] (-3.5,-0.5) circle (0.35);
\filldraw[green] (-3.5,-1.5) circle (0.35);
\filldraw[green] (-4.5,0.5) circle (0.35);
\filldraw[green] (-4.5,1.5) circle (0.35);
\filldraw[green] (-4.5,-0.5) circle (0.35);
\filldraw[green] (-4.5,-1.5) circle (0.35);
\filldraw[green] (4.5,0.5) circle (0.35);
\filldraw[green] (4.5,1.5) circle (0.35);
\filldraw[green] (4.5,-0.5) circle (0.35);
\filldraw[green] (4.5,-1.5) circle (0.35);
\filldraw[green] (3.5,0.5) circle (0.35);
\filldraw[green] (3.5,1.5) circle (0.35);
\filldraw[green] (3.5,-0.5) circle (0.35);
\filldraw[green] (3.5,-1.5) circle (0.35);
\filldraw[green] (-5.5,0.5) circle (0.35);
\filldraw[green] (-5.5,-0.5) circle (0.35);
\filldraw[green] (5.5,0.5) circle (0.35);
\filldraw[green] (5.5,-0.5) circle (0.35);
\draw[very thick] (0,-3) -- (0,-2);
\draw[very thick] (1,-3) -- (1,-2);
\draw[very thick] (0,-2) -- (1,-2);
\draw[very thick] (0,-3) -- (1,-3);
\draw node at (0,4.25)  {$\boldsymbol{\Omega}$}; 
\draw node at (0.5,-5) {$\boldsymbol{\Omega_{m}}$};
\draw node at (-4.5,-5) {$\boldsymbol{\Omega_{n}^{\star}}$};
\draw node at (4.5,-5) {$\boldsymbol{D_{j}}$};
\draw node at (2,4)  {$\textcolor{black}{\kappa(a)}$}; 
\draw[dashed,thick,->] (4.5,-4.7) -- (4.5,-1.5);
\draw[dashed,thick,->] (-4.5,-4.7) -- (-4.5,-2.2);
\draw[dashed,thick,->] (0.5,-4.7) -- (0.5,-3);
\draw[red,thick,<->] (2.75,2.75) -- (3.1,3.1);
\draw[red,->] (2.1,3.7) -- (2.85,2.95);
\end{tikzpicture}
\end{center}
\caption{An illustration of how the droplets are distributed in $\Omega$.}
\label{Fig1}
\end{figure}
We are concerned with the case where we have the number $M$ of droplets of the order 
\begin{equation}\label{M-}
M \sim a^{h-1}, \;  a \ll 1 \quad \text{with} \quad 0 \leq h<1,
\end{equation}
and, then, the minimum distance between the droplets is 
\begin{equation}\label{dmin}
d := \underset{i \neq j \atop 1 \leq i , j \leq M}{\min} \left\vert z_{i} - z_{j} \right\vert \sim a^{\frac{(1-h)}{3}},\;  a \ll 1 \quad \text{with} \quad 0 \leq h<1,   
\end{equation}
as $M\sim d^{-3}$. The choice in (\ref{M-}) is dictated by the behavior of scattering coefficient (or the polarization tensor) in  (\ref{Scattering-coefficient}) which is of the order $a^{1-h}$ and $M$ should be inversely proportional to it. This behavior allows to generate a non trivial effective medium in the homogenization process.
\bigskip

\noindent With these notations at hand, let us state the perturbed problem as follows:
\begin{align}\label{Equavf}
\begin{cases}  
\left( \Delta  + \omega^{2} \, n^{2}(\cdot) \, ( 1 \, - \, \underset{D}{\chi}) +  \omega^{2} \,  \dfrac{\rho_{1}}{k_{1}} \, \underset{D}{\chi} \right) \, v^{g} = 0 \quad \text{in} \quad  \Omega,  \\ 
\qquad \qquad \qquad \qquad \qquad \qquad \; \qquad \partial_{\nu} v^{g}  =  g \quad  \text{on} \quad \partial \Omega. 
\end{cases}
\end{align}

\noindent Under the condition that $\omega^2$ is not an eigenvalue of $-n^{-2}\Delta$ with zero Neumann boundary condition on $\partial \Omega$ and that $c_{n_0}$ and $a$ are small enough, the problem (\ref{Equavf}) is well posed. Indeed, it is clear that the operator solution of the problem in (\ref{Equavf}) is a compact perturbation of the problem (\ref{EquaKg-introdution}) which is well posed. By \textbf{Lemma \ref{ADZ-Lemma}}, we deduce the uniqueness of the solution of (\ref{Equavf}). Actually, by \textbf{Lemma \ref{ADZ-Lemma}} we also derive the related estimate of the well-posedness for problem $(\ref{Equavf})$.

\subsection{From the original NtD map $\Lambda_{D}$ to the effective NtD map $\Lambda_{P}$: Theorem \ref{principal-Thm}}
~\\

Let $v^{g}(\cdot)$  be the solution of the problem (\ref{Equavf}). Multiplying $(\ref{Equavf})$ by $p^{f}(\cdot)$, solution of $(\ref{EquaKg-introdution})$, and integrating over $\Omega$, we obtain 
\begin{equation}\label{1st-integration}
\langle \Lambda_0\left( f \right);  g \rangle_{\mathbb{H}^{\frac{1}{2}}(\partial \Omega) \times \mathbb{H}^{-\frac{1}{2}}(\partial \Omega)} =  \langle \nabla v^{g}; \nabla p^{f} \rangle_{\mathbb{L}^{2}(\Omega)}  - \omega^{2} \, \langle  n^{2} \, v^{g};  p^{f} \rangle_{\mathbb{L}^{2}(\Omega)} - \omega^{2} \,  
\left\langle \left(\frac{\rho_{1}}{k_{1}} \, - \, n^{2}\right) v^{g}; p^{f} \right\rangle_{\mathbb{L}^{2}(D)}, 
\end{equation}
where $\Lambda_0(\cdot)$ is the NtD map defined from $\mathbb{H}^{-\frac{1}{2}}\left( \partial \Omega \right)$ to $\mathbb{H}^{\frac{1}{2}}\left( \partial \Omega \right)$ by  
\begin{equation*}
\langle \Lambda_0 \left( f \right) ; g \rangle_{\mathbb{H}^{\frac{1}{2}}(\partial \Omega) \times \mathbb{H}^{-\frac{1}{2}}(\partial \Omega)}  :=  \int_{\partial \Omega} p^{f}(x) \, g(x) \, d\sigma(x)
\end{equation*}
where we use integrals to simplify notations. We set $\Lambda_{D}(\cdot)$ to be the NtD map of the background after injecting a cluster of droplets, i.e. the problem (\ref{Equavf}). Multiplying $(\ref{EquaKg-introdution})$ by $v^{g}(\cdot)$ and integrating over $\Omega$, using the selfadjointness of $\Lambda_D$ and (\ref{1st-integration}), we end up with the coming formula
\begin{equation}\label{Equa1Lambda}
\langle \Lambda_{D}(f); g \rangle_{\mathbb{H}^{\frac{1}{2}}(\partial \Omega) \times \mathbb{H}^{-\frac{1}{2}}(\partial \Omega)} - \langle \Lambda_0 \left( f \right) ; g \rangle_{\mathbb{H}^{\frac{1}{2}}(\partial \Omega) \times \mathbb{H}^{-\frac{1}{2}}(\partial \Omega)} \, = \omega^{2}  \; \left\langle \left( \dfrac{\rho_{1}}{k_{1}} - n^{2} \right) v^{g}; p^{f} \right\rangle_{\mathbb{L}^{2}(D)},
\end{equation}

In a similar way, we define $u^{g}(\cdot)$ to be solution of 
\begin{align}\label{EquaWf}
\begin{cases}  
\left( \Delta  + \omega^{2} \, n^{2}(\cdot) - P^{2} \right) \, u^{g} \; = 0 \quad \text{in} \quad  \Omega,  \\ 
\qquad \qquad \qquad \qquad \partial_{\nu} u^{g} \; =  g \;\;\;  \text{on} \quad \partial \Omega. 
\end{cases}
\end{align}
Here \footnote{The assumption that $ \langle 1; \overline{e}_{n_{0}} \rangle_{\mathbb{L}^{2}(B)}\neq 0$ is reasonable. When $B$ is a ball,  we have an infinite sequence of eigenvalues $\lambda^B_{n_0}$ for which the corresponding eigenfunctions satisfy $\langle 1; \overline{e}_{n_{0}} \rangle_{\mathbb{L}^{2}(B)}\neq 0$, see \cite{kalmenov2011boundary} for instance.}
\begin{equation}\label{DefP2cn0}
P^{2} := \frac{ - \, k_{0} \; \left( \langle 1; \overline{e}_{n_{0}} \rangle_{\mathbb{L}^{2}(B)} \right)^{2}}{\lambda_{n_{0}}^{B} \; c_{n_{0}}}, 
\end{equation}
where $\overline{e}_{n_{0}}(\cdot)$ is the eigenfunction associated to the eigenvalue $\lambda_{n_{0}}^{B}$ related to the Newtonian operator, given by $(\ref{DefNPO})$, defined in the domain $B$. We set $\Lambda_{P}(\cdot)$ to be the NtD map of the equivalent background, then we obtain  
\begin{equation}\label{Equa2Lambda}
\langle \Lambda_{P}(f); g \rangle_{\mathbb{H}^{\frac{1}{2}}(\partial \Omega) \times \mathbb{H}^{-\frac{1}{2}}(\partial \Omega)} - \langle \Lambda_0 \left( f \right) ; g \rangle_{\mathbb{H}^{\frac{1}{2}}(\partial \Omega) \times \mathbb{H}^{-\frac{1}{2}}(\partial \Omega)} =  - \, P^{2} \, \langle  u^{g}; p^{f} \rangle_{\mathbb{L}^{2}(\Omega)},
\end{equation}

From $(\ref{Equa1Lambda})$ and $(\ref{Equa2Lambda})$, we see that
\begin{equation}\label{Lambda-d--Lambda-P}
\langle \left( \Lambda_{D}\,  - \, \Lambda_{P} \right)(f); g \rangle_{\mathbb{H}^{\frac{1}{2}}(\partial \Omega) \times \mathbb{H}^{-\frac{1}{2}}(\partial \Omega)}  \, = \, \omega^{2} \, \left\langle \left( \frac{\rho_{1}}{k_{1}} - n^{2} \right) v^{g};  p^{f} \right\rangle_{\mathbb{L}^{2}(D)} \, + \, P^{2} \, \langle  u^{g};  p^{f} \rangle_{\mathbb{L}^{2}(\Omega)}.
\end{equation}
In the sequel, we prove that when $M$ is large, or $a$ is small, the perturbed medium, after injecting a cluster of $M$ droplets, behaves like the equivalent background. In other words, the map $\Lambda_{D}(\cdot)$ converges to $\Lambda_{P}(\cdot)$.
\begin{theorem}\label{principal-Thm}
Let the domain $\Omega$ be $C^2$-regular, the index of refraction $n^{2}(\cdot) \in \mathbb{W}^{1,\infty}(\Omega)$, the used frequency $\omega$ satisfying $(\ref{cn0})$, the parameter $h$ be such that $\frac{1}{3} \, < \, h \, < \, 1$, and the droplets $D_{m}$'s are distributed as explained in \textbf{Assumption \ref{AssDmDist}}. Then, we have the following convergence
\begin{equation*}
\langle \Lambda_{D}(f); g \rangle_{\mathbb{H}^{\frac{1}{2}}(\partial \Omega) \times \mathbb{H}^{-\frac{1}{2}}(\partial \Omega)}  \underset{a \rightarrow 0}{\longrightarrow} \langle \Lambda_{P}(f); g \rangle_{\mathbb{H}^{\frac{1}{2}}(\partial \Omega) \times \mathbb{H}^{-\frac{1}{2}}(\partial \Omega)},
\end{equation*}
uniformly in terms of $ \left( f, g \right) \in \mathbb{H}^{-\frac{1}{2}}\left(\partial \Omega \right) \times  \mathbb{H}^{-\frac{1}{2}}\left(\partial \Omega \right)$. 
Precisely, we have the following rate \footnote{The $\mathbb{W}^{1, \infty}$-regularity of $k$, and hence $n$, is used to derive the rate in  $(\ref{energy-D})$. The $\mathbb{L}^{\infty}(\Omega)$-regularity is enough to derive the convergence (without rates).} 
\begin{equation}\label{energy-D}
\left\Vert \Lambda_{D} - \Lambda_{P} \right\Vert_{\mathcal{L}(\mathbb{H}^{-\frac{1}{2}}(\partial \Omega), \mathbb{H}^{\frac{1}{2}}(\partial \Omega))}  \lesssim  a^{\frac{(1-h) (9 - 5 \delta)}{18 (3-\delta)}} \, P^{6}, \,\; a \ll 1,
\end{equation}
where $\delta$ a sufficiently small but arbitrarily positive number. 
\end{theorem}

\begin{remark}
Two comments are in order.
    \begin{enumerate}
         \item[] 
         \item Since $M \sim a^{h-1}$ and $\delta$ is very small, we can rewrite $(\ref{energy-D})$ as: 
         \begin{equation*}
\left\Vert \Lambda_{D} - \Lambda_{P} \right\Vert_{\mathcal{L}(\mathbb{H}^{-\frac{1}{2}}(\partial \Omega), \mathbb{H}^{\frac{1}{2}}(\partial \Omega))} 
 \lesssim  M^{\frac{(5 \delta - 9)}{18 (3-\delta)}} \, P^{6},\; M\gg 1.
\end{equation*}
We can choose $M$, i.e. $a$, such that \begin{equation}\label{M-P}
M^{\frac{(5 \delta - 9)}{18 (3-\delta)}} \, P^{6} \ll 1.
\end{equation}
\item[] 
\item The parameter $\delta$ in $(\ref{energy-D})$ is linked to the $\mathbb{L}^{3-\delta}(\Omega)$-integrability of the fundamental solution $\Phi_{0}(\cdot;\cdot)$, given by $(\ref{DefPhi0Equa1100})$.        
\item[] 
     \end{enumerate}
\end{remark}
As we assume to know the NtD map $\Lambda_D(\cdot)$, for $M$ large, the previous theorem suggests the following result. 
\begin{corollary} Under the condition of \textbf{Theorem \ref{principal-Thm}}, the NtD map $\Lambda_P(\cdot)$ is approximately known.
\end{corollary}

The proof of \textbf{Theorem \ref{principal-Thm}} is based on the point-interaction approximation, or the so-called Foldy-Lax approximation. We first approximate the left part in (\ref{energy-D}) by a linear combination of elements of a vector which is solution of an algebraic system. This algebraic system captures the multiple scattering between the injected droplets through an interaction matrix where the interaction coefficients, that are also called scattering coefficients, are all positive due to the choice made in (\ref{cn0}) of the sign of $c_{n_0}$. To prove the invertibility of this algebraic system, uniformly of the large number $M$ of droplets, we first justify the invertibility of the related continuous integral equation and then, we show, with quite tedious computations, that the algebraic equation is 'a discrete form' of this continuous integral equation.

\begin{remark} 
Two remarks are in order.
\begin{enumerate}

\item In $(\ref{DefP2cn0})$, we take the constant $c_{n_{0}} < 0$ and the parameter $P^{2}$ such that  
\begin{equation*}
P^{2} > \omega_{0}^{2} \; \left\Vert n^{2} \right\Vert_{\mathbb{L}^{\infty}(\Omega)} := P_{\min}, 
\end{equation*}
where, we recall that, $\omega_{0}^{2} = \dfrac{k_{0}}{\lambda^{B}_{n_{0}}}$.
This is possible if we choose the parameter $c_{n_0}$ to satisfy\footnote{We assume that we have an a priori information on $\underset{y \in \Omega}{Inf} \left\vert k(y) \right\vert$.}
\begin{equation*}
- \, \rho^{-1} \, \underset{y \in \Omega}{Inf} \left\vert k(y) \right\vert \, \left( \langle 1 ; \overline{e}_{n_{0}} \rangle_{\mathbb{L}^{2}(B)} \right)^{2} \, < \, c_{n_{0}} \, < \, 0 \quad \text{and} \quad c_{n_{0}} \rightarrow 0^{-}.  
\end{equation*}
We recall that the parameter $c_{n_0}$ appears in (\ref{cn0}) and we have (\ref{DefP2cn0}). The coefficient $c_{n_0}$ is taken small, and hence $P$ large, but satisfies (\ref{M-P}).
\bigskip

\item The parameter $h$ appearing in (\ref{cn0}) and (\ref{M-}) models how dilute, or dense, is the distribution of the injected droplets in $\Omega$. If $h$ is close to $0$, we have a dense distribution and when $h$ is close to $1$ we have a light distribution.
\end{enumerate}
\end{remark}

\smallskip

\subsection{The linearization of the effective NtD map $\Lambda_{P}(\cdot)$: Theorem \ref{THMLinearization}}
\begin{theorem}\label{THMLinearization}
  We have the following linearisation of $\Lambda_{P}(\cdot)$, in the $\mathbb{H}^{\frac{1}{2}}(\partial \Omega)$ sense,
\begin{equation}\label{ASMTV}
\Lambda_P(f) - q^{f} = \omega^{2} \, \gamma \left( \bm{W}^{q^{f}} \right)  + \mathcal{O}\left( \left\Vert f \right\Vert_{\mathbb{H}^{-\frac{1}{2}}(\partial \Omega)} \, \frac{1}{P^{4}}  \right),
\end{equation}
where $f(\cdot) \in \mathbb{H}^{-\frac{1}{2}}(\partial \Omega)$, $\gamma(\cdot)$ is the trace operator defined from $\mathbb{H}^{s}(\Omega)$ to $\mathbb{H}^{s-\frac{1}{2}}(\partial \Omega)$, $s\geq \frac{1}{2}$, \, $q^{f}(\cdot)$ is solution of 
\begin{align}\label{Equaqf}
\begin{cases}  
\left( \Delta - P^{2}  \right) q^{f}  = 0 \quad \, \text{in} \quad  \Omega,  \\ 
\qquad \; \, \quad \partial_{\nu} q^{f}  = f \quad  \text{on} \quad \partial \Omega, 
\end{cases}
\end{align} 

and $\bm{W}^{q^{f}}(\cdot)$ satisfies

\begin{align*}\label{ASKD-intro}
\begin{cases}  
\left(  \Delta -  P^{2} \, I \right) \bm{W}^{q^{f}} \, = - \, n^{2} \, q^{f} \quad \text{in} \quad  \Omega,  \\ 
\qquad \qquad \, \partial_{\nu} \bm{W}^{q^{f}} = 0  \quad \quad \; \quad  \text{on} \quad \partial \Omega. 
\end{cases}
\end{align*}
\end{theorem}
Therefore knowing $\Lambda_P(f)$ allows us to construct $\bm{W}^{q^{f}}$, for $f \in \mathbb{H}^{-\frac{1}{2}}(\partial \Omega)$.
\medskip

The proof of \textbf{Theorem \ref{THMLinearization}} is based on the observation that the solution operator (i.e. the Lippmann-Schwinger operator) of the problem  $(\ref{EquaWf})$ can be seen as the one of the problem $(\ref{Equaqf})$ plus a 'small' perturbation. The smallness of this perturbation permits us to justify the related linearization. The arguments of the analysis are based on the spectral and scaling properties of the Newtonian operator of the solution operator of  $(\ref{Equaqf})$  via Calderon-Zygmund type estimates.  

\subsection{Construction of $n^{2}(\cdot)$ from the linearization of $\Lambda_{P}(\cdot)$: Theorem\ref{THMReconstruction}}
~\\

The next theorem describes a way how we can reconstruct the sound speed from the linearized part of $\Lambda_{P}(\cdot)$.
\begin{theorem}\label{THMReconstruction}
For every  $l := \left(l_{1};l_{2};l_{3} \right) \in \mathbb{Z}^{3}$, we choose 
\begin{equation}\label{Existencexi}
\xi = \frac{P^{2+\varsigma} \, \left\vert l \right\vert^{2+\varsigma}}{\sqrt{2} \, \sqrt{l^{2}_{2} + l^{2}_{3}}} \; \begin{pmatrix}
- i \left(l_{2}^{2} + l_{3}^{2} \right) \\
- \left\vert l \right\vert \, l_{3} + i \, l_{1} \, l_{2}  \\
- \left\vert l \right\vert \, l_{2} + i \, l_{1} \, l_{3} 
\end{pmatrix},
\text{with} \;\; \varsigma \in \mathbb{R}^{+}.
\end{equation}
Hence,
\begin{equation}\label{normxi}
\left\vert \xi \right\vert = P^{2+\varsigma} \, \left\vert l \right\vert^{3+\varsigma}.
\end{equation}

We set $q^{f}(\cdot):=q^{l, \xi}$ the function defined by 
\begin{equation*}
q^{l, \xi}(x) := e^{i \, \xi \cdot x} \, \left( e^{i \,  x \cdot l} \, + r_{1}(x) \right), \quad x \in \mathbb{R}^{3},
\end{equation*}
and $r_{1}(\cdot)$ is such that
\begin{equation}\label{DiffEquar1Intro}
\left(\Delta + 2 \, i \, \xi \cdot \nabla - P^{2} \right) r_{1}(x) = \left(\left\vert l \right\vert^{2} + P^{2} \right) \; e^{i \, x \cdot l}, \quad \text{in} \quad \Omega.
\end{equation}

In the same manner we set $q^{g}(\cdot):=q^{\xi}(\cdot)$ to be the function defined by 
\begin{equation*}
q^{\xi}(x) := e^{- i \, \xi \cdot x} \left( 1 + r_{2}(x) \right), \quad x \in \mathbb{R}^{3},
\end{equation*}
where $r_{2}(\cdot)$ is such that 
\begin{equation}\label{DiffEquar2Intro}
\left(- \Delta + 2 \, i \, \xi \cdot \nabla + P^{2} \right) r_{2}(x) = - \, P^{2}  \; , \quad \text{in} \quad \Omega.
\end{equation}

Then we have the following approximate reconstruction formula:  
\begin{equation}\label{Reconstruction-formula}
n^{2}(x)   =   \left(2 \, \pi \right)^{-3} \; \sum_{\ell \in \mathbb{Z}^{3}} \langle \bm{W}^{q^{l, \xi}} ; \partial_{\nu} q^{\xi} \; \rangle \;\; e^{i \, \ell \cdot x}  + \mathcal{O}\left( P^{-\varsigma} \right),
\end{equation}
in the $\mathbb{L}^{2}(\Omega)$ sense.
\end{theorem}

The justification of the existence and uniqueness of solutions corresponding to the problems $(\ref{DiffEquar1Intro})$ and $(\ref{DiffEquar2Intro})$ can be found in \cite[Section 3.2]{salo2008calderon}. More precisely, in \cite[Theorem 3.7]{salo2008calderon} the result is proved first for the free case equation, i.e. equation of the form $\left(\Delta + 2 \, i \, \xi \cdot \nabla \right) r = f$, where $r(\cdot)$ is a correction term and $f$ is the source data. Then, in \cite[Theorem 3.8]{salo2008calderon} the general case, i.e. equation of the form $\left(\Delta + 2 \, i \, \xi \cdot \nabla + \textbf{q} \right) r = f$, where $\textbf{q}$ is a potential, was proved under the conditions 
$\xi \cdot \xi = 0 \quad \text{and} \quad \left\vert \xi \right\vert \geq \max\left(C_{0} \, \left\Vert \textbf{q} \right\Vert_{\mathbb{L}^{\infty}(\Omega)}; 1 \right)$, where $C_{0}$ is a constant depending on the domain $\Omega$ and the space dimension.
These nicely re-derived estimates were initially proved in the seminal work \cite[Theorem 1.1, Proposition 2.1]{JSGU}. 

The key observation here is that these CGOs are solutions of fully explicit equations, see (\ref{DiffEquar1Intro}) and (\ref{DiffEquar2Intro}),  which make the representation in (\ref{Reconstruction-formula}) constructive.

\begin{remark}\label{Calderon-test-functions}
In Theorem \ref{THMReconstruction}, we have shown how to construct $n^{2}(\cdot)$ using a discrete series expansion. Actually, we can also use the classical Calder\'on idea to construct the Fourier transform of $n^{2}(\cdot)$. Indeed, choosing any $v$ which solves $\left(\Delta - P^2\right)v =0$ and multiplying it with the PDE for $W^{q^{f}}$, we have
\begin{align}\label{int ineq}
    \int_\Omega n^2(x) q^f(x)\, v(x) \, ds(x) = \int_{\partial\Omega} \partial_\nu v(x) \, W^f(x) \, ds(x).
\end{align}
Now, for $0\neq\xi\in\mathbb{R}^3$, let us consider an orthonormal family $\left\{e_1:=\frac{\xi}{|\xi|}, e_2, e_3 \right\}$ of $\mathbb{R}^3$. Using this basis, we take 
\begin{align*}
    \zeta_1 = \frac{|\xi|}{2} e_1 + i e_2 \sqrt{P^2 + \frac{|\xi|^2}{4}} , \quad \quad \quad \zeta_2 = \frac{|\xi|}{2} e_1 - i e_2 \sqrt{P^2 + \frac{|\xi|^2}{4}}
\end{align*}
and then consider the two functions $
    q^f(x) := e^{i\zeta_1\cdot x}$ and $ v(x) := e^{i\zeta_2\cdot x}$. We see that $~(\Delta - P^2)q^f=0$ and $(\Delta - P^2)v=0$ since $
    \zeta_1 \cdot\zeta_1 = \zeta_2 \cdot\zeta_2 = - P^2.$ We also see that $
    \zeta_1+\zeta_2 =\vert \xi\vert e_2=\xi.$ With this choice, it is immediate that
\begin{align*}
    \int_\Omega n^2(x) q^f(x)\, v(x) \, d(x) = \int_{\Omega} e^{i \xi\cdot x} n^2(x) \, d(x) (\xi),~~~ \xi\in\mathbb{R}^3. 
\end{align*}
which allows to construct the Fourier transform of $n^{2}(\cdot)$ from $(\ref{int ineq})$.
\end{remark}

\bigskip

The remaining parts of the paper are organized as follows. In \textbf{Section \ref{Linearization-step}}, we discuss and justify the linearization step and in \textbf{Section \ref{Construction-C}} we deal with the reconstruction of $n^{2}(\cdot)$ from the linearized NtD map. The justification for the effective NtD is stated in \textbf{Section \ref{effective-NtD}}. This choice is taken because this step is technically the most involved part. 
Finally, we postpone several technical steps to be developed and justified in \textbf{Section \ref{Appendix}} stated as an appendix. 

\section{Proof of Theorem \ref{THMLinearization}}\label{Linearization-step}
The goal of this section is to derive a linearization, up to a first order term, of the NtD map of the equivalent background, i.e., $\Lambda_P(\cdot)$.
Let $u^{f}(\cdot)$ be the solution of the following Lippmann-Schwinger Equation (L.S.E in short) 
\begin{equation}\label{UfLSE}
u^{f}(x) - \omega^{2} \, N^{p}\left(n^{2} \, u^{f} \right)(x) \, = \, q^{f}(x), \quad x \in \Omega,  
\end{equation}
where $q^{f}(\cdot)$ is solution of $(\ref{Equaqf})$, and $N^{p}(\cdot)$ is the Newtonian operator defined, from $\mathbb{L}^{2}(\Omega)$ to $\mathbb{H}^{2}(\Omega)$, by 
\begin{equation}\label{DefNp1752}
    N^{p}(f)(x) \, := \, \int_{\Omega} G_{p}(x,y) \, f(y) \; dy, \quad x \in \Omega, 
\end{equation}
with $G_{p}(\cdot,\cdot)$ is solution of 
\begin{align}\label{Green's-Kernel-with-P}
\begin{cases}  
\left( \underset{x}{\Delta} \, - \, P^{2}  \right) G_{p}(x,y) = - \, \underset{y}{\boldsymbol{\delta}}(x) \quad \text{in} \quad  \Omega,  \\ 
\qquad \; \; \quad \partial_{\nu_{x}} G_{p}(x,y) = 0 \qquad  \text{on} \quad \partial \Omega. 
\end{cases}
\end{align}
In effortless manner we can check that $u^{f}(\cdot)$, solution of $(\ref{UfLSE})$, is also solution of $(\ref{EquaWf})$. Moreover, by an induction process on the L.S.E, given by $(\ref{UfLSE})$, we prove that 
\begin{equation}\label{AV}
u^{f}(x) - q^{f}(x) = \omega^{2} \, \gamma N^{p}\left(n^{2} \,  q^{f} \right)(x) + \sum_{j \geq 2} \left(K_{j} \overset{j}{\otimes} \left( n^{2} \right) \right)(x), \quad x \in \partial \Omega, 
\end{equation}
where 
\begin{eqnarray*}
\left(K_{2} \overset{2}{\otimes} \left( n^{2} \right) \right)(x) 
&=& \left( \omega^{2} \right)^{2} \, \gamma N^{p}\left( n^{2} \,  N^{p}\left( n^{2} \, q^{f} \right)\right)(x) \\
\left(K_{3} \overset{3}{\otimes} \left( n^{2} \right) \right)(x) 
& = & (\omega^{2})^{3} \, \gamma N^{p}\left(n^{2} \,  N^{p}\left( n^{2} \, N^{p}\left(n^{2} \, q^{f} \right) \right)\right)(x) \\
& \vdots & \\
\left(K_{j} \overset{j}{\otimes} \left( n^{2} \right) \right) &=& \left( \omega^{2} \right)^{j} \, \gamma \int_{\Omega} \cdots \int_{\Omega}  G_{p} \cdots G_{p} \, n^{2} \cdots  n^{2} \, q^{f} \, dy_{1} \cdots dy_{j},
\end{eqnarray*}
with $N^{p}(\cdot)$ is the Newtonian operator defined by $(\ref{DefNp1752})$, and $\gamma(\cdot)$ is the trace operator defined from $\mathbb{H}^{s}(\Omega)$ to $\mathbb{H}^{s-\frac{1}{2}}(\partial \Omega)$, $s\geq \frac{1}{2}$, with $\Omega$ a smooth domain. The coming lemma is useful to study the convergence of the previous series with respect to the $\mathbb{H}^{\frac{1}{2}}(\partial \Omega)$-norm. 
\begin{lemma}\label{LemmaNp}
The Newtonian operator given by $(\ref{DefNp1752})$ admits the following estimations, 
\begin{equation}\label{NormNewtonian}
\left\Vert N^{p} \right\Vert_{\mathcal{L}\left(\mathbb{L}^{2}(\Omega); \mathbb{L}^{2}(\Omega) \right)} = \mathcal{O}\left( \frac{1}{P^{2}} \right),
\end{equation}
and 
\begin{equation}\label{TraceNormNewtonian}
\left\Vert \gamma N^{p} \right\Vert_{\mathcal{L}\left(\mathbb{L}^{2}(\Omega); \mathbb{H}^{\frac{1}{2}}(\partial \Omega) \right)} = \mathcal{O}\left( \frac{1}{P} \right).
\end{equation}
\end{lemma}
\begin{proof}
See \textbf{Subsection \ref{AS2331}}. 
\end{proof}
For the convergence of the series given into $(\ref{AV})$, we have
\begin{equation}\label{DirectSeries}
\left\Vert \sum_{j \geq 2} \left(K_{j} \overset{j}{\otimes} \left( n^{2} \right) \right) \right\Vert_{\mathbb{H}^{\frac{1}{2}}(\partial \Omega)}  \leq \sum_{j \geq 2} \left\Vert K_{j} \overset{j}{\otimes} \left( n^{2} \right) \right\Vert_{\mathbb{H}^{\frac{1}{2}}(\partial \Omega)}.
\end{equation}
Now, we estimate the terms appearing in the previous series. 
\begin{enumerate}
\item For $j=2$, 
\begin{eqnarray*}
\nonumber
\left\Vert K_{2} \overset{2}{\otimes} \left( n^{2} \right) \right\Vert_{\mathbb{H}^{\frac{1}{2}}(\partial \Omega)} &=& \left( \omega^{2} \right)^{2} \,\left\Vert \gamma N^{p}\left(n^{2} \,  N^{p}\left( n^{2} \, q^{f} \right) \right)\right\Vert_{\mathbb{H}^{\frac{1}{2}}(\partial \Omega)} \\ 
& \leq & \left( \omega^{2} \right)^{2} \, \left\Vert \gamma N^{p} \right\Vert_{\mathcal{L}(\mathbb{L}^{2}(\Omega);\mathbb{H}^{\frac{1}{2}}(\partial \Omega))} \; \left\Vert n^{2} \right\Vert^{2}_{\mathbb{L}^{\infty}(\Omega)} \; \left\Vert N^{p} \right\Vert_{\mathcal{L}(\mathbb{L}^{2}(\Omega);\mathbb{L}^{2}(\Omega))} \; \left\Vert q^{f}  \right\Vert_{\mathbb{L}^{2}(\Omega)}.
\end{eqnarray*}
\item For $j=3$, 
\begin{eqnarray*}
\left\Vert K_{3} \overset{3}{\otimes} \left( n^{2} \right) \right\Vert_{\mathbb{H}^{\frac{1}{2}}(\partial \Omega)} &=& (\omega^{2})^{3} \, \left\Vert \gamma N^{p}\left(n^{2} \,  N^{p}\left( n^{2} \, N^{p}\left(n^{2} \, q^{f} \right) \right)\right) \right\Vert_{\mathbb{H}^{\frac{1}{2}}(\partial \Omega)} \\ 
& \leq & (\omega^{2})^{3} \, \left\Vert \gamma N^{p} \right\Vert_{\mathcal{L}(\mathbb{L}^{2}(\Omega);\mathbb{H}^{\frac{1}{2}}(\partial \Omega))} \; \left\Vert n^{2} \right\Vert^{3}_{\mathbb{L}^{\infty}(\Omega)} \; \left\Vert N^{p} \right\Vert^{2}_{\mathcal{L}(\mathbb{L}^{2}(\Omega);\mathbb{L}^{2}(\Omega))} \; \left\Vert q^{f}  \right\Vert_{\mathbb{L}^{2}(\Omega)}.  
\end{eqnarray*}
\item For an arbitrary $j$, by induction, we can prove that 
\begin{equation}\label{EstimationKj}
\left\Vert K_{j} \overset{j}{\otimes} \left( n^{2} \right) \right\Vert_{\mathbb{H}^{\frac{1}{2}}(\partial \Omega)} \leq \; \bm{\Xi} \;  \, \left\Vert n^{2} \right\Vert_{\mathbb{L}^{\infty}(\Omega)} \; \left( \omega^{2} \, \left\Vert  N^{p} \right\Vert_{\mathcal{L}(\mathbb{L}^{2}(\Omega);\mathbb{L}^{2}(\Omega))} \; \left\Vert n^{2} \right\Vert_{\mathbb{L}^{\infty}(\Omega)} \right)^{j-1},
\end{equation}
where 
\begin{equation}\label{ConstantXi}
\bm{\Xi} := \omega^{2}  \, \left\Vert  q^{f}  \right\Vert_{\mathbb{L}^{2}(\Omega)} \left\Vert  \gamma N^{p} \right\Vert_{\mathcal{L}(\mathbb{L}^{2}(\Omega);\mathbb{H}^{\frac{1}{2}}(\partial \Omega))}. 
\end{equation}
\end{enumerate}
Therefore, by going back to $(\ref{DirectSeries})$ and using the estimation $(\ref{EstimationKj})$, we obtain
\begin{eqnarray}\label{||uf-vf||}
\nonumber
\left\Vert \sum_{j \geq 2} \left(K_{j} \overset{j}{\otimes} \left( n^{2} \right) \right) \right\Vert_{\mathbb{H}^{\frac{1}{2}}(\partial \Omega)} & \leq & \sum_{j \geq 2} \; \bm{\Xi} \; \left\Vert n^{2} \right\Vert_{\mathbb{L}^{\infty}(\Omega)} \; \left( \omega^{2} \; \left\Vert N^{p} \right\Vert_{\mathcal{L}(\mathbb{L}^{2}(\Omega);\mathbb{L}^{2}(\Omega))}  \; \left\Vert n^{2} \right\Vert_{\mathbb{L}^{\infty}(\Omega)} \right)^{j-1} \\ &=& \bm{\Xi} \; \; \left\Vert n^{2} \right\Vert^{2}_{\mathbb{L}^{\infty}(\Omega)} \;  \omega^{2} \; \left\Vert N^{p} \right\Vert_{\mathcal{L}(\mathbb{L}^{2}(\Omega);\mathbb{L}^{2}(\Omega))} \; \sum_{j \geq 0}  \bm{\kappa}^{j},
\end{eqnarray}
where $\bm{\kappa}$ is the parameter given by   $ \;\;   \bm{\kappa} \, := \, \omega^{2} \; \left\Vert n^{2} \right\Vert_{\mathbb{L}^{\infty}(\Omega)} \left\Vert N^{p} \right\Vert_{\mathcal{L}(\mathbb{L}^{2}(\Omega);\mathbb{L}^{2}(\Omega))}.$
Under the condition 
\begin{equation}\label{CdtCvgS}
\bm{\kappa} \; < \; 1,
\end{equation}
the previous series converges. Now, because $\left\Vert N^{p} \right\Vert_{\mathcal{L}(\mathbb{L}^{2}(\Omega);\mathbb{L}^{2}(\Omega))} = \mathcal{O}\left(P^{-2} \right)$, see $(\ref{NormNewtonian})$, then with $P$ large enough; knowing that $\omega^{2} \; \left\Vert n^{2} \right\Vert_{\mathbb{L}^{\infty}(\Omega)} $ is a bounded term, we deduce that the condition $(\ref{CdtCvgS})$ is satisfied. In addition, from $(\ref{||uf-vf||})$, we have
\begin{eqnarray}\label{DSuf-vf}
\nonumber
\left\Vert \sum_{j \geq 2} \left(K_{j} \overset{j}{\otimes} \left( n^{2} \right) \right) \right\Vert_{\mathbb{H}^{\frac{1}{2}}(\partial \Omega)} &=& \mathcal{O}\left( \bm{\Xi} \; \left\Vert N^{p} \right\Vert_{\mathcal{L}(\mathbb{L}^{2}(\Omega);\mathbb{L}^{2}(\Omega))} \right) \\ \nonumber & \overset{(\ref{ConstantXi})}{=} & \mathcal{O}\left(\left\Vert N^{p} \right\Vert_{\mathcal{L}(\mathbb{L}^{2}(\Omega);\mathbb{L}^{2}(\Omega))} \, \left\Vert  q^{f}  \right\Vert_{\mathbb{L}^{2}(\Omega)} \, \left\Vert  \gamma N^{p}\right\Vert_{\mathcal{L}(\mathbb{L}^{2}(\Omega);\mathbb{H}^{\frac{1}{2}}(\partial \Omega))} \right) \\ & \overset{\textbf{Lemma \; \ref{LemmaNp}}}{=} & \mathcal{O}\left(\left\Vert  q^{f}  \right\Vert_{\mathbb{L}^{2}(\Omega)} \, \frac{1}{P^{3}} \right).
\end{eqnarray}
The coming lemma is important to get an estimation of $\left\Vert \displaystyle\sum_{j \geq 2} \left(K_{j} \overset{j}{\otimes} \left( n^{2} \right) \right) \right\Vert_{\mathbb{H}^{\frac{1}{2}}(\partial \Omega)}$, with respect to the data $f(\cdot)$ and the parameter $P$. 
\begin{lemma}\label{Lemma22}
The function $q^{f}(\cdot)$, solution of $(\ref{Equaqf})$, satisfies: 
\begin{equation}\label{Equa214}
\left\Vert q^{f} \right\Vert_{\mathbb{L}^{2}(\Omega)} = \mathcal{O}\left( \left\Vert f \right\Vert_{\mathbb{H}^{-\frac{1}{2}}(\partial \Omega)} \, \frac{1}{P}  \right).
\end{equation}
\end{lemma}
\begin{proof}
The solution $q^{f}(\cdot)$ to the problem $(\ref{Equaqf})$ can be represented as $q^{f}(x) \, = \, SL^{p}\left( f \right)(x)$, for $x \, \in \, \Omega$, where $SL^{p}\left( \cdot \right)$ is the Single-Layer operator defined, from $\mathbb{H}^{-\frac{1}{2}}\left( \partial \Omega \right)$ to $\mathbb{H}^{1}\left(\Omega \right)$, by
\begin{equation}\label{Equa0848}
SL^{p}\left( f \right)(x) \, := \, \int_{\partial \Omega} G_{p}(x,y) \, f(y) \, d\sigma(y), \quad x \in \Omega,
\end{equation}
with $G_{p}(\cdot,\cdot)$ is the Green's kernel solution of $(\ref{Green's-Kernel-with-P})$. 
It is clear that $\left( \Delta - \, P^{2} \right) \, q^{f} \, = \, 0$, in $\Omega$. In \textbf{Subsection \ref{JSLp}}, we show that $\partial_\nu SL^p(f)=f$ on $\partial \Omega$.
Multiplying the previous equation by $\overline{q^{f}}(\cdot)$ and integrating over $\Omega$, gives us: 
\begin{eqnarray*}
\left\Vert q^{f} \right\Vert^{2}_{\mathbb{L}^{2}(\Omega)}  =   \left\langle q^{f} ; SL^{p}\left( f \right) \right\rangle_{\mathbb{L}^{2}(\Omega)} &=& 
\left\langle f ; \gamma N^{p} \left( q^{f} \right) \right\rangle_{\mathbb{H}^{-\frac{1}{2}}(\partial \Omega) \times \mathbb{H}^{\frac{1}{2}}(\partial \Omega)}  \\
& \leq & \left\Vert f \right\Vert_{\mathbb{H}^{-\frac{1}{2}}(\partial \Omega)} \, \left\Vert \gamma N^{p} \left( q^{f} \right) \right\Vert_{\mathbb{H}^{\frac{1}{2}}(\partial \Omega)} \\ & \leq & \left\Vert f \right\Vert_{\mathbb{H}^{-\frac{1}{2}}(\partial \Omega)} \, \left\Vert \gamma N^{p} \right\Vert_{\mathcal{L}(\mathbb{L}^{2}(\Omega);\mathbb{H}^{\frac{1}{2}}(\partial \Omega))}  \, \left\Vert q^{f} \right\Vert_{\mathbb{L}^{2}(\Omega)}.
\end{eqnarray*}
Then, 
\begin{equation*}
\left\Vert q^{f} \right\Vert_{\mathbb{L}^{2}( \Omega)} \, \leq \,   \left\Vert f \right\Vert_{\mathbb{H}^{-\frac{1}{2}}(\partial \Omega)} \, \left\Vert \gamma N^{p} \right\Vert_{\mathcal{L}(\mathbb{L}^{2}(\Omega);\mathbb{H}^{\frac{1}{2}}(\partial \Omega))}   \overset{(\ref{TraceNormNewtonian})}{=} \mathcal{O}\left( \left\Vert f \right\Vert_{\mathbb{H}^{-\frac{1}{2}}(\partial \Omega)} \, \frac{1}{P}  \right).
\end{equation*}
This concludes the proof of \textbf{Lemma \ref{Lemma22}}. 
\end{proof}
Using $(\ref{Equa214})$, the estimation $(\ref{DSuf-vf})$ becomes
\begin{equation*}
\left\Vert \sum_{j \geq 2} \left(K_{j} \overset{j}{\otimes} \left( n^{2} \right) \right) \right\Vert_{\mathbb{H}^{\frac{1}{2}}(\partial \Omega)} =  \mathcal{O}\left(\left\Vert  f  \right\Vert_{\mathbb{H}^{-\frac{1}{2}}(\partial \Omega)} \, \frac{1}{P^{4}} \right).
\end{equation*}
Hence, from $(\ref{AV})$, we get 
\begin{equation}\label{ASMTV0521}
u^{f}(x) - q^{f}(x) = \omega^{2} \, \gamma \left( \bm{W}^{q^{f}} \right) (x) + \mathcal{O}\left(\left\Vert  f  \right\Vert_{\mathbb{H}^{-\frac{1}{2}}(\partial \Omega)} \, \frac{1}{P^{4}} \right), \quad x \in \partial \Omega,
\end{equation}
where $\bm{W}^{q^{f}} = N^{p}\left(n^{2} \,  q^{f} \right)$ is the function satisfying
\begin{align}\label{ASKD}
\begin{cases}  
\left(  \Delta -  P^{2} \, I \right) \bm{W}^{q^{f}} \, = - \, n^{2} \, q^{f} \quad \text{in} \quad  \Omega,  \\ 
\qquad \qquad \, \partial_{\nu} \bm{W}^{q^{f}} = 0  \quad \quad \; \quad  \text{on} \quad \partial \Omega. 
\end{cases}
\end{align}
Because on the boundary $\partial \Omega$, we have $u^{f} = \Lambda_{P}\left(\partial_{\nu} u^{f} \right) \overset{(\ref{EquaWf})}{=} \Lambda_{P}\left( f \right)$ and by plugging it into $(\ref{ASMTV0521})$ we derive $(\ref{ASMTV})$. This concludes the proof of \textbf{Theorem \ref{THMLinearization}}.
\section{Proof of Theorem \ref{THMReconstruction}}\label{Construction-C}
The purpose of this section is to explain how the linearized NtD map (measured on the boundary $\partial \Omega$) can be utilized with CGO solutions to reconstruct the Fourier coefficients associated with the unknown refraction index $n^{2}(\cdot)$. Hence, we reconstruct the refraction index of $n^{2}(\cdot)$ inside $\Omega$ as a discrete series expansion using the reconstructed Fourier coefficients.  
From the previous section, we deduce that measuring $u^{f}(\cdot) - q^{f}(\cdot)$ means measuring, approximately, $\bm{W}^{q^{f}}(\cdot)$, on the boundary $\partial \Omega$. We set $q^{g}(\cdot)$ to be the solution of 
\begin{align}\label{Equavg}
\begin{cases}  
\left( \Delta -  P^{2} \, I \right) q^{g} \, = 0 \quad \text{in} \quad  \Omega,  \\ 
\qquad \, \quad \quad \partial_{\nu} q^{g} = g  \;\;\;  \text{on} \quad \partial \Omega. 
\end{cases}
\end{align}
Multiplying the first equation of $\left( \ref{ASKD} \right)$ with $q^{g}$, solution of $(\ref{Equavg})$, and integrating over $\Omega$, we get:
\begin{equation*}
\langle \nabla \bm{W}^{q^{f}}; \nabla q^{g} \rangle_{\mathbb{L}^{2}(\Omega)} + P^{2} \, \langle  \bm{W}^{q^{f}}; q^{g} \rangle_{\mathbb{L}^{2}(\Omega)} = \langle n^{2} \, q^{f}; q^{g} \rangle_{\mathbb{L}^{2}(\Omega)}.
\end{equation*}
Moreover, by multiplying $\left( \ref{Equavg} \right)$ with $\bm{W}^{q^{f}}$, solution of $(\ref{ASKD})$, and integrating over $\Omega$, we get  
\begin{equation*}
\langle \nabla \bm{W}^{q^{f}}; \nabla q^{g} \rangle_{\mathbb{L}^{2}(\Omega)} + P^{2} \, \langle  \bm{W}^{q^{f}}; q^{g} \rangle_{\mathbb{L}^{2}(\Omega)}  = \langle \bm{W}^{q^{f}};g \rangle_{\mathbb{H}^{\frac{1}{2}}(\partial \Omega) \times \mathbb{H}^{-\frac{1}{2}}(\partial \Omega)}.
\end{equation*}
Then, by subtracting the two previous equations we end up with 
\begin{equation}\label{Equavfvg}
\langle \bm{W}^{q^{f}};g \rangle_{\mathbb{H}^{\frac{1}{2}}(\partial \Omega) \times \mathbb{H}^{-\frac{1}{2}}(\partial \Omega)} = \langle n^{2} \, q^{f}; q^{g} \rangle_{\mathbb{L}^{2}(\Omega)}, \;  \quad \forall \, \left( f, g \right) \in \mathbb{H}^{-\frac{1}{2}}(\partial \Omega) \times \mathbb{H}^{-\frac{1}{2}}(\partial \Omega).
\end{equation}
Knowing that $\bm{W}^{q^{f}}$ can be measured, on the boundary $\partial \Omega$, and $g$ is a data function, we deduce that the L.H.S is a known term. The goal is then to reconstruct $\, n^{2}(\cdot) \,$, inside $\Omega$. To achieve this, we start by fixing $\eta \in \mathbb{R}^{3}$ and choosing $\xi \in \mathbb{C}^{3}$ such that
\begin{equation}\label{cdtxixi}
\xi \cdot \xi = 0. 
\end{equation}
We set $q^{f}(\cdot)$ the function defined by 
\begin{equation}\label{Defvf}
q^{f}(x) := e^{i \, \xi \cdot x} \, \left( e^{i \,  x \cdot \eta} \, + r_{1}(x) \right), \quad x \in \mathbb{R}^{3},
\end{equation}
where $\xi$ is chosen such that\footnote{For every fixed $\eta \in \mathbb{R}^{3}$, we choose $\xi \in \mathbb{C}^{3}$ such that $(\ref{cdtxixi})$ and $(\ref{cdtetaxi})$ will be fulfilled. Such $\xi$ exists, see $(\ref{Existencexi})$.} 
\begin{equation}\label{cdtetaxi}
\eta \cdot \xi = 0,
\end{equation}
and $r_{1}(\cdot)$ is such that 
\begin{equation}\label{DiffEquar1}
\left(\Delta + 2 \, i \, \xi \cdot \nabla - P^{2} \right) r_{1}(x) = \left(\left\vert \eta \right\vert^{2} + P^{2} \right) \; e^{i \, x \cdot \eta}, \quad \text{in} \quad \Omega.
\end{equation}
Observe that the R.H.S is depending on $\eta$ and $P$,  then  $r_{1}(\cdot)$ will also depends on both $\eta$ and $P$. Later, to mark this dependence, we note $r_{1,\eta,p}(\cdot)$ instead of $r_{1}(\cdot)$. Thanks to \cite[Theorem 3.8]{salo2008calderon}, we know that under the condition 
\begin{equation}\label{Cdtxi}
\left\vert \xi \right\vert \geq \max\left( C_{0} \, P^{2} ; 1 \right) = \, C_{0} \, P^{2}, 
\end{equation}
where the last equality is a consequence of the fact that $P \gg 1$,  and $C_{0}$ is a constant depending on $\Omega$, the equation $(\ref{DiffEquar1})$
has a solution $r_{1,\eta,p}(\cdot) \in \mathbb{H}^{1}(\Omega)$ satisfying 
\begin{equation}\label{Estimationr1}
\left\Vert r_{1,\eta,p} \right\Vert_{\mathbb{L}^{2}(\Omega)} \leq  \frac{C_{0}}{\left\vert \xi \right\vert} \; \left(\left\vert \eta \right\vert^{2} + P^{2} \right) \;  \left\vert \Omega \right\vert^{\frac{1}{2}} \quad \text{and} \quad \left\Vert \nabla r_{1,\eta,p} \right\Vert_{\mathbb{L}^{2}(\Omega)} \leq  C_{0} \; \left(\left\vert \eta \right\vert^{2} + P^{2} \right) \;  \left\vert \Omega \right\vert^{\frac{1}{2}}.
\end{equation} 
In the same manner we set $q^{g}(\cdot)$ to be the function defined by 
\begin{equation}\label{Defvg}
q^{g}(x) := e^{- i \, \xi \cdot x} \left( 1 + r_{2}(x) \right), \quad x \in \mathbb{R}^{3},
\end{equation}
where $r_{2}(\cdot)$ is such that 
\begin{equation}\label{DiffEquar2}
\left(- \Delta + 2 \, i \, \xi \cdot \nabla + P^{2} \right) r_{2}(x) = - \, P^{2}  \; , \quad \text{in} \quad \Omega.
\end{equation}
Because the R.H.S is depending on $P$, the solution $r_{2}(\cdot)$ will also depends on  $P$. Later, to mark this dependence, we note $r_{2,p}(\cdot)$ instead of $r_{2}(\cdot)$.
Again, thanks to \cite[Theorem 3.8]{salo2008calderon}, we know that under the condition $(\ref{Cdtxi})$, the equation $(\ref{DiffEquar2})$
has a solution $r_{2,p}(\cdot) \in \mathbb{H}^{1}(\Omega)$, satisfying 
\begin{equation}\label{Estimationr2}
\left\Vert r_{2,p} \right\Vert_{\mathbb{L}^{2}(\Omega)} \leq \frac{C_{0}}{\left\vert \xi \right\vert} \;  P^{2}  \;  \left\vert \Omega \right\vert^{\frac{1}{2}} \quad \text{and} \quad \left\Vert \nabla r_{2,p} \right\Vert_{\mathbb{L}^{2}(\Omega)} \leq  C_{0} \;  P^{2}  \;  \left\vert \Omega \right\vert^{\frac{1}{2}}.
\end{equation} 
Now, we take unit vectors $\omega_{1}$ and $\omega_{2}$ in $\mathbb{R}^{3}$ such that $\{ \omega_{1}; \omega_{2}; \eta \}$ is an orthogonal set. In addition, we choose $\xi = s \left(\omega_{1} + i \, \omega_{2} \right)$, so that $\left\vert \xi \right\vert = s \, \sqrt{2}$ and $\xi \cdot \xi = 0$. Using the fact that $P \gg 1$ and taking the parameter $s$ sufficiently large, such that $(\ref{Cdtxi})$ will be satisfied, we reduce the estimation of the $\mathbb{L}^{2}(\Omega)-$norm of $r_{1,\eta,p}(\cdot)$ and $r_{2,p}(\cdot)$ to
\begin{equation}\label{Estimationrj}
\left\Vert r_{1,\eta,p} \right\Vert_{\mathbb{L}^{2}(\Omega)} = \mathcal{O}\left( \frac{P^{2}}{s} \right) \quad \text{and} \quad \left\Vert r_{2,p} \right\Vert_{\mathbb{L}^{2}(\Omega)} = \mathcal{O}\left( \frac{P^{2}}{s} \right).
\end{equation} 
Next, by taking the product between $q^{f}(\cdot)$, given by $(\ref{Defvf})$, and $q^{g}(\cdot)$, given by $(\ref{Defvg})$, we obtain 
\begin{equation}\label{CGOvfvg}
\left( q^{f} \cdot q^{g} \right)(x) = e^{i \,  x \cdot \eta} + r_{1,\eta ,p}(x) + e^{i \,  x \cdot \eta} \, r_{2,p}(x) + r_{1, \eta, p}(x) r_{2,p}(x),
\end{equation}
and we would like to choose the solution in such a way that $\left( q^{f} \cdot q^{g} \right)(\cdot)$ is close to $e^{i \, \cdot \cdot \eta}$, since the functions $\left\{ e^{i \, \cdot \cdot \eta} \right\}$ form a dense set, see \cite[Theorem 1.1]{VI}, in $\mathbb{L}^{1}(\Omega)$. By going back to $(\ref{Equavfvg})$, we have 
\begin{equation*}
\langle \bm{W}^{q^{f}};g \rangle_{\mathbb{H}^{\frac{1}{2}}(\partial \Omega) \times \mathbb{H}^{-\frac{1}{2}}(\partial \Omega)} = \int_{\Omega} \, n^{2}(x) \; q^{f}(x) \; q^{g}(x) \; dx 
 \overset{(\ref{CGOvfvg})}{=}  \int_{\Omega} \, n^{2}(x) \; e^{i \,  x \cdot \eta}  \; dx + Error(\eta ,p),
\end{equation*}
where
\begin{equation*}
Error(\eta ,p) := \int_{\Omega} \, n^{2}(x) \;  r_{1, \eta ,p}(x)  \; dx + \int_{\Omega} \, n^{2}(x) \;  e^{i \,  x \cdot \eta} \, r_{2, p}(x) \; dx + \int_{\Omega} \, n^{2}(x) \;  r_{1, \eta , p}(x) r_{2, p}(x) \; dx,
\end{equation*}
which can be estimated as
\begin{equation*}
\left\vert Error(\eta ,p) \right\vert 
 \leq  \left\Vert n^{2} \right\Vert_{\mathbb{L}^{\infty}(\Omega)} \left[ \left\Vert r_{1, \eta ,p} \right\Vert_{\mathbb{L}^{2}(\Omega)} \, \left\vert \Omega \right\vert^{\frac{1}{2}} + \left\Vert r_{2, p} \right\Vert_{\mathbb{L}^{2}(\Omega)} \, \left\vert \Omega \right\vert^{\frac{1}{2}} +  \left\Vert r_{1, \eta ,p} \right\Vert_{\mathbb{L}^{2}(\Omega)} \, \left\Vert r_{2, p} \right\Vert_{\mathbb{L}^{2}(\Omega)} \right],
\end{equation*}
which, based on $(\ref{Estimationr1})$ and $(\ref{Estimationr2})$, can be reduced to
\begin{equation}\label{GzGz}
\left\vert Error(\eta ,p) \right\vert \lesssim  \frac{\left( \left\vert \eta \right\vert^{2} + P^{2} \right)}{\left\vert \xi \right\vert}  \overset{(\ref{Estimationrj})}{=}  \mathcal{O}\left( \frac{P^{2}}{s}  \right) =  \mathcal{O}\left( \frac{P^{2}}{\left\vert \xi \right\vert}  \right).
\end{equation}
Moreover, based on its construction, see $(\ref{Defvf})$, the function $q^{f}(\cdot)$ depends on $\eta$ and this implies the dependency of $\bm{W}^{q^{f}}(\cdot)$ with respect to $\eta$. We mark explicitly this dependence and we write:
\begin{equation}\label{FC}
\langle \bm{W}^{q^{f}}_{\eta} ; g \rangle_{\mathbb{H}^{\frac{1}{2}}(\partial \Omega) \times \mathbb{H}^{-\frac{1}{2}}(\partial \Omega)} - Error(\eta ,p) =  \int_{\Omega} \, n^{2}(x) \; e^{i \,  x \cdot \eta}  \; dx,  
\end{equation}
which is valid in $\Lambda_{\eta} := \left\{ \xi \in \mathbb{C}^{3},\quad \text{such that} \quad \left\vert \xi \right\vert \gg 1 \text{,} \quad \xi \cdot \xi = 0 \quad \text{and} \quad \xi \cdot \eta = 0 \right\},$
where $\eta$ is fixed in $\mathbb{R}^{3}$. The set $\Lambda_{\eta}$ is not empty, see $(\ref{Existencexi})$. By restricting $\eta$ to $\mathbb{Z}^{3}$, i.e. $\eta = - \ell$ with $\ell \in \mathbb{Z}^{3}$, we rewrite $(\ref{FC})$ as 
\begin{equation}\label{FCk}
\langle \bm{W}^{q^{f}}_{-\ell};g \rangle_{\mathbb{H}^{\frac{1}{2}}(\partial \Omega) \times \mathbb{H}^{-\frac{1}{2}}(\partial \Omega)} - Error(-\ell, p) =  \int_{\Omega} \, n^{2}(x) \; e^{- \, i \,  x \cdot \ell}  \; dx = \left( 2 \pi \right)^{3} \, \mathcal{F}\left(n^{2} \, \underset{\Omega}{\chi}\right)(\ell),  
\end{equation}
where $\mathcal{F}(\cdot)$ is the 3D-Fourier transform operator\footnote{We recall that we have 
\begin{equation*}
\mathcal{F}(f)(\ell) := \left( 2 \, \pi \right)^{-3} \, \int_{\mathbb{R}^{3}} f(x) \, e^{- i \, \ell \cdot x} \; dx, \quad \ell \in \mathbb{Z}^{3}.
\end{equation*}
}. Now, thanks to \cite[Theorem 2.3]{salo2008calderon}, we know that  
\begin{equation*}
n^{2}(x) = \sum_{\ell \in \mathbb{Z}^{3}} \mathcal{F}\left(n^{2} \, \underset{\Omega}{\chi}\right)(\ell) \; e^{i \, \ell \cdot x}, \quad \, x \in \Omega,
\end{equation*} 
with convergence in the $\mathbb{L}^{2}(\Omega)$-norm. Then, by gathering the previous expression and $(\ref{FCk})$, we end up with
\begin{equation}\label{NR0821}
n^{2}(x)  =   \left(2 \, \pi \right)^{-3} \; \sum_{\ell \in \mathbb{Z}^{3}} \int_{\partial \Omega} \bm{W}^{q^{f}}_{-\ell}(x) \; g(x) \; d\sigma(x)    \;\; e^{i \, \ell \cdot x} + \textbf{Error(x,p)},
\end{equation}
in the $\mathbb{L}^2(\Omega)$ sense, 
where $\textbf{Error(x,p)}$ is a trigonometric series given by 
\begin{equation*}
\textbf{Error(x,p)} := - \left(2 \, \pi \right)^{-3} \; \sum_{\ell \in \mathbb{Z}^{3}}  Error(-\ell , p) \; e^{i \, \ell \cdot x}, \quad x \in \Omega.
\end{equation*}
Next, we estimate the $\mathbb{L}^{2}(\Omega)$ norm of $\textbf{Error($\cdot$ ,p)}$. We have, 
\begin{equation*}
\left\Vert \textbf{Error($\cdot$ ,p)} \right\Vert_{\mathbb{L}^{2}(\Omega)} \lesssim  \sum_{\ell \in \mathbb{Z}^{3}} \left\vert  Error(-\ell , p) \right\vert \overset{(\ref{GzGz})}{\lesssim}  \sum_{\ell \in \mathbb{Z}^{3}} \frac{\left( \left\vert \ell \right\vert^{2} + P^{2} \right)}{\left\vert \xi \right\vert}.
\end{equation*}
At this stage, we recall that for every fixed $\ell \in \mathbb{Z}^{3}$, we choose $\xi \in \mathbb{C}^{3}$ such that 
\begin{equation*}
\xi \cdot \xi = 0, \;\; \ell \cdot \xi = 0 \;\; \text{and} \;\; \left\vert \xi \right\vert \gg 1.
\end{equation*}
Such $\xi$ exists,  see $(\ref{Existencexi})$. Without loss of generality, we take $\xi$ satisfying $(\ref{normxi})$, hence $\left\vert \xi \right\vert = P^{2+\varsigma} \, \left\vert \ell \right\vert^{3+\varsigma}$, with $\varsigma \in \mathbb{R}^{+}$. Then, 
\begin{equation*}
\left\Vert \textbf{Error($\cdot$ ,p)} \right\Vert_{\mathbb{L}^{2}(\Omega)} \lesssim  \sum_{\ell \in \mathbb{Z}^{3}} \frac{\left( \left\vert \ell \right\vert^{2} + P^{2} \right)}{P^{2+\varsigma} \, \left\vert \ell \right\vert^{3+\varsigma}} = P^{-2-\varsigma} \sum_{\ell \in \mathbb{Z}^{3}} \frac{1}{\, \left\vert \ell \right\vert^{1+\varsigma}} + P^{-\varsigma} \sum_{\ell \in \mathbb{Z}^{3}} \frac{1}{ \left\vert \ell \right\vert^{3+\varsigma}}.
\end{equation*}
After that, we use the convergence of the two previous series to reduce the last estimation to
\begin{equation*}
\left\Vert \textbf{Error($\cdot$ ,p)} \right\Vert_{\mathbb{L}^{2}(\Omega)} = \mathcal{O}\left( P^{-\varsigma} \right).
\end{equation*}
Hence, $(\ref{NR0821})$ becomes, 
\begin{equation*}\label{NR0836}
n^{2}(x)  =   \left(2 \, \pi \right)^{-3} \; \sum_{\ell \in \mathbb{Z}^{3}} \langle \bm{W}^{q^{f}}_
{-\ell};g \rangle_{\mathbb{H}^{\frac{1}{2}}(\partial \Omega) \times \mathbb{H}^{-\frac{1}{2}}(\partial \Omega)}    \;\; e^{i \, \ell \cdot x} +\mathcal{O}\left( P^{-\varsigma} \right), 
\end{equation*}
in the $\mathbb{L}^2(\Omega)$ sense. This ends the proof of \textbf{Theorem \ref{THMReconstruction}}.   
\addtocontents{toc}{\setcounter{tocdepth}{1}}
\section{Proof of Theorem \ref{principal-Thm}}\label{effective-NtD}
The purpose of this section is to prove Theorem \ref{principal-Thm}. To ensure easy reading of the proof, we have divided this section into four subsections. The goal of the first subsection is to extract the dominant term of  
\begin{equation*}
\bm{I_1} := \omega^{2} \, \left\langle \left(\frac{\rho_{1}}{k_{1}} \, - \, n^{2} \right) v^{g}; p^{f} \right\rangle_{\mathbb{L}^{2}(D)}, 
\end{equation*}
where we prove that 
\begin{equation*}
\bm{I}_{1}  =  \omega^{2} \, \dfrac{\rho_{1}}{k_{1}}  \, \sum_{j=1}^{M} \, p^{f}(z_{j}) \, \int_{D_{j}}   v^{g}_{j}(x) \,  dx + Error, 
\end{equation*}
see $(\ref{I1-formula-Intro})$. In the second subsection we derive and we justify the invertibility of the discrete  algebraic system satisfied by the vector $\left( \int_{D_{j}}   v^{g}_{j}(x) \,  dx \right)_{j=1,\cdots,M}$, contained in $\bm{I_1}$, see $(\ref{0820})$ and \textbf{Lemma \ref{MTR}}. The third subsection consists in writing down the L.S.E, satisfied by $u^{g}(\cdot)$, where $u^{g}(\cdot)$ is the function appearing in 
\begin{equation*}
\bm{I_2} := - \, P^{2} \, \langle u^{g}; p^{f} \rangle_{\mathbb{L}^{2}(\Omega)}, 
\end{equation*}
see $(\ref{NewL.S.E})$. Then, we prove that the discrete algebraic system can approximate the continuous L.S.E, see $(\ref{maxY-Y})$. The goal of the last subsection lies in the justification of the convergence of $\bm{I_1}$ to $\bm{I_2}$ for a large number of droplets, that is, $M \gg 1$.
\medskip
\newline
To avoid making this section heavy and cumbersome, we have noted six lemmas without proofs. The proof of each lemma can be found in \textbf{Section \ref{Appendix}}.
\subsection{Extraction of the dominant term of $I_{1}$}
We set
\begin{equation*}
\bm{I}_{1} := \omega^{2} \, \dfrac{\rho_{1}}{k_{1}} \, \, \langle v^{g}; p^{f} \rangle_{\mathbb{L}^{2}(D)} \, - \, \omega^{2}  \, \langle n^{2} \, v^{g}; p^{f} \rangle_{\mathbb{L}^{2}(D)} = \omega^{2} \, \dfrac{\rho_{1}}{k_{1}}  \, \sum_{j=1}^{M} \, \int_{D_{j}}   v^{g}_{j}(x) \, p^{f}(x) \, dx\,- \, \omega^{2}  \, \langle n^{2} \, v^{g}; p^{f} \rangle_{\mathbb{L}^{2}(D)},
\end{equation*}
where $v^{g}(\cdot)$ satisfies $(\ref{Equavf})$, $p^{f}(\cdot)$ is solution of $(\ref{EquaKg-introdution})$ and we have used the notation $v^{g}_{j}(\cdot) \, := \, v^{g}_{|_{D_{j}}}(\cdot),$ for $j=1,\cdots,M$. In addition, as the coefficients $n^{2}(\cdot)$ is $\mathbb{W}^{1, \infty}$-regular, then $p^f(\cdot)$, which is in $\mathbb{H}^{1}(\Omega)$, enjoys a $\mathbb{W}^{2, \infty}$-interior regularity. Based on this, we use Taylor expansion near the centers, $z_j$, to get 
\begin{equation}\label{I1=Dp+J1}
\bm{I}_{1}  \, = \, \omega^{2} \, \dfrac{\rho_{1}}{k_{1}}  \, \sum_{j=1}^{M} \, p^{f}(z_{j}) \, \int_{D_{j}}   v^{g}_{j}(x) \,  dx \, + \, J_{1}, 
\end{equation}
where 
\begin{equation*}
    J_{1} \, := \,  \omega^{2} \, \dfrac{\rho_{1}}{k_{1}} \, \sum_{j=1}^{M} \, \int_{D_{j}}   v^{g}_{j}(x) \, \int_{0}^{1} \nabla p^{f}(z_{j}+t(x-z_{j})) \cdot (x-z_{j}) \, dt \, dx \, - \, \omega^{2}  \, \langle n^{2} \, v^{g}; p^{f} \rangle_{\mathbb{L}^{2}(D)}.
\end{equation*}
We estimate the term $J_{1}$ as 
\begin{eqnarray}\label{EstJ1}
\nonumber
\left\vert J_{1} \right\vert 
& \lesssim & a^{-2} \, \sum_{j=1}^{M} \left\Vert v^{g}_{j} \right\Vert_{\mathbb{L}^{2}(D_{j})} \; \left\Vert \int_{0}^{1}  \nabla p^{f}(z_{j}+t(\cdot - z_{j})) \cdot ( \cdot - z_{j}) \, dt \,  \right\Vert_{\mathbb{L}^{2}(D_{j})} \, +  \, \left\Vert v^{g} \right\Vert_{\mathbb{L}^{2}(D)} \; \left\Vert p^{f} \right\Vert_{\mathbb{L}^{2}(D)} \\ \nonumber
& \leq & a^{-2} \, \left\Vert v^{g} \right\Vert_{\mathbb{L}^{2}(D)} \; \left( \sum_{j=1}^{M}  \left\Vert \int_{0}^{1}  \nabla p^{f}(z_{j}+t(\cdot - z_{j})) \cdot ( \cdot - z_{j}) \, dt \,  \right\Vert^{2}_{\mathbb{L}^{2}(D_{j})} \right)^{\frac{1}{2}} +  \, \left\Vert v^{g} \right\Vert_{\mathbb{L}^{2}(D)} \; \left\Vert p^{f} \right\Vert_{\mathbb{L}^{2}(D)} \\
& = & \mathcal{O}\left( \left\Vert v^{g} \right\Vert_{\mathbb{L}^{2}(D)} \; \left[   a^{-1} \;   \left\Vert \nabla p^{f} \,  \right\Vert_{\mathbb{L}^{2}(D)} \, + \, \left\Vert p^{f} \,  \right\Vert_{\mathbb{L}^{2}(D)} \right]  \right).
\end{eqnarray}
Moreover, based on $(\ref{EquaKg-introdution})$ we deduce that $p^{f}(\cdot)$ can be represented as a Single-Layer with density $f(\cdot)$, i.e.,
\begin{equation}\label{pf=Slf}
    p^{f}(x) \, = \, \mathscr{S}(f)(x) \, := \, \int_{\partial \Omega} G(x,y) \, f(y) \, d\sigma(y) \, = \, \langle G(x,\cdot); f \rangle_{\mathbb{H}^{\frac{1}{2}}(\partial \Omega) \times \mathbb{H}^{-\frac{1}{2}}(\partial \Omega)}, \quad x \in \Omega,
\end{equation}
where $G(\cdot,\cdot)$ is the Green's kernel defined by 
\begin{align}\label{EquaGKernel}
\begin{cases}  
\underset{x}{\Delta} G(x,y) \, + \, \omega^{2} \, n^{2}(x) \, G(x,y) \, = \, - \, \underset{y}{\boldsymbol{\delta}}(x) \quad \text{in} \quad  \Omega,  \\ 
\qquad \qquad \qquad \qquad \partial_{\nu_{x}} G(x,y) \, = \, \quad 0 \;\;\;  \text{on} \quad \partial \Omega. 
\end{cases}
\end{align}
    The existence, uniqueness of $G(\cdot,\cdot)$ and its singularity analysis, with point-wise estimates, can be found in \cite{McLean}. Based on $(\ref{pf=Slf})$, we have 
\begin{equation}\label{InPa}
 \left\Vert p^{f} \right\Vert_{\mathbb{L}^{2}(D)} \, \leq \, \left\Vert f \right\Vert_{\mathbb{H}^{-\frac{1}{2}}(\partial \Omega)} \, \left[\int_{D} \left\Vert G(x,\cdot) \right\Vert^{2}_{\mathbb{H}^{\frac{1}{2}}(\partial \Omega)} \, dx \right]^{\frac{1}{2}}. 
\end{equation}
In addition, we have 
\begin{equation*}
    \left\Vert G(x,\cdot)   \right\Vert_{\mathbb{H}^{\frac{1}{2}}(\partial \Omega)} \, = \, \underset{\mathsf{G}(x,\cdot) \in \mathbb{H}^{1}(\Omega) \atop \mathsf{G}(x,\cdot)|_{\partial \Omega} \, = \, G(x,\cdot)}{Inf} \left\Vert \mathsf{G}(x,\cdot)   \right\Vert_{\mathbb{H}^{1}(\Omega)}.
\end{equation*}
Let the domain $\Omega^{\diamond} \, \equiv \, \Omega \setminus \overline{D}$, and let $\mathsf{G}(x,\cdot) \, := \, G(x,\cdot) \, \chi_{\Omega^{\diamond}}(\cdot)$. Thus, by its construction $\mathsf{G}(x,\cdot) \in \mathbb{H}^{1}(\Omega)$, for $x \in D$, and $\gamma\left( \mathsf{G}(x,\cdot)\right) \, = \, G(x,\cdot)|_{\partial \Omega}$, on $\partial \Omega$, where $\gamma(\cdot)$ is the trace operator. This implies that 
\begin{equation}\label{Equa1014IP}
    \left\Vert G(x,\cdot)   \right\Vert_{\mathbb{H}^{\frac{1}{2}}(\partial \Omega)} \, \leq \,  \left\Vert G(x,\cdot) \, \chi_{\Omega^{\diamond}}(\cdot)   \right\Vert_{\mathbb{H}^{1}(\Omega)} \, = \, \left\Vert G(x,\cdot)  \right\Vert_{\mathbb{H}^{1}(\Omega^{\diamond})}.
\end{equation}
Then, by plugging $(\ref{Equa1014IP})$ into $(\ref{InPa})$, we deduce 
\begin{equation*}
\left\Vert p^{f} \right\Vert_{\mathbb{L}^{2}(D)}  \,   \leq \, \left\Vert f \right\Vert_{\mathbb{H}^{-\frac{1}{2}}(\partial \Omega)} \, \left[  \int_{D} \left\Vert G(x,\cdot)   \right\Vert^{2}_{\mathbb{H}^{1}(\Omega^{\diamond})} \, dx \, \right]^{\frac{1}{2}}.
\end{equation*}
Recalling that 
\begin{equation*}
    \left\Vert G(x,\cdot)   \right\Vert^{2}_{\mathbb{H}^{1}(\Omega^{\diamond})} \, := \,     \left\Vert \nabla G(x,\cdot)   \right\Vert^{2}_{\mathbb{L}^{2}(\Omega^{\diamond})} \, + \,     \left\Vert G(x,\cdot)   \right\Vert^{2}_{\mathbb{L}^{2}(\Omega^{\diamond})}, 
\end{equation*}
we deduce, since $\vert \nabla G(x, y)\vert \, = \, \mathcal{O}\left(\vert x-y\vert^{-2}\right)$, that
\begin{equation}\label{ZMEqua0445}
\left\Vert p^{f} \right\Vert_{\mathbb{L}^{2}(D)}  \,    \lesssim  \, \left\Vert f \right\Vert_{\mathbb{H}^{-\frac{1}{2}}(\partial \Omega)} \, \left[ \int_{D} \left\Vert \nabla G(x,\cdot)   \right\Vert^{2}_{\mathbb{L}^{2}(\Omega^{\diamond})} \, dx \, \right]^{\frac{1}{2}} \, \lesssim  \, \left\Vert f \right\Vert_{\mathbb{H}^{-\frac{1}{2}}(\partial \Omega)} \, \left[ \int_{D} \int_{\Omega^{\diamond}} \frac{1}{\left\vert x - y \right\vert^{4}} \, dy  \, dx \, \right]^{\frac{1}{2}}.
\end{equation}
We have, for $x \in D$ and $y \in \Omega^{\diamond}$, see \textbf{Assumption \ref{AssDmDist}}, that
\begin{equation}\label{distxykappaa}
    \left\vert x - y \right\vert \, \geq \, \dist\left( D, \partial \Omega^{\diamond} \right) \, = \, \dist\left( D, \partial \Omega \right)\, \geq \, \kappa(a),
\end{equation}
which implies, 
\begin{eqnarray}\label{SH-Add-proof-1}
\nonumber
\left\Vert p^{f} \right\Vert_{\mathbb{L}^{2}(D)}   \, & \lesssim & \, \left\Vert f \right\Vert_{\mathbb{H}^{-\frac{1}{2}}(\partial \Omega)} \, \left[ \left\vert D \right\vert \, \left( \kappa(a) \right)^{-4}  \, \right]^{\frac{1}{2}} \\ &=& \, \left\Vert f \right\Vert_{\mathbb{H}^{-\frac{1}{2}}(\partial \Omega)} \, \left[ \left\vert D_{m_{0}} \right\vert \, M \, \left( \kappa(a) \right)^{-4}  \, \right]^{\frac{1}{2}} \, \overset{(\ref{distto})}{=} \, \mathcal{O}\left(\left\Vert f \right\Vert_{\mathbb{H}^{-\frac{1}{2}}(\partial \Omega)} \, a^{\frac{(2+7h)}{6}} \right).
\end{eqnarray}
Similarly, using $(\ref{pf=Slf})$, the vector function $\nabla p^{f}(\cdot)$ can be expressed as
\begin{equation}\label{ETVKT}
\nabla p^{f}(x) \, = \, \underset{x}{\nabla} \int_{\partial \Omega} G(x,y) \, f(y) \, d\sigma(y) \, = \, \langle \nabla G(x,\cdot) ; f \rangle_{\mathbb{H}^{\frac{1}{2}}(\partial \Omega) \times \mathbb{H}^{-\frac{1}{2}}(\partial \Omega)}, \quad x \in \Omega, 
\end{equation}
Then, by repeating the same computations as $(\ref{InPa})-(\ref{ZMEqua0445})$, we obtain
\begin{equation}\label{HS-Add-proof-1}
\left\Vert \nabla p^{f} \,  \right\Vert_{\mathbb{L}^{2}(D)} \; \lesssim \; \left\Vert  f \right\Vert_{\mathbb{H}^{-\frac{1}{2}}(\partial \Omega)} \; \left[ \int_{D} \int_{\Omega^{\diamond}} \frac{1}{\left\vert x - y \right\vert^{6}} \, dy \, dx  \right]^{\frac{1}{2}} \, = \, \mathcal{O}\left( a^{\frac{3 \, h}{2}} \; \left\Vert  f \right\Vert_{\mathbb{H}^{-\frac{1}{2}}(\partial \Omega)} \right). 
\end{equation}
Hence, by returning to $(\ref{EstJ1})$, using $(\ref{SH-Add-proof-1})$ and $(\ref{HS-Add-proof-1})$,  
\begin{equation}\label{ECEstJ1}
J_{1} \, =  \, \mathcal{O}\left( a^{\frac{3h}{2}-1} \; \left\Vert f \right\Vert_{\mathbb{H}^{-\frac{1}{2}}(\partial \Omega)} \; \left\Vert v^{g} \right\Vert_{\mathbb{L}^{2}(D)} \right).
\end{equation}
The following lemma gives us an a priori estimate satisfied by $v^{g}(\cdot)$. 
\begin{lemma}\label{ADZ-Lemma} 
We have the following a priori estimate 
\begin{equation}\label{AprioriEstvf}
\left\Vert v^{g} \right\Vert_{\mathbb{L}^{2}(D)}   \lesssim  a^{\frac{(5-2h)}{6}} \; \left\Vert g \right\Vert_{\mathbb{H}^{-\frac{1}{2}}(\partial \Omega)}. 
\end{equation}
\end{lemma}
\begin{proof}
See \textbf{Subsection \ref{Proof-A-Priori-Estimate}}.
\end{proof}
Thanks to the previous lemma, the estimation of $J_{1}$ given by $(\ref{ECEstJ1})$, we reduce the estimation of $\bm{I}_{1}$, given by $(\ref{I1=Dp+J1})$, to: 
\begin{equation}\label{I1-formula-Intro}
\bm{I}_{1}  =  \omega^{2} \, \dfrac{\rho_{1}}{k_{1}} \, \sum_{j=1}^{M} \, p^{f}(z_{j}) \, \int_{D_{j}}   v^{g}_{j}(x) \,  dx + \mathcal{O}\left( a^{\frac{(5h-2)}{3}} \; \left\Vert f \right\Vert_{\mathbb{H}^{-\frac{1}{2}}(\partial \Omega)} \; \left\Vert g \right\Vert_{\mathbb{H}^{-\frac{1}{2}}(\partial \Omega)} \right). 
\end{equation}
The goal of the following subsection is to derive the algebraic system satisfied by the vector \\ $ \left( \int_{D_{j}}   v^{g}_{j}(x) \,  dx \right)_{j=1,\cdots,M}$ and justify its invertibility.
\subsection{Algebraic system}
We start with the following L.S.E, with $v^{g}(\cdot)$ as the solution of $(\ref{Equavf})$,
\begin{equation}\label{LZEqua1157}
    v^{g}(x) \, - \, \omega^{2} \, \int_{D} G(x,y) \, v^{g}(y) \, \left( \frac{\rho_{1}}{k_{1}} \, - \, n^{2}(y) \right) \, dy \, = \, S(x), \quad x \in \Omega, 
\end{equation}
where $S(\cdot)$ is solution of 
\begin{align}\label{EquaSg}
\begin{cases}  
\Delta S + \omega^{2} \, n^{2}(\cdot)  \, S = 0 \quad \text{in} \quad  \Omega,  \\ 
\qquad \qquad \quad \,  \partial_{\nu} S = g \;\;\;  \text{on} \quad \partial \Omega. 
\end{cases}
\end{align}
and $G(\cdot,\cdot)$ is the Green's kernel solution of $(\ref{EquaGKernel})$.
Now, by restricting $(\ref{LZEqua1157})$ into $D$, we obtain 
\begin{equation}\label{L.S.E.vf}
v^{g}(x) - \omega^{2}  \, \int_{D} G(x,y) \, v^{g}(y) \, \left(\frac{\rho_{1}}{k_{1}} \, - \, n^{2}(y) \right) \, dy \, = \, S(x), \quad x \in D.
\end{equation}
The coming lemma, on the decomposition of the Green's kernel $G(\cdot;\cdot)$, is useful for the next step.
\begin{lemma}\label{LemmaG=phi+Remainder}
The Green's kernel $G(\cdot;\cdot)$, solution of $(\ref{EquaGKernel})$, admits the following decomposition: 
\begin{equation}\label{G=phi+Remainder}
G(x,y) \, = \, \Phi_{0}(x,y) \, + \, \mathcal{R}(x,y), \quad x \neq y,
\end{equation} 
where $\Phi_{0}(\cdot,\cdot)$ is given by $(\ref{DefPhi0Equa1100})$, and the remainder term  $\mathcal{R}(\cdot,\cdot)$ satisfies 
\begin{align}\label{RemainderPDE}
\begin{cases}  
\underset{x}{\Delta}\left( \mathcal{R}(x,y) \right) \, + \, \omega^{2} \, n^{2}(x)  \, \mathcal{R}(x,y) \, = \, - \, \omega^{2} \, n^{2}(x)  \, \Phi_{0}(x,y) \quad \text{in} \quad  \Omega,  \\ 
\qquad \qquad \qquad \quad \quad  \partial_{\nu_{x}} \left( \mathcal{R}(x,y) \right)  = \, - \, \partial_{\nu_{x}} \left( \Phi_{0}(x,y) \right) \;\;\;  \text{on} \quad \partial \Omega.
\end{cases}
\end{align}  
For $x, y \in \Omega$, the term  $\mathcal{R}(\cdot,\cdot)$ is such that\footnote{In general, we can prove that 
\begin{equation*}
    \left\vert \mathcal{R}(x,y) \right\vert \,   \lesssim  \, \left( \frac{1}{\dist(x,\partial \Omega)} \right)^{\mathfrak{q}} \, \left(   \frac{1}{\dist(y,\partial \Omega)} \right)^{\mathfrak{p}}, \quad x, y \in \Omega, 
\end{equation*}
where $\mathfrak{p}$ and $\mathfrak{q}$ are positive real numbers such that $\mathfrak{p} + \mathfrak{q} \, = \, 1$.} 
\begin{equation}\label{FphOOEqua0645}
    \left\vert \mathcal{R}(x,y) \right\vert \,   \lesssim  \, \left( \frac{1}{\dist(x,\partial \Omega)} \right)^{\frac{1}{3}} \, \left(   \frac{1}{\dist(y,\partial \Omega)} \right)^{\frac{2}{3}}.
\end{equation}
\end{lemma}
\begin{proof}
See \textbf{Subsection \ref{SubsectionProofLemma2.3}}.
\end{proof}
For $x \in D_{m}$, we rewrite $(\ref{L.S.E.vf})$ as 
\begin{eqnarray}\label{Equa1432}
\nonumber
\left( I - \omega^{2} \, \dfrac{\rho_{1}}{k_{1}} \, N_{D_{m}} \right)(v^{g}_{m})(x) &-& \omega^{2} \, \dfrac{\rho_{1}}{k_{1}} \, \sum_{j \neq m}^{M} \int_{D_{j}} G(x,y) \, v^{g}_{j}(y) \, dy \\  &=& S_{m}(x) + \omega^{2} \, \dfrac{\rho_{1}}{k_{1}} \, \int_{D_{m}} \mathcal{R}(x,y) \, v^{g}_{m}(y) \, dy \, -  \omega^{2} \, \int_{D} G(x,y) \, v^{g}(y) \, n^{2}(y) \, dy, 
\end{eqnarray}
where $S_{m}(\cdot) := S(\cdot)_{|_{D_{m}}}$, $\mathcal{R}(\cdot,\cdot)$ is solution of $(\ref{RemainderPDE})$ and $N_{D_{m}}(\cdot)$ is the Newtonian operator defined, from $\mathbb{L}^{2}(D_{m})$ to $\mathbb{L}^{2}(D_{m})$, by $(\ref{DefNPO})$. In both sides of $(\ref{Equa1432})$, successively, we multiply by $\dfrac{k_{1}}{\omega^{2} \, \rho_{1}}$, we take the inverse operator of $\left( \dfrac{k_{1}}{\omega^{2} \, \rho_{1}} \, I -  N_{D_{m}} \right)$ and integrate over $D_{m}$, the obtained equation, to get
\begin{eqnarray}\label{Equa1451}
\nonumber
\beta_{m} \, \int_{D_{m}} v_{m}^{g}(x) \; dx  &-& \sum_{j = 1 \atop j \neq m}^{M} \int_{D_{m}} W_{m}(x) \, \int_{D_{j}} G(x;y) \;  v_{j}^{g}(y) \; dy \; dx = \dfrac{k_{1}}{\omega^{2} \, \rho_{1}} \; \int_{D_{m}} W_{m}(x) \, S_{m}(x) \, dx \\ \nonumber
&+& \int_{D_{m}} W_{m}(x) \, \int_{D_{m}} \int_{0}^{1} \underset{y}{\nabla} \mathcal{R}(x,z_{m}+t(y-z_{m})) \cdot (y-z_{m}) \, dt \, v^{g}_{m}(y) \, dy \, dx \\
&-&  \frac{k_{1}}{\rho_{1}} \, \int_{D_{m}} W_{m}(x) \int_{D} G(x,y) \, v^{g}(y) \, n^{2}(y) \, dy \, dx,
\end{eqnarray}
where $W_{m}(\cdot)$ is solution of
\begin{equation}\label{EquaWm3D}
\dfrac{k_{1}}{\omega^{2} \, \rho_{1}} \; W_{m}(x) -  N_{D_{m}}\left( W_{m} \right)(x) \, = \, 1, \quad x \in D_{m},
\end{equation} 
and $\beta_{m} \in \mathbb{C}$ is the constant given by 
\begin{equation}\label{Def-Beta_m-AS}
\beta_{m} \, := \, \left(1 \, - \, \int_{D_{m}} W_{m}(x) \, \mathcal{R}(x,z_{m})\, dx  \right).
\end{equation}
We have,  
\begin{equation}\label{Est-Int-Wm-Rm}
\left\vert \int_{D_{m}} W_{m}(x) \, \mathcal{R}(x,z_{m})\, dx \right\vert \,  \leq  \, \left\Vert W_{m} \right\Vert_{\mathbb{L}^{2}(D_{m})} \, \left\Vert \mathcal{R}(\cdot,z_{m}) \right\Vert_{\mathbb{L}^{2}(D_{m})}  \underset{(\ref{Equa1020})}{\overset{(\ref{FphOOEqua0645})}{\lesssim}}  \, \frac{a^{(1-h)}}{\dist\left(D_{m};\partial \Omega \right)} \, \overset{(\ref{distto})}{=} \mathcal{O}\left( a^{\frac{2 \, (1-h)}{3}} \right).   
\end{equation}
Hence, 
\begin{equation}\label{DTBeta_m}
\beta_{m} \, = \, 1 \, + \, \mathcal{O}\left( a^{\frac{2 \, (1-h)}{3}} \right), \quad  \text{for} \quad m=1,\cdots,M.
\end{equation}
Next, to derive the desired algebraic system, we expand in the equation $(\ref{Equa1451})$ the Green's kernel $G(\cdot,\cdot)$ and the source term $S(\cdot)$, near the centers, to obtain:
\begin{equation}\label{ASEqua1510}
\beta_{m} \, \int_{D_{m}} v_{m}^{g}(x) \; dx \; -  \; \alpha_{m} \; \sum_{j = 1 \atop j \neq m}^{M} G(z_{m};z_{j}) \; \int_{D_{j}} v_{j}^{g}(x) \; dx = \dfrac{k_{1}}{\omega^{2} \, \rho_{1}} \; \alpha_{m} \; S(z_{m}) + Rest_{m},
\end{equation}
where $\alpha_{m}$ is the scattering coefficient given by 
\begin{equation}\label{Defalpha}
\alpha_{m} := \int_{D_{m}} W_{m}(x) \, dx,
\end{equation}
and 
\begin{eqnarray}\label{DefRestm}
\nonumber
Rest_{m} &:=&  \sum_{j = 1 \atop j \neq m}^{M} \int_{D_{m}} W_{m}(x)  \int_{0}^{1}   \nabla G(z_{m}+t(x-z_{m});z_{j}) \cdot (x-z_{m}) \; dt \; dx \, \int_{D_{j}} \, v_{j}^{g}(y) \, dy  \\ \nonumber 
&+&  \sum_{j = 1 \atop j \neq m}^{M} \int_{D_{m}} W_{m}(x) \int_{D_{j}} \int_{0}^{1}   \nabla G(x;z_{j}+t(y-z_{j})) \cdot (y-z_{j}) \; dt \, v_{j}^{g}(y) \, dy \; dx \\ \nonumber 
&+& \dfrac{k_{1}}{\omega^{2} \, \rho_{1}} \; \int_{D_{m}} W_{m}(x) \, \int_{0}^{1}   \nabla S_{m}(z_{m}+t(x-z_{m})) \cdot (x-z_{m}) dt \, dx  \\ \nonumber
&+&  \int_{D_{m}} W_{m}(x) \, \int_{D_{m}} \int_{0}^{1} \underset{y}{\nabla} \mathcal{R}(x,z_{m}+t(y-z_{m})) \cdot (y-z_{m}) \, dt \, v^{g}_{m}(y) \, dy \, dx \\
&-& \frac{k_{1}}{\rho_{1}} \, \int_{D_{m}} W_{m}(x) \int_{D} G(x,y) \, v^{g}(y) \, n^{2}(y) \, dy \, dx.
\end{eqnarray}
As $D_j$'s are translations and scales of the same domain $B$, i.e. $D_j=z_j +a B$ and $\rho_j=\rho_i$ with $k_j=k_i$ for $i, j:1, ..., M,$ we deduce that $\alpha_{m} = \alpha$, for $m=1,\cdots,M$. In addition, by multiplying its both sides by $\dfrac{\omega^{2} \, \rho_{1}}{k_{1} \, \alpha}$ and, then, setting \footnote{In the equation $\alpha \, = \, - P^{2} \, a^{1-h}$, the term $a^{1-h}$ comes from the estimation of $ \alpha $, see \textbf{Lemma \ref{EstimationRealphaImalpha}}.} 
\begin{equation}\label{DefYj}
Y_{m} := \dfrac{\beta_{m} \, \omega^{2} \, \rho_{1}}{\alpha \; k_{1}}  \; \int_{D_{m}} v_{m}^{g}(x) \; dx
\end{equation}
with $\alpha = \, - \, P^{2} \, a^{1-h}$ we obtain 
\begin{equation}\label{0820}
Y_{m} \, +  \, \sum_{j=1 \atop j \neq m}^{M} G(z_{m};z_{j}) \; P^{2} \, a^{1-h} \, \frac{1}{\beta_{j}} \, Y_{j} \, = \, S(z_{m})   + \dfrac{\omega^{2} \, \rho_{1}}{k_{1}}  \;  \frac{Rest_{m}}{\alpha}.
\end{equation}
The next lemma ensures the invertibility of the previous algebraic system. 
\begin{lemma}\label{MTR}
The algebraic system  $(\ref{0820})$ is invertible from $\ell_{2}$ to itself. In addition, the following estimation holds
\begin{equation*}
    \left( \sum_{m=1}^{M} \left\vert Y_{m} \right\vert^{2} \right)^{\frac{1}{2}} \; \lesssim  \; \left( \sum_{m=1}^{M} \left\vert S(z_{m}) \right\vert^{2} \right)^{\frac{1}{2}} \; + \; a^{(h-3)} \, \left( \sum_{m=1}^{M} \left\vert Rest_{m} \right\vert^{2} \right)^{\frac{1}{2}}.
\end{equation*}
In particular,
\begin{equation}\label{Est-ell2-Ym-g}
    \left( \sum_{m=1}^{M} \left\vert Y_{m} \right\vert^{2} \right)^{\frac{1}{2}} \; \lesssim  \; a^{\frac{2(h-1)}{3}} \; \left\Vert g \right\Vert_{\mathbb{H}^{-\frac{1}{2}}\left( \partial \Omega \right)}.
\end{equation}
\end{lemma}
\begin{proof}
See \textbf{Subsection \ref{SectionInjective}}.    
\end{proof}

\subsection{The L.S.E satisfied by \textbf{$u^{g}(\cdot)$}}\label{SubSection43} 
We define $\left(\tilde{Y}_{1}, \cdots, \tilde{Y}_{M} \right)$ as the solution of the unperturbed algebraic system related to $(\ref{0820})$. More precisely,
\begin{equation}\label{1924}
\tilde{Y}_{m} + \sum_{j = 1 \atop j \neq m}^{M} G(z_{m};z_{j}) \; P^{2} \, a^{1-h} \, \tilde{Y}_{j} \, \frac{1}{\beta_{j}}  =   S(z_{m}).
\end{equation}
We set the following L.S.E, 
\begin{equation}\label{NewL.S.E}
Y(z) \, + \, P^{2} \,  \int_{\Omega} G(z;y) \, Y(y) \, dy =  S(z), \quad z \in \Omega, 
\end{equation}
where $G(\cdot,\cdot)$ is solution of $(\ref{EquaGKernel})$ and $S(\cdot)$ is solution of $(\ref{EquaSg})$. We need the following lemma. 
\begin{lemma}\label{ExistenceYEstimation}
There exists one and only one solution $Y(\cdot)$ of the L.S.E $(\ref{NewL.S.E})$, and it satisfies the estimates
\begin{equation}\label{Lemma6.4YS}
\left\Vert Y \right\Vert_{\mathbb{H}^{1}(\Omega)} \, \lesssim  \, P^{2} \, \left\Vert S \right\Vert_{\mathbb{H}^{1}(\Omega)}. 
\end{equation}
\end{lemma}
\begin{proof}
The equation $(\ref{NewL.S.E})$ is invertible from $\mathbb{L}^{2}(\Omega)$ to $\mathbb{L}^{2}(\Omega)$\ and this gives us the estimation 
\begin{equation}\label{KS1}
 \left\Vert Y \right\Vert_{\mathbb{L}^{2}(\Omega)} \lesssim \left\Vert S \right\Vert_{\mathbb{L}^{2}(\Omega)} \le \left\Vert S \right\Vert_{\mathbb{H}^{1}(\Omega)}.   
\end{equation}
Now, by taking the $\mathbb{H}^{1}(\Omega)$-norm in both sides of $(\ref{NewL.S.E})$, we get: 
\begin{equation*}
    \left\Vert Y \right\Vert_{\mathbb{H}^{1}(\Omega)}  \lesssim \left\Vert S \right\Vert_{\mathbb{H}^{1}(\Omega)} \, +  \, P^{2} \, \left\Vert \mathscr{N}\left( Y \right) \right\Vert_{\mathbb{H}^{1}(\Omega)}, 
\end{equation*}
where $\mathscr{N}(\cdot)$ is the Newtonian operator defined by 
    \begin{equation}\label{DefNewG}
    \mathscr{N}(f)(x) \, := \,  \int_{\Omega} G(x,y) \, f(y) \, dy, \quad x \in \Omega. 
    \end{equation}
    Then, using the continuity of the Newtonian operator, from $\mathbb{L}^{2}(\Omega)$ to $\mathbb{H}^{1}(\Omega)$ , we obtain
\begin{equation*}
    \left\Vert Y \right\Vert_{\mathbb{H}^{1}(\Omega)}  \lesssim \left\Vert S \right\Vert_{\mathbb{H}^{1}(\Omega)} +  \, P^{2} \, \left\Vert Y \right\Vert_{\mathbb{L}^{2}(\Omega)} \overset{(\ref{KS1})}{\lesssim}  \, P^{2} \, \left\Vert S \right\Vert_{\mathbb{H}^{1}(\Omega)}. 
\end{equation*}
This ends the proof of \textbf{Lemma \ref{ExistenceYEstimation}}. 
\end{proof}
\begin{remark}
    The function $S(\cdot)$, solution of $(\ref{EquaSg})$, can be represented as a single layer potential with density function given by $g(\cdot)$, i.e.,  
    \begin{equation}\label{NMP1}
        S(x) \, = \,  \mathscr{S}(g)(x) \, := \, \int_{\partial \Omega} G(x,y) \, g(y) \, d\sigma(y), \quad x \in \Omega,
    \end{equation}
    with $G(\cdot,\cdot)$ is the Green's kernel solution of $(\ref{EquaGKernel})$.
    Then, from $(\ref{Lemma6.4YS})$, we obtain
\begin{equation*}
    \left\Vert Y \right\Vert_{\mathbb{H}^{1}(\Omega)} \lesssim \left\Vert \mathscr{S}(g) \right\Vert_{\mathbb{H}^{1}(\Omega)}, 
\end{equation*}
and using the continuity of the single layer operator, from $\mathbb{H}^{-\frac{1}{2}}(\partial \Omega)$ to $\mathbb{H}^{1}(\Omega)$, we end up with the following estimation: 
\begin{equation}\label{Yfctg}
    \left\Vert Y \right\Vert_{\mathbb{H}^{1}(\Omega)} \lesssim \left\Vert g \right\Vert_{\mathbb{H}^{-\frac{1}{2}}(\partial \Omega)}. 
\end{equation}
\end{remark}
\medskip
Based on the introduced notations in \textbf{Assumption \ref{AssDmDist}}, in particular $(\ref{Decoupage-Omega})$ and the fact that $\vert \Omega_j\vert=a^{1-h}$, for $1 \leq j \leq M$ and $0 \leq h < 1$, \footnote{We have $\vert \Omega_j\vert=a^{1-h}$ and $\vert \Omega^{\star}_j\vert\sim a^{1-h}$ but, as these $\Omega^{\star}_j$'s intersect $\partial \Omega$, we cannot necessarily replace $\sim$ with $=$.} we can rewrite $(\ref{NewL.S.E})$ as 
\begin{equation}\label{LSE=RHS+AI+AII}
Y(z_{m}) \, + \, P^{2}  \, \sum_{j = 1 \atop j \neq m}^{M} G(z_{m};z_{j})  \, a^{1-h}  \, \frac{1}{\beta_{j}} \, Y(z_{j}) =  S(z_{m}) \, - \, P^{2}  \, \bm{\tilde{I}}(z_{m}),
\end{equation}
where 
\begin{equation}\label{ASAB}
\bm{\tilde{I}}(z_{m}) \, := \,  \int_{\Omega} G(z_{m};y) \, Y(y) \, dy - \, \sum_{j = 1 \atop j \neq m}^{M} G(z_{m};z_{j})    \, Y(z_{j}) \, \, \frac{1}{\beta_{j}} 
 \, \left\vert \Omega_{j} \right\vert. 
\end{equation}
To estimate $\bm{\tilde{I}}(z_{m})$, we first consider the term 
\begin{equation}\label{Omega=Omega_Cube+Omega_r}
  \int_{\Omega} G(z_{m};y) \, Y(y) \, dy \, := \, \int_{\Omega_{cube}} G(z_{m};y) \, Y(y) \, dy+\int_{\Omega_r} G(z_{m};y) \, Y(y) \, dy.  
\end{equation}
and show that the second term is negligible.
In fact, for the domains $\Omega_{j}^{\star}$, which are not necessarily cubes, they have the property of non-empty intersection with $\partial \Omega$, i.e. $\partial \Omega_{n}^{\star} \cap \partial \Omega \neq \{ \emptyset \}$, for $1 \leq n \leq \aleph$. Since each $\Omega_{j}$ has volume equal to $a^{1-h}$, and then its maximum radius is of the order $a^{\frac{1}{3}(1-h)}$, then intersecting surfaces with $\partial \Omega$ have a volume of the order $a^{\frac{2}{3}(1-h)}$. As the volume of $\partial \Omega$ is of order one, we conclude that the number of such cubes will not exceed the order $a^{-\frac{2}{3}(1-h)}$. Hence, the volume of $\Omega_{r}$ will not exceed the order $a^{\frac{1}{3}(1-h)} \, \xrightarrow{a \rightarrow 0} \, 0$, i.e., 
\begin{equation}\label{VolOmegar}
    \left\vert \Omega_{r} \right\vert = \mathcal{O}\left( a^{\frac{1}{3}(1-h)} \right).
\end{equation}
Regarding the second term on the R.H.S of $(\ref{Omega=Omega_Cube+Omega_r})$, we have\footnote{We recall from \textbf{Lemma \ref{LemmaG=phi+Remainder}} that $G(z_{m};\cdot) \in \mathbb{L}^{3-\delta}(\Omega)$, with $z_{m}$ fixed, where $\delta$ is an arbitrarily and sufficiently small positive number.}
\begin{eqnarray*}
    \left\vert \int_{\Omega_r} G(z_{m};y) \, Y(y) \, dy \right\vert & \leq & \Vert Y \Vert_{\mathbb{L}^{2}(\Omega_{r})} \left\Vert G(z_{m};\cdot) \right\Vert_{\mathbb{L}^{2}(\Omega_{r})} \\ & \leq & \, \left\vert \Omega_{r} \right\vert^{\frac{1}{3}} \, \Vert Y \Vert_{\mathbb{L}^{6}(\Omega_{r})} \left\Vert G(z_{m};\cdot) \right\Vert_{\mathbb{L}^{2}(\Omega_{r})} \, \leq \, \left\vert \Omega_{r} \right\vert^{\frac{1}{3}} \, \Vert Y \Vert_{\mathbb{L}^{6}(\Omega)} \left\Vert G(z_{m};\cdot) \right\Vert_{\mathbb{L}^{2}(\Omega_{r})}.
\end{eqnarray*}
Thanks to \textbf{Lemma \ref{ExistenceYEstimation}}, we know that  $Y(\cdot) \in \mathbb{H}^{1}(\Omega) \subset \mathbb{L}^{6}(\Omega)$. Then, 
\begin{equation*}
        \left\vert \int_{\Omega_r} G(z_{m};y) \, Y(y) \, dy \right\vert   \lesssim  \left\vert \Omega_{r} \right\vert^{\frac{1}{3}} \,  \left\Vert Y \right\Vert_{\mathbb{H}^{1}(\Omega)} \, \left\Vert G(z_{m};\cdot) \right\Vert_{\mathbb{L}^{2}(\Omega_{r})} \, \overset{(\ref{Lemma6.4YS})}{\lesssim} \, P^{2} \, \left\vert \Omega_{r} \right\vert^{\frac{1}{3}} \,  \left\Vert S \right\Vert_{\mathbb{H}^{1}(\Omega)} \, \left\Vert G(z_{m};\cdot) \right\Vert_{\mathbb{L}^{2}(\Omega_{r})},
\end{equation*}
which, by using $(\ref{NMP1})$ and the continuity of the Single-Layer operator from $\mathbb{H}^{-\frac{1}{2}}(\partial \Omega)$ to $\mathbb{H}^{1}(\Omega)$, can be reduced to    
\begin{eqnarray*}
        \left\vert \int_{\Omega_r} G(z_{m};y) \, Y(y) \, dy \right\vert \, & \lesssim & \, P^{2} \, \left\vert \Omega_{r} \right\vert^{\frac{1}{3}} \,  \left\Vert g \right\Vert_{\mathbb{H}^{-\frac{1}{2}}(\partial \Omega)} \, \left\Vert G(z_{m};\cdot) \right\Vert_{\mathbb{L}^{2}(\Omega_{r})} \\ & \lesssim &  \, P^{2} \, \left\vert \Omega_{r} \right\vert^{\frac{(9-5\delta)}{6(3-\delta)}}   \, \left\Vert g \right\Vert_{\mathbb{H}^{-\frac{1}{2}}(\partial \Omega)}  \, \left\Vert G(z_{m};\cdot) \right\Vert_{\mathbb{L}^{3-\delta}(\Omega_{r})}.
\end{eqnarray*}
Next, using the fact that $G(z_{m};\cdot) \in \mathbb{L}^{3-\delta}(\Omega_{r})$, hence $\left\Vert G(z_{m};\cdot) \right\Vert_{\mathbb{L}^{3-\delta}(\Omega_{r})} \sim 1$, we reduce the previous estimation to
\begin{equation}\label{ASAB1258}
    \left\vert \int_{\Omega_r} G(z_{m};y) \, Y(y) \, dy \right\vert \, \lesssim \, P^{2} \, \left\vert \Omega_{r} \right\vert^{\frac{(9-5\delta)}{6(3-\delta)}}    \, \left\Vert g \right\Vert_{\mathbb{H}^{-\frac{1}{2}}(\partial \Omega)} \, \overset{(\ref{VolOmegar})}{=} \, \mathcal{O}\left(P^{2} \, \left\Vert g \right\Vert_{\mathbb{H}^{-\frac{1}{2}}(\partial \Omega)} \; a^{\frac{(1-h) \, (9 - 5 \delta)}{18(3-\delta)}} \right).
\end{equation}
Therefore, by gathering $(\ref{ASAB}),(\ref{Omega=Omega_Cube+Omega_r})$, and the estimation $(\ref{ASAB1258})$, we deduce
\begin{equation}\label{tildeI=I+err}
\bm{\tilde{I}}(z_{m}) \, = \, \bm{I}(z_{m}) \, + \, \mathcal{O}\left(P^{2} \,  \left\Vert g \right\Vert_{\mathbb{H}^{-\frac{1}{2}}(\partial \Omega)}  \; a^{\frac{(1-h) \, (9 - 5 \delta)}{18(3-\delta)}}  \right), \quad \text{as} \quad a \ll 1,    
\end{equation}
where 
\begin{equation*}
\bm{I}(z_{m}) \, := \, \sum_{\ell = 1}^{M} \int_{\Omega_{\ell}} G(z_{m};y) \, Y(y) \, dy - \sum_{j = 1 \atop j \neq m}^{M} G(z_{m};z_{j})  \, Y(z_{j}) \, \frac{1}{\beta_{j}}  \, \left\vert \Omega_{j} \right\vert.
\end{equation*}
Let us now estimate $\bm{I}(z_{m})$. We write
\begin{eqnarray*}
\bm{I}(z_{m}) \,
&=& \, \sum_{\ell = 1 \atop \ell \neq m}^{M} \int_{\Omega_{\ell}} \left[ G(z_{m};y)  Y(y) - G(z_{m};z_{\ell})  Y(z_{\ell}) \, \frac{1}{\beta_{\ell}}  \right] \, dy +  \int_{\Omega_{m}} G(z_{m};y) \,  Y(y) \, dy \\
&\overset{(\ref{Def-Beta_m-AS})}{=}& \, \sum_{\ell = 1 \atop \ell \neq m}^{M} \int_{\Omega_{\ell}} \left[ G(z_{m};y)  Y(y) - G(z_{m};z_{\ell})  Y(z_{\ell})  \right] \, dy +  \int_{\Omega_{m}} G(z_{m};y) \,  Y(y) \, dy \\
&-& \sum_{\ell = 1 \atop \ell \neq m}^{M}  G(z_{m};z_{\ell}) \, Y(z_{\ell}) \, \left\vert \Omega_{\ell} \right\vert \, \frac{\int_{D_{\ell}} W_{\ell}(x) \mathcal{R}(x,z_{\ell}) \, dx}{1 \, - \, \int_{D_{\ell}} W_{\ell}(x) \mathcal{R}(x,z_{\ell}) \, dx},
\end{eqnarray*}
and, by using Taylor expansion for the function $G(z_{m},\cdot) \, Y(\cdot)$ near the point $z_{\ell}$, we get: 
\begin{eqnarray}\label{EquaIzm}
\nonumber
\bm{I}(z_{m}) &=&  \sum_{\ell = 1 \atop \ell \neq m}^{M}   \int_{\Omega_{\ell}} \int_{0}^{1} G(z_{m};z_{\ell} + t (y-z_{\ell})) \nabla Y\left(z_{\ell} + t (y-z_{\ell}\right) \cdot (y-z_{\ell}) \, dt  \, dy \\  \nonumber
&+&\sum_{\ell = 1 \atop \ell \neq m}^{M}  \int_{\Omega_{\ell}} \int_{0}^{1} Y\left(z_{\ell} + t (y-z_{\ell}) \right) \nabla G(z_{m};z_{\ell} + t (y-z_{\ell})) \cdot (y-z_{\ell})  dt   dy + \int_{\Omega_{m}} G(z_{m};y)  Y(y) dy \\
&-& \sum_{\ell = 1 \atop \ell \neq m}^{M}  G(z_{m};z_{\ell}) \, Y(z_{\ell}) \, \left\vert \Omega_{\ell} \right\vert \, \frac{\int_{D_{\ell}} W_{\ell}(x) \mathcal{R}(x,z_{\ell}) \, dx}{1 \, - \, \int_{D_{\ell}} W_{\ell}(x) \mathcal{R}(x,z_{\ell}) \, dx}.
\end{eqnarray}
From \textbf{Lemma \ref{LemmaG=phi+Remainder}}, we know that $G(x,y) = \Phi_{0}(x,y) + \mathcal{R}(x,y)$, for $x \neq y$, where the dominant part $\Phi_{0}(\cdot,\cdot)$ is given by $(\ref{DefPhi0Equa1100})$. In the sequel, we keep only the dominant part of $G(\cdot,\cdot)$. More precisely, we have
\begin{equation}\label{SG}
\left\vert G(x,y) \right\vert \lesssim \frac{1}{\left\vert x - y \right\vert} \quad \text{and} \quad \left\vert \nabla G(x,y) \right\vert \lesssim \frac{1}{\left\vert x - y \right\vert^{2}}, \quad \text{for} \quad x \neq y.
\end{equation}
By taking the modulus in both sides of $(\ref{EquaIzm})$ and using $(\ref{SG})$, we deduce: 
\begin{eqnarray}\label{EquaIzm=...+Lzm}
\nonumber
\left\vert \bm{I}(z_{m}) \right\vert & \lesssim & a^{\frac{(1-h)}{3}}
\sum_{\ell = 1 \atop \ell \neq m}^{M}   \int_{\Omega_{\ell}} \int_{0}^{1} \frac{1}{ \left\vert z_{m} - z_{\ell} - t (y-z_{\ell}) \right\vert} \left\vert \nabla Y\left(z_{\ell} + t (y-z_{\ell}\right) \right\vert  dt  \, dy  \\ 
&+& a^{\frac{(1-h)}{3}} \sum_{\ell = 1 \atop \ell \neq m}^{M}  \int_{\Omega_{\ell}} \int_{0}^{1}  \frac{\left\vert Y\left(z_{\ell} + t (y-z_{\ell}) \right) \right\vert}{ \left\vert z_{m} - z_{\ell} - t (y-z_{\ell})\right\vert^{2}} \, dt  \, dy + \int_{\Omega_{m}} \left\vert G(z_{m};y) \, \right\vert \left\vert Y(y) \, \right\vert dy \, + \,  \left\vert \bm{L}(z_{m})\right\vert,
\end{eqnarray}
where $\bm{L}(z_{m})$ is the term given by
\begin{equation}\label{DefLzm}
    \bm{L}(z_{m}) \, :=  \, \sum_{\ell = 1 \atop \ell \neq m}^{M}  G(z_{m};z_{\ell}) \, Y(z_{\ell}) \, \left\vert \Omega_{\ell} \right\vert \, \frac{\int_{D_{\ell}} W_{\ell}(x) \mathcal{R}(x,z_{\ell}) \, dx}{1 \, - \, \int_{D_{\ell}} W_{\ell}(x) \mathcal{R}(x,z_{\ell}) \, dx}.
\end{equation}
 Next, we delay the estimation of $\bm{I}(z_{m})$ until we estimate first $\bm{L}(z_{m})$. To do this, by taking the absolute value on both sides of $(\ref{DefLzm})$, using the fact that $\left\vert \Omega_{\ell} \right\vert \, = \, a^{1-h}$, the estimation $(\ref{SG})$, with the estimation given by $(\ref{Est-Int-Wm-Rm})$, gives us 
 \begin{equation*}
         \left\vert \bm{L}(z_{m}) \right\vert \,  \lesssim  \, a^{\frac{5(1-h)}{3}} \, \sum_{\ell = 1 \atop \ell \neq m}^{M} \frac{1}{\left\vert z_{m} \, - \, z_{\ell} \right\vert}  \, \left\vert Y(z_{\ell}) \right\vert \, \leq \, a^{\frac{5(1-h)}{3}} \, \left( \sum_{\ell = 1 \atop \ell \neq m}^{M} \frac{1}{\left\vert z_{m} \, - \, z_{\ell} \right\vert^{2}} \right)^{\frac{1}{2}}  \, \left( \sum_{\ell = 1 \atop \ell \neq m}^{M} \left\vert Y(z_{\ell}) \right\vert^{2} \right)^{\frac{1}{2}}.
 \end{equation*}
 Besides, by applying a Cauchy-Schwartz inequality, we obtain 
\begin{eqnarray*}
 \left\vert \bm{L}(z_{m}) \right\vert \, & \leq & \, a^{\frac{5(1-h)}{3}} \, \left( \sum_{\ell = 1 \atop \ell \neq m}^{M} \frac{1}{\left\vert z_{m} \, - \, z_{\ell} \right\vert^{2}} \right)^{\frac{1}{2}} \, \left[  \left( \sum_{\ell = 1}^{M} \left\vert Y(z_{\ell}) - \tilde{Y}_{\ell} \right\vert^{2} \right)^{\frac{1}{2}} \, + \,  \left( \sum_{\ell = 1}^{M} \, \left\vert \tilde{Y}_{\ell} \right\vert^{2} \right)^{\frac{1}{2}}\right],
\end{eqnarray*}
where $\left(\tilde{Y}_{1}; \cdots ; \tilde{Y}_{M} \right)$ is solution of $(\ref{1924})$. In addition, by using $(\ref{Est-ell2-Ym-g})$, it can be reduced to 
\begin{equation}\label{EstLzmEqua0815}
    \left\vert \bm{L}(z_{m}) \right\vert \,  \lesssim  \,  \,  \left( \sum_{\ell = 1 \atop \ell \neq m}^{M} \frac{1}{\left\vert z_{m} \, - \, z_{\ell} \right\vert^{2}} \right)^{\frac{1}{2}} \left[ a^{\frac{5(1-h)}{3}} \, \left( \sum_{\ell = 1}^{M} \left\vert Y(z_{\ell}) - \tilde{Y}_{\ell} \right\vert^{2} \right)^{\frac{1}{2}} \, + \, a^{\frac{7(1-h)}{3}} \, \left\Vert g \right\Vert_{\mathbb{H}^{-\frac{1}{2}}\left( \partial \Omega \right)}\right].
\end{equation}
Now, by making again use of the Cauchy-Schwartz inequality in $(\ref{EquaIzm=...+Lzm})$ and using $(\ref{EstLzmEqua0815})$, we obtain:   
\begin{eqnarray*}
\left\vert \bm{I}(z_{m}) \right\vert  & \lesssim & a^{\frac{5(1-h)}{6}} \; \left\Vert \nabla Y  \right\Vert_{\mathbb{L}^{2}(\Omega)} \; \left(
\sum_{\ell = 1 \atop \ell \neq m}^{M}    \frac{1}{ \left\vert z_{m} - z_{\ell}  \right\vert^{2}} \right)^{\frac{1}{2}}    + a^{\frac{5(1-h)}{6}} \; \left\Vert Y  \right\Vert_{\mathbb{L}^{2}(\Omega)} \; \left(
\sum_{\ell = 1 \atop \ell \neq m}^{M}    \frac{1}{ \left\vert z_{m} - z_{\ell}  \right\vert^{4}} \right)^{\frac{1}{2}} \\ &+&  \left\Vert  Y  \right\Vert_{\mathbb{L}^{2}(\Omega)} \; \left( \int_{\Omega_{m}} \frac{1}{\left\vert z_{m} - y  \, \right\vert^{2}} \, dy \right)^{\frac{1}{2}}  +  a^{\frac{5(1-h)}{3}} \, \left( \sum_{\ell = 1 \atop \ell \neq m}^{M} \frac{1}{\left\vert z_{m} \, - \, z_{\ell} \right\vert^{2}} \right)^{\frac{1}{2}} \, \left( \sum_{\ell = 1}^{M} \left\vert Y(z_{\ell}) - \tilde{Y}_{\ell} \right\vert^{2} \right)^{\frac{1}{2}} \\ &+& \, a^{\frac{7(1-h)}{3}} \, \left( \sum_{\ell = 1 \atop \ell \neq m}^{M} \frac{1}{\left\vert z_{m} \, - \, z_{\ell} \right\vert^{2}} \right)^{\frac{1}{2}}\, \left\Vert g \right\Vert_{\mathbb{H}^{-\frac{1}{2}}\left( \partial \Omega \right)}.
\end{eqnarray*}
The use of the lemma below allows for the estimation of the discrete sum of the inverse power-weighted distance between the center of the droplets, and the discrete sum of the inverse power-weighted distance at the boundary.
\begin{lemma}\label{distLemma}
Let $\left\{ D_{m} \, = \, z_{m} \, + \, a \, B  \right\}_{m=1}^{M} \subset \Omega$. Then, we have the following estimates
\begin{enumerate}
    \item Inverse distance between the droplets centers,
\begin{equation}\label{SID}
    \sum_{j = 1 \atop j \neq m}^{M} \left\vert z_{m} - z_{j} \right\vert^{-k} = \begin{cases}
			\mathcal{O}\left(d^{-3}\right), & \text{for $k < 3$}\\
            & \\
            \mathcal{O}\left(d^{-k}\right), & \text{for $k > 3$}
		 \end{cases}.
\end{equation}
    \item[]
    \item Inverse distance to the boundary,
    \begin{equation}\label{SID+}
    \sum_{j = 1}^{M} \frac{1}{\dist^{k}\left(D_{j}; \partial \Omega \right)}  = \begin{cases}
			\mathcal{O}\left(d^{-3}\right), & \text{for $k < 3$}\\
            & \\
            \mathcal{O}\left(d^{-k}\right), & \text{for $k > 3$}
		 \end{cases}.
\end{equation}
\end{enumerate}
\end{lemma}
\begin{proof}
    see \textbf{Subsection \ref{ProofCounitingLemma}}.
\end{proof}
Thanks to $(\ref{SID})$ and the estimation $(\ref{dmin})$, we have 
\begin{eqnarray}\label{AI3RHS}
\nonumber
\left\vert \bm{I}(z_{m}) \right\vert  & \lesssim &  a^{\frac{5(1-h)}{6}}  \; \left\Vert \nabla Y  \right\Vert_{\mathbb{L}^{2}(\Omega)} \; d^{-\frac{3}{2}} + a^{\frac{5(1-h)}{6}} \; \left\Vert  Y  \right\Vert_{\mathbb{L}^{2}(\Omega)} \; d^{-2} + \left\Vert  Y  \right\Vert_{\mathbb{L}^{2}(\Omega)} \; \left( \int_{\Omega_{m}} \frac{1}{\left\vert z_{m} - y \right\vert^{2}}  dy \right)^{\frac{1}{2}} \\ &+&  a^{\frac{7(1-h)}{6}} \, \left( \sum_{\ell = 1}^{M} \left\vert Y(z_{\ell}) - \tilde{Y}_{\ell} \right\vert^{2} \right)^{\frac{1}{2}}  \, + \, a^{\frac{11(1-h)}{6}} \, \left\Vert g \right\Vert_{\mathbb{H}^{-\frac{1}{2}}\left( \partial \Omega \right)}.
\end{eqnarray}
To achieve the estimation of $\bm{I}(z_{m})$, we set and estimate the third term appearing on the R.H.S of the above equation
\begin{equation*}
\bm{I_3}(z_{m}) := \int_{B(z_{m};r)}  \frac{1}{\left\vert z_{m} - y \right\vert^{2}} \, dy + \int_{\Omega_{m} \setminus B(z_{m};r)}  \frac{1}{\left\vert z_{m} - y \right\vert^{2}} \,  dy,
\end{equation*}
where $B(z_{m};r)$ is the ball of center $z_{m}$ and radius $r$, where $r$ is such that $r \in \bm{I_4} := \left[ 0 ; \frac{\sqrt{3}}{2} \, a^{\frac{(1-h)}{3}} \right]$.
Then, 
\begin{eqnarray*}
\bm{I_3}(z_{m}) & \lesssim &  \int_{0}^{r} \int_{\partial \, B(z_{m};s)}  \frac{1}{\left\vert z_{m} - y \right\vert^{2}} \, d\sigma(y) \, ds +  \frac{1}{r^{2}} \; \left\vert \Omega_{m} \setminus B(z_{m};r) \right\vert  \\
& = &  \int_{0}^{r} \, \frac{1}{s^{2}} \, \left\vert \partial \, B(z_{m};s) \right\vert  \,  ds +  \frac{1}{r^{2}} \; \left\vert \Omega_{m} \setminus B(z_{m};r) \right\vert  =  \frac{8 \, \pi \, r}{3} +  \frac{1}{r^{2}} \; a^{1-h}  \leq   \; \underset{r \in \bm{I_4}}{\max} \, \rho(r,a),
\end{eqnarray*}
where 
\begin{equation*}
\rho(r,a) :=    \frac{8 \, \pi}{3} \, r +  \frac{1}{r^{2}} \; a^{1-h}.  
\end{equation*}
We have $\underset{r \in \bm{I_4}}{\max} \, \rho(r,a) = \rho(r_{sol},a)$, where $r_{sol}$ is such that $\partial_{r} \rho(r_{sol},a) = 0$. Straightforward computations, gives us $r_{sol} =  \left( \frac{3}{4 \, \pi} \, a^{1-h} \right)^{\frac{1}{3}}$. Consequently, $\;\; \underset{r \in \bm{I_4}}{\max} \, \rho(r,a) = \left( 48 \, \pi^{2} \right)^{\frac{1}{3}} \, a^{\frac{(1-h)}{3}}. $
Hence, 
\begin{equation}\label{EstimationAI3}
\bm{I_3}(z_{m})  =  \mathcal{O}\left( a^{\frac{(1-h)}{3}} \right). 
\end{equation}
Finally, by gathering $(\ref{EstimationAI3})$ and  $(\ref{AI3RHS})$ and using the fact that $d \sim a^{\frac{1-h}{3}}$, we end up with 
\begin{eqnarray*}
\left\vert \bm{I}(z_{m}) \right\vert & \lesssim & a^{\frac{(1-h)}{3}} \; \left\Vert \nabla Y \right\Vert_{\mathbb{L}^{2}(\Omega)} + a^{\frac{(1-h)}{6}} \; \left\Vert  Y \right\Vert_{\mathbb{L}^{2}(\Omega)} \\ &+& a^{\frac{7(1-h)}{6}} \, \left( \sum_{\ell = 1}^{M} \left\vert Y(z_{\ell}) - \tilde{Y}_{\ell} \right\vert^{2} \right)^{\frac{1}{2}}  \, + \, a^{\frac{11(1-h)}{6}} \, \left\Vert g \right\Vert_{\mathbb{H}^{-\frac{1}{2}}\left( \partial \Omega \right)} \\ & \overset{(\ref{Yfctg})}{\lesssim} & a^{\frac{(1-h)}{6}} \; \left\Vert g \right\Vert_{\mathbb{H}^{-\frac{1}{2}}(\partial \Omega)} \, + \, a^{\frac{7(1-h)}{6}} \, \left( \sum_{\ell = 1}^{M} \left\vert Y(z_{\ell}) - \tilde{Y}_{\ell} \right\vert^{2} \right)^{\frac{1}{2}}.
\end{eqnarray*}
Hence, plugging the above estimation into $(\ref{tildeI=I+err})$, we get
\begin{equation}\label{EstimationAI}
\bm{\tilde{I}}(z_{m}) = \mathcal{O}\left(P^{2} \,  a^{\frac{(9 - 5 \delta) \, (1-h)}{18(3-\delta)}} \; \left\Vert g \right\Vert_{\mathbb{H}^{-\frac{1}{2}}(\partial \Omega)} + \, a^{\frac{7(1-h)}{6}} \, \left( \sum_{\ell = 1}^{M} \left\vert Y(z_{\ell}) - \tilde{Y}_{\ell} \right\vert^{2} \right)^{\frac{1}{2}}\right).
\end{equation}
Finally, by going back to $(\ref{LSE=RHS+AI+AII})$ and making use of the estimation $(\ref{EstimationAI})$, we obtain 
\begin{eqnarray}\label{1918}
\nonumber
Y(z_{m}) -  \alpha  \, \sum_{j = 1 \atop j \neq m}^{M} G(z_{m};z_{j}) \, \frac{1}{\beta_{j}} \, Y(z_{j}) & = & S(z_{m}) + \mathcal{O}\left( P^{4} \,  a^{\frac{(9 - 5 \delta) \, (1-h)}{18(3-\delta)}} \; \left\Vert g \right\Vert_{\mathbb{H}^{-\frac{1}{2}}(\partial \Omega)} \right) \\
&+& \mathcal{O}\left(a^{\frac{7(1-h)}{6}} \, P^{2} \, \left( \sum_{\ell = 1}^{M} \left\vert Y(z_{\ell}) - \tilde{Y}_{\ell} \right\vert^{2} \right)^{\frac{1}{2}}\right).
\end{eqnarray}
Taking the difference between $(\ref{1924})$ and $(\ref{1918})$ gives us the following  algebraic system:
\begin{eqnarray*}
\left( \tilde{Y}_{m} - Y(z_{m}) \right) +  \sum_{j = 1 \atop j \neq m}^{M} G(z_{m};z_{j}) \; P^{2} \; a^{1-h} \, \frac{1}{\beta_{j}} \, \left( \tilde{Y}_{j} - Y(z_{j}) \right) &=& \mathcal{O}\left( P^{4} \,  a^{\frac{(9 - 5 \delta) \, (1-h)}{18(3-\delta)}} \; \left\Vert g \right\Vert_{\mathbb{H}^{-\frac{1}{2}}(\partial \Omega)} \right) \\
&+& \mathcal{O}\left(a^{\frac{7(1-h)}{6}} \, P^{2} \, \left( \sum_{\ell = 1}^{M} \left\vert Y(z_{\ell}) - \tilde{Y}_{\ell} \right\vert^{2} \right)^{\frac{1}{2}}\right).
\end{eqnarray*}
Consequently, using \textbf{Lemma \ref{MTR}}, we have
\begin{equation*}
\left( \sum_{m=1}^{M} \left\vert \tilde{Y}_{m} - Y(z_{m}) \right\vert^{2} \right)^{\frac{1}{2}} \, \lesssim \, P^{4} \, a^{\frac{-(1-h)(9 - 2\delta)}{9(3-\delta)}} \; \left\Vert  g \right\Vert_{\mathbb{H}^{-\frac{1}{2}}(\partial \Omega)}  \, + \, a^{\frac{2(1-h)}{3}} \, P^{2} \, \left( \sum_{\ell = 1}^{M} \left\vert Y(z_{\ell}) - \tilde{Y}_{\ell} \right\vert^{2} \right)^{\frac{1}{2}},
\end{equation*}
which, by knowing that $h<1$, can be reduced to
\begin{equation}\label{maxY-Y}
\left( \sum_{m=1}^{M} \left\vert \tilde{Y}_{m} - Y(z_{m}) \right\vert^{2} \right)^{\frac{1}{2}} \, = \, \mathcal{O}\left( P^{4} \, a^{\frac{-(1-h)(9 - 2\delta)}{9(3-\delta)}} \; \left\Vert  g \right\Vert_{\mathbb{H}^{-\frac{1}{2}}(\partial \Omega)} \right).
\end{equation}
The previous estimation confirm the convergence of the discrete algebraic system to the continuous L.S.E.
\subsection{Finishing the proof of Theorem \ref{principal-Thm}}
We set $\bm{J}$ to be: 
\begin{eqnarray}\label{DefJ}
\nonumber
\bm{J} &:=& \omega^{2} \, \frac{\rho_{1}}{k_{1}} \,  \langle v^{g}; p^{f} \rangle_{\mathbb{L}^{2}(D)} + P^{2} \, \langle u^{g}; p^{f} \rangle_{\mathbb{L}^{2}(\Omega)} \, - \, \omega^{2} \, \langle n^{2} \, v^{g}; p^{f} \rangle_{\mathbb{L}^{2}(D)} \\ & = & \sum_{j=1}^{M} \left[\omega^{2} \, \frac{\rho_{1}}{k_{1}} \,   \langle v^{g}_{j}; p^{f} \rangle_{\mathbb{L}^{2}(D_{j})} + P^{2} \,  \langle u^{g}_{j}; p^{f} \rangle_{\mathbb{L}^{2}(\Omega_{j})} \right] \, - \, \omega^{2} \, \langle n^{2} \, v^{g}; p^{f} \rangle_{\mathbb{L}^{2}(D)}. 
\end{eqnarray}
Then, by using the Taylor expansion for the function $p^{f}(\cdot)$ near the centers, we get: 
\begin{equation}\label{Equa0729}
\bm{J} =  \sum_{j=1}^{M} p^{f}(z_{j}) \, \left[ \omega^{2} \, \frac{\rho_{1}}{k_{1}}   \, \int_{D_{j}} v^{g}_{j}(x) \, dx + P^{2} \,  \int_{\Omega_{j}} u^{g}_{j}(x) \, dx \right] + \bm{Err_J},
\end{equation}
where 
\begin{eqnarray*}
\bm{Err_J} &:=& \omega^{2} \, \frac{\rho_{1}}{k_{1}} \,  \sum_{j=1}^{M} \int_{D_{j}} v^{g}_{j}(x) \, \int_{0}^{1}  \nabla p^{f}(z_{j} + t ( x - z_{j}) ) \cdot (x - z_{j})  \, dt \, dx \\ &+& P^{2} \, \sum_{j=1}^{M} \int_{\Omega_{j}} u^{g}_{j}(x)  \, \int_{0}^{1}  \nabla p^{f}(z_{j} + t ( x - z_{j}) ) \cdot (x - z_{j}) \, dt  \, dx \, - \,  \omega^{2} \, \langle n^{2} \, v^{g}; p^{f} \rangle_{\mathbb{L}^{2}(D)},  
\end{eqnarray*}
which can be estimated, by recalling that $\rho_{1} \sim 1; \, k_{1} \sim a^{2}$ and $\left\vert \Omega_{j} \right\vert \sim a^{1-h}$, as
\begin{eqnarray}\label{EMAS}
\nonumber
\left\vert \bm{Err_J} \right\vert 
& \lesssim & a^{-1} \, \sum_{j=1}^{M} \left\Vert v^{g}_{j} \right\Vert_{\mathbb{L}^{2}(D_{j})} \, \left\Vert  \nabla p^{f} \right\Vert_{\mathbb{L}^{2}(D_{j})} + a^{\frac{(1-h)}{3}} \, P^{2} \, \sum_{j=1}^{M} \left\Vert u^{g}_{j} \right\Vert_{\mathbb{L}^{2}(\Omega_{j})}  \, \left\Vert \nabla p^{f} \right\Vert_{\mathbb{L}^{2}(\Omega_{j})} \\ \nonumber &+& \, \left\Vert v^{g} \right\Vert_{\mathbb{L}^{2}(D)} \, \left\Vert p^{f} \right\Vert_{\mathbb{L}^{2}(D)} \\ \nonumber
& \lesssim & a^{-1} \, \left\Vert v^{g} \right\Vert_{\mathbb{L}^{2}(D)} \, \left\Vert  \nabla p^{f} \right\Vert_{\mathbb{L}^{2}(D)} + a^{\frac{(1-h)}{3}} \, P^{2} \, \left\Vert u^{g} \right\Vert_{\mathbb{L}^{2}(\Omega)}  \, \left\Vert \nabla p^{f} \right\Vert_{\mathbb{L}^{2}(\Omega)} + \left\Vert v^{g} \right\Vert_{\mathbb{L}^{2}(D)} \, \left\Vert p^{f} \right\Vert_{\mathbb{L}^{2}(D)} \\
& \overset{(\ref{HS-Add-proof-1})}{\underset{(\ref{SH-Add-proof-1})}{\lesssim}} & a^{\frac{(3h-2)}{2}} \, \left\Vert v^{g} \right\Vert_{\mathbb{L}^{2}(D)} \, \left\Vert f \right\Vert_{\mathbb{H}^{-\frac{1}{2}}(\partial \Omega)} + a^{\frac{(1-h)}{3}} \, P^{2} \, \left\Vert u^{g} \right\Vert_{\mathbb{L}^{2}(\Omega)}  \, \left\Vert \nabla p^{f} \right\Vert_{\mathbb{L}^{2}(\Omega)}.
\end{eqnarray}
Next, we estimate $\left\Vert \nabla p^{f} \right\Vert_{\mathbb{L}^{2}(\Omega)}$. To do this, we recall from $(\ref{ETVKT})$ that $\nabla p^{f}(x) = \nabla \mathscr{S}(f)(x)$, $x \in \Omega$,
where $\mathscr{S}(\cdot)$ is the Single-Layer operator defined by $(\ref{NMP1})$,
Hence, as $\mathscr{S}(\cdot) \, : \, \mathbb{H}^{-\frac{1}{2}}(\partial \Omega) \rightarrow \mathbb{H}^{1}(\Omega)$, we have 
\begin{equation}\label{nablapfOmega}
\left\Vert \nabla p^{f} \,  \right\Vert_{\mathbb{L}^{2}(\Omega)} = \mathcal{O}\left( \left\Vert  f \right\Vert_{\mathbb{H}^{-\frac{1}{2}}(\partial \Omega)} \right). 
\end{equation}
Then, by using $(\ref{AprioriEstvf})$ and $(\ref{nablapfOmega})$ into $(\ref{EMAS})$, we have
\begin{equation}\label{ASEqua1005}
\left\vert \bm{Err_J} \right\vert \,  \lesssim  \,  \left\Vert f \right\Vert_{\mathbb{H}^{-\frac{1}{2}}(\partial \Omega)} \, \left[ \left\Vert g \right\Vert_{\mathbb{H}^{-\frac{1}{2}}(\partial \Omega)} \, a^{\frac{(7h-1)}{6}}  + a^{\frac{(1-h)}{3}} \, P^{2} \, \left\Vert u^{g} \right\Vert_{\mathbb{L}^{2}(\Omega)} \right]. 
\end{equation}
Let us now estimate $u^{g}(\cdot)$ in terms of $g(\cdot)$. As $u^{g}(\cdot)$ is solution of $(\ref{EquaWf})$, then it satisfies the following integral equation 
\begin{equation}\label{SHEqua1001}
    u^{g}(\cdot) \, + \, P^{2} \, \mathscr{N}(u^{g})(\cdot) \, = \, \mathscr{S}(g)(\cdot), \quad \text{in} \;\; \Omega, 
\end{equation}
 where $\mathscr{N}(\cdot)$ is the Newtonian operator defined by $(\ref{DefNewG})$. Hence, by gathering $(\ref{NewL.S.E}), (\ref{SHEqua1001})$, and $(\ref{Yfctg})$, we deduce that 
    \begin{equation}\label{ASEqua1006}
        \left\Vert u^{g} \right\Vert_{\mathbb{L}^{2}(\Omega)} \, \lesssim \, \left\Vert g \right\Vert_{\mathbb{H}^{-\frac{1}{2}}(\partial \Omega)}. 
    \end{equation}
Then, by plugging $(\ref{ASEqua1006})$ into $(\ref{ASEqua1005})$, and using the fact that $h \, > \, \frac{1}{3}$, we deduce that    
\begin{equation*}
 \left\vert \bm{Err_J} \right\vert \, \lesssim  \, a^{\frac{(1-h)}{3}} \, P^{2} \,  \left\Vert g \right\Vert_{\mathbb{H}^{-\frac{1}{2}}(\partial \Omega)} \,    \left\Vert f \right\Vert_{\mathbb{H}^{-\frac{1}{2}}(\partial \Omega)}.
\end{equation*}
Going back to $(\ref{Equa0729})$, we obtain  
\begin{eqnarray*}
\bm{J} &=& \sum_{j=1}^{M} p^{f}(z_{j}) \, \left[ \omega^{2} \, \frac{\rho_{1}}{k_{1}} \, \int_{D_{j}} v^{g}_{j}(x) \, dx + P^{2} \,  \int_{\Omega_{j}} u^{g}_{j}(x) \, dx \right] \, + \, \mathcal{O}\left( a^{\frac{(1-h)}{3}} \, P^{2} \, \left\Vert f \right\Vert_{\mathbb{H}^{-\frac{1}{2}}(\partial \Omega)}  \, \left\Vert g \right\Vert_{\mathbb{H}^{-\frac{1}{2}}(\partial \Omega)} \right). 
\end{eqnarray*}
Remark that $u^{g}(\cdot)$, solution of $(\ref{EquaWf})$, is solution of the L.S.E given by $(\ref{NewL.S.E})$,  i.e., $u^{g}(\cdot) = Y(\cdot)$, in $\Omega$. Using this, we obtain
\begin{equation*}
\int_{\Omega_{j}} u^{g}_{j}(x) \, dx =  \int_{\Omega_{j}} Y(x) \, dx =  Y(z_{j}) \, \left\vert \Omega_{j} \right\vert +  \int_{\Omega_{j}} \int_{0}^{1} (x-z_{j}) \cdot \nabla Y(z_{j}+t(x-z_{j})) \, dt \, dx. 
\end{equation*}
Then, 
\begin{eqnarray}\label{J2T}
\nonumber
\bm{J} &=&  \sum_{j=1}^{M} p^{f}(z_{j}) \, \left[ \omega^{2} \, \frac{\rho_{1}}{k_{1}}  \, \int_{D_{j}} v^{f}_{j}(x) \, dx + P^{2} \, Y(z_{j}) \, \left\vert \Omega_{j} \right\vert  \right] \\ \nonumber
&+&  \sum_{j=1}^{M} p^{f}(z_{j}) \, \left[ P^{2} \,     \int_{\Omega_{j}} \int_{0}^{1} (x-z_{j}) \cdot \nabla Y(z_{j}+t(x-z_{j})) \, dt \, dx \right] \\ &+&  \mathcal{O}\left( a^{\frac{(1-h)}{3}} \, P^{2} \, \left\Vert f \right\Vert_{\mathbb{H}^{-\frac{1}{2}}(\partial \Omega)}  \, \left\Vert g \right\Vert_{\mathbb{H}^{-\frac{1}{2}}(\partial \Omega)} \right). 
\end{eqnarray}
We estimate the second term on the R.H.S, as 
\begin{eqnarray}\label{0539}
\nonumber
T_{2} &:=& \sum_{j=1}^{M} p^{f}(z_{j}) \, \left[ P^{2} \,     \int_{\Omega_{j}} \int_{0}^{1} (x-z_{j}) \cdot \nabla Y(z_{j}+t(x-z_{j})) \, dt \, dx \right] \\ \nonumber
\left\vert T_{2} \right\vert & \lesssim & P^{2}  \sum_{j=1}^{M} \left\vert p^{f}(z_{j}) \right\vert \, \left\vert      \int_{\Omega_{j}} \int_{0}^{1} (x-z_{j}) \cdot \nabla Y(z_{j}+t(x-z_{j})) \, dt \, dx \right\vert \\ \nonumber
& \leq & P^{2} \,  \left( \sum_{j=1}^{M} \left\vert p^{f}(z_{j}) \right\vert^{2} \right)^{\frac{1}{2}} \, \left( \sum_{j=1}^{M} \left\vert \int_{\Omega_{j}} \int_{0}^{1} (x-z_{j}) \cdot \nabla Y(z_{j}+t(x-z_{j})) \, dt \, dx \right\vert^{2} \right)^{\frac{1}{2}} \\
&=& \mathcal{O}\left( P^{2} \, \left( \sum_{j=1}^{M} \left\vert p^{f}(z_{j}) \right\vert^{2} \right)^{\frac{1}{2}} \,  a^{\frac{5}{6}(1-h)}  \, \left\Vert \nabla Y \right\Vert_{\mathbb{L}^{2}(\Omega)} \right).
\end{eqnarray}
At this stage, we need first to estimate $\overset{M}{\underset{j=1}{\sum}}  \left\vert p^{f}(z_{j}) \right\vert^{2}$. To achieve this, we recall that $p^{f}(\cdot)$ is solution of $(\ref{EquaKg-introdution})$ and we set $\tilde{p}^{f}(\cdot)$ to be solution of 
\begin{align}\label{Equatildepf}
\begin{cases}  
(\Delta +\omega^2)\, \tilde{p}^{f} = 0 \quad \text{in} \quad  \Omega,  \\ 
\partial_{\nu} \tilde{p}^{f} =  f \;\;\;  \text{on} \quad \partial \Omega. 
\end{cases}
\end{align}
Now, by subtracting $(\ref{EquaKg-introdution})$ from $(\ref{Equatildepf})$, we get 
\begin{align*}
\begin{cases}  
\left( \Delta + \omega^{2} \, n^{2} \right) \, \left( p^{f} - \tilde{p}^{f} \right) = \omega^{2} (- \,\, n^{2}+1) \, \tilde{p}^{f}  \quad \text{in} \quad  \Omega,  \\ 
\qquad \qquad \;\; \partial_{\nu} \left( p^{f} - \tilde{p}^{f} \right) =  0 \;\;\;  \text{on} \quad \partial \Omega. 
\end{cases}
\end{align*}
Its solution takes the following form 
\begin{equation*}
    \left( p^{f} - \tilde{p}^{f} \right)(z) =  \; \omega^{2}\int_{\Omega} G(z,y) \,  (-\, n^{2}+1)(y) \, \tilde{p}^{f}(y) \, dy, \quad z \in \Omega,
\end{equation*}
where $G(\cdot,\cdot)$ is the Green's kernel solution of $(\ref{EquaGKernel})$. By taking the modulus we get
\begin{equation}\label{pf-tildepf}
   \left\vert \left( p^{f} - \tilde{p}^{f} \right)(z) \right\vert \leq \; \left\Vert  \, \omega^{2}(- n^{2}+1) \right\Vert_{\mathbb{L}^{\infty}(\Omega)} \; \left\Vert G(z,\cdot) \right\Vert_{\mathbb{L}^{2}(\Omega)} \; \left\Vert \tilde{p}^{f} \right\Vert_{\mathbb{L}^{2}(\Omega)} \lesssim \left\Vert \tilde{p}^{f} \right\Vert_{\mathbb{L}^{2}(\Omega)},  
\end{equation}
where the last estimation is a consequence of the $\mathbb{L}^{2}(\Omega)$-integrability of the Green's kernel $G(z,\cdot)$, uniformly on $z \in \Omega$, and the boundedness of $\left\Vert n^{2} \right\Vert_{\mathbb{L}^{\infty}(\Omega)}$. In addition, we use the fact that $(\ref{Equatildepf})$ is a well posed problem to derive 
\begin{equation}\label{WPP}
\left\Vert \tilde{p}^{f} \right\Vert_{\mathbb{L}^{2}(\Omega)} \lesssim \left\Vert f \right\Vert_{\mathbb{H}^{-\frac{1}{2}}(\partial \Omega)}. 
\end{equation}
Then, by gathering $(\ref{pf-tildepf})$ and $(\ref{WPP})$, we obtain:
\begin{equation}\label{MMTT}
   \left\vert \left( p^{f} - \tilde{p}^{f} \right)(z) \right\vert  \lesssim \left\Vert f \right\Vert_{\mathbb{H}^{-\frac{1}{2}}(\partial \Omega)}.  
\end{equation}
As $\tilde{p}^{f}(\cdot)$ satisfies a Helmholtz equation in $\Omega$, see $(\ref{Equatildepf})$, then we have the following mean value integral formula
\begin{equation}\label{MVTHEqS}
  \frac{\sin\left(\omega \, r^{\prime} \right)}{\omega \, r^{\prime}} \,  \tilde{p}^{f}(z_{j}) \, = \, \frac{1}{\left\vert \partial \mathcal{B}_{j}(r^{\prime})  \right\vert} \; \int_{\partial \mathcal{B}_{j}(r^{\prime})} \tilde{p}^{f}(x) \; d\sigma(x), 
\end{equation}
where $\mathcal{B}_{j}(r^{\prime})$ is the  ball, centered at $z_{j}$ with radius $r^{\prime}$, contained in the cube $\Omega_{j}$. See \cite[Formula (36), page 288]{courant2024methods}. Now, by integrating both sides of $(\ref{MVTHEqS})$ with respect to the variable $r^{\prime}$, from $0$ to $r$, where $r$ is such that $\mathcal{B}_{j} \, := \, \mathcal{B}_{j}(r)$ is the largest ball, centered at $z_{j}$ with radius $r^{\prime}$, contained in the cube $\Omega_{j}$, we obtain
\begin{equation}\label{Equa0551}
    \tilde{p}^{f}(z_{j}) \, = \, \frac{\omega^{3}}{4 \, \pi \, \left(\sin\left(\omega \, r \right) \, - \, \omega \, r \, \cos(\omega \, r) \right)} \; \int_{\mathcal{B}_{j}} \tilde{p}^{f}(x) \; dx. 
\end{equation}
In addition, because $r$ is small the following approximation holds 
\begin{equation}\label{D-Equa0551}
    4 \, \pi \, \left(\sin\left(\omega \, r \right) \, - \, \omega \, r \, \cos(\omega \, r) \right) \, = \, \frac{4 \, \pi}{3} \, \omega^{3} \, r^{3} \, + \, \mathcal{O}\left( r^{5} \right) \, = \, \omega^{3} \, \left\vert \mathcal{B}_{j}\right\vert \, + \, \mathcal{O}\left( r^{5} \right).
\end{equation}
Then, by plugging $(\ref{D-Equa0551})$ into $(\ref{Equa0551})$, we obtain
\begin{equation*}
    \tilde{p}^{f}(z_{j}) \, = \, \frac{1}{\left\vert \mathcal{B}_{j}\right\vert \, + \, \mathcal{O}\left( r^{5} \right)} \; \int_{\mathcal{B}_{j}} \tilde{p}^{f}(x) \; dx. 
\end{equation*}
We observe that, for $1 \leq j \leq M$, we have $\left\vert \mathcal{B}_{j} \right\vert = \left\vert \mathcal{B}_{j_{0}} \right\vert \sim a^{1-h} \sim M^{-1}$.  Then, using the Cauchy-Schwartz inequality, in the above formula, we deduce that: 
\begin{equation}\label{LAM}
    \left\vert \tilde{p}^{f}(z_{j}) \right\vert \lesssim\left\vert \mathcal{B}_{j} \right\vert^{-\frac{1}{2}} \; \left\Vert \tilde{p}^{f} \right\Vert_{\mathbb{L}^{2}( \mathcal{B}_{j})}. 
\end{equation}
Therefore, 
\begin{equation*}
    \sum_{j=1}^{M} \left\vert p^{f}(z_{j}) \right\vert^{2} = \sum_{j=1}^{M} \left\vert \tilde{p}^{f}(z_{j}) + \left( p^{f}(z_{j}) -  \tilde{p}^{f}(z_{j}) \right) \right\vert^{2}   \lesssim  \sum_{j=1}^{M} \left\vert \tilde{p}^{f}(z_{j}) \right\vert^{2} + \sum_{j=1}^{M} \left\vert  p^{f}(z_{j}) -  \tilde{p}^{f}(z_{j})  \right\vert^{2}. 
\end{equation*}
By making use of $(\ref{LAM})$ and $(\ref{MMTT})$, we obtain: 
\begin{equation*}
    \sum_{j=1}^{M} \left\vert p^{f}(z_{j}) \right\vert^{2} \, \lesssim \, \sum_{j=1}^{M} \left\vert \mathcal{B}_{j} \right\vert^{-1} \; \left\Vert \tilde{p}^{f} \right\Vert^{2}_{\mathbb{L}^{2}( \mathcal{B}_{j})} + M \; \left\Vert f \right\Vert^{2}_{\mathbb{H}^{-\frac{1}{2}}(\partial \Omega)} \,
     \lesssim \, \left\vert \mathcal{B}_{j_{0}} \right\vert^{-1} \;  \left\Vert \tilde{p}^{f} \right\Vert^{2}_{\mathbb{L}^{2}\left( \overset{M}{\underset{j=1}{\cup}} \mathcal{B}_{j} \right)} + M \; \left\Vert f \right\Vert^{2}_{\mathbb{H}^{-\frac{1}{2}}(\partial \Omega)}. 
\end{equation*}
As $\left\vert \mathcal{B}_{j_{0}} \right\vert^{-1}  \sim M$ and $\overset{M}{\underset{j=1}{\cup}} \mathcal{B}_{j} \subset \Omega$, we obtain:
\begin{equation}\label{PfMf}
    \sum_{j=1}^{M} \left\vert p^{f}(z_{j}) \right\vert^{2}  \lesssim  M \;  \left(   \left\Vert \tilde{p}^{f} \right\Vert^{2}_{\mathbb{L}^{2}( \Omega )} +  \left\Vert f \right\Vert^{2}_{\mathbb{H}^{-\frac{1}{2}}(\partial \Omega)} \right) \overset{(\ref{WPP})}{\lesssim} M \;    \left\Vert f \right\Vert^{2}_{\mathbb{H}^{-\frac{1}{2}}(\partial \Omega)}. 
\end{equation}
We continue with our estimation of $(\ref{0539})$ by using $(\ref{PfMf})$ to get
\begin{equation*}
\left\vert T_{2} \right\vert  \lesssim P^{2} \,  M^{\frac{1}{2}} \, \left\Vert f \right\Vert_{\mathbb{H}^{-\frac{1}{2}}(\partial \Omega)} \, a^{\frac{5}{6}(1-h)}  \, \left\Vert \nabla Y \right\Vert_{\mathbb{L}^{2}(\Omega)} 
\overset{(\ref{Yfctg})}{=} \mathcal{O}\left(  P^{2} \, a^{\frac{(1-h)}{3}}  \, \left\Vert f \right\Vert_{\mathbb{H}^{-\frac{1}{2}}(\partial \Omega)} \, \left\Vert g \right\Vert_{\mathbb{H}^{-\frac{1}{2}}(\partial \Omega)} \right).
\end{equation*}
Then, using the above estimation of the term $\bm{J}$, given by $(\ref{J2T})$, becomes
\begin{equation}\label{HAJ}
\bm{J} = \sum_{j=1}^{M} p^{f}(z_{j}) \, \left[ \omega^{2} \, \frac{\rho_{1}}{k_{1}} \, \int_{D_{j}} v^{g}_{j}(x) \, dx + P^{2} \, Y(z_{j}) \, \left\vert \Omega_{j} \right\vert  \right] + \mathcal{O}\left( a^{\frac{(1-h)}{3}} \, P^{2} \, \left\Vert f \right\Vert_{\mathbb{H}^{-\frac{1}{2}}(\partial \Omega)}  \, \left\Vert g \right\Vert_{\mathbb{H}^{-\frac{1}{2}}(\partial \Omega)} \right). 
\end{equation}
To see how the parameter $P^{2}$ is related to the scattering coefficient $\alpha$, we set the following lemma.  
\begin{lemma}\label{EstimationRealphaImalpha}
The scattering coefficient $\alpha$, given by $(\ref{Defalpha})$, admits
the following estimation
\begin{equation}\label{EstimationRealpha}
\alpha  \, = \, - \; P^{2} \; a^{1-h}  + \mathcal{O}\left( a \right), 
\end{equation}
where $0 < h < 1$, and
\begin{equation*}
P^{2} := \frac{ - \, k_{0} \; \left( \langle 1; \overline{e}_{n_{0}} \rangle_{\mathbb{L}^{2}(B)} \right)^{2}}{\lambda_{n_{0}}^{B} \; c_{n_{0}}}. 
\end{equation*}
In addition, the following estimation holds
\begin{equation}\label{Equa1020}
    \left\Vert W_{m} \right\Vert_{\mathbb{L}^{2}(D_{m})} \,  \lesssim \, a^{-(2+h)} \, \left\Vert 1 \right\Vert_{\mathbb{L}^{2}(D_{m})},
\end{equation}
where $W_{m}(\cdot)$ is solution of $(\ref{EquaWm3D})$. 
\end{lemma}
\begin{proof}
See \textbf{Subsection \ref{ESC}}. 
\end{proof}
Knowing that $\left\vert \Omega_{j} \right\vert = a^{1-h}$ and using $(\ref{EstimationRealpha})$, we deduce that $P^{2} \, Y(z_{j}) \, \left\vert \Omega_{j} \right\vert = - \, \alpha \, Y(z_{j})$. In addition, from $(\ref{DefYj})$, we have
\begin{equation*}
 \frac{\omega^{2} \, \rho_{1} \, \beta_{j}}{k_{1}}  \, \int_{D_{j}} v_{j}^{g}(x) \; dx =  \alpha  \, Y_{j}.
\end{equation*}
Hence, the equation $(\ref{HAJ})$ becomes, 
\begin{eqnarray}\label{DTBN}
\nonumber
\bm{J} &=&  \sum_{j=1}^{M} p^{f}(z_{j}) \, \alpha  \, \left[ \frac{1}{\beta_{j}} \, Y_{j}  -  Y(z_{j}) \right] + \mathcal{O}\left( a^{\frac{(1-h)}{3}} \, P^{2} \, \left\Vert f \right\Vert_{\mathbb{H}^{-\frac{1}{2}}(\partial \Omega)}  \, \left\Vert g \right\Vert_{\mathbb{H}^{-\frac{1}{2}}(\partial \Omega)} \right) \\ \nonumber
\bm{J} &=& \sum_{j=1}^{M} p^{f}(z_{j}) \, \alpha  \, \left[  \tilde{Y}_{j}  -  Y(z_{j}) \right] \, + \, \sum_{j=1}^{M} p^{f}(z_{j}) \, \alpha  \, \left[ \frac{1}{\beta_{j}} \, Y_{j}  -  \tilde{Y}_{j} \right] \\ &+& \mathcal{O}\left( a^{\frac{(1-h)}{3}} \, P^{2} \, \left\Vert f \right\Vert_{\mathbb{H}^{-\frac{1}{2}}(\partial \Omega)}  \, \left\Vert g \right\Vert_{\mathbb{H}^{-\frac{1}{2}}(\partial \Omega)} \right).
\end{eqnarray}
Next, we set and estimate the second term on the R.H.S. To do this, we have
\begin{equation*}
    \bm{Q} \, := \, \sum_{j=1}^{M} p^{f}(z_{j}) \, \alpha  \, \left[ \frac{1}{\beta_{j}} \, Y_{j}  -  \tilde{Y}_{j} \right] \, \overset{(\ref{Def-Beta_m-AS})}{=} \, \sum_{j=1}^{M} p^{f}(z_{j}) \, \alpha  \, \left[  Y_{j}  -  \tilde{Y}_{j} \right] \, + \,\sum_{j=1}^{M} p^{f}(z_{j}) \, \alpha  \, \frac{\int_{D_{j}} W_{j}(x) \, \mathcal{R}(x,z_{j})\, dx}{\beta_{j}} \, Y_{j}.
\end{equation*}
Then, using $(\ref{EstimationRealpha}), (\ref{Est-Int-Wm-Rm})$ and $(\ref{DTBeta_m})$, we obtain 
\begin{eqnarray}\label{EstQ}
\nonumber
    \left\vert \bm{Q} \right\vert \, & \lesssim & \, P^{2} \, a^{1-h} \, \left( \sum_{j=1}^{M} \left\vert p^{f}(z_{j}) \right\vert^{2} \right)^{\frac{1}{2}}  \,  \left[ \left( \sum_{j=1}^{M} \left\vert  Y_{j}  -  \tilde{Y}_{j} \right\vert^{2} \right)^{\frac{1}{2}} \, +  \, a^{\frac{2(1-h)}{3}} \, \left( \sum_{j=1}^{M} \left\vert Y_{j} \right\vert^{2} \right)^{\frac{1}{2}} \right] \\ \nonumber
   & \overset{(\ref{PfMf})}{\lesssim} & \, P^{2} \, a^{1-h} \, M^{\frac{1}{2}} \;    \left\Vert f \right\Vert_{\mathbb{H}^{-\frac{1}{2}}(\partial \Omega)} \, \left[ \left( \sum_{j=1}^{M} \left\vert  Y_{j}  -  \tilde{Y}_{j} \right\vert^{2} \right)^{\frac{1}{2}} \, + \, a^{\frac{2(1-h)}{3}} \,  \, \left( \sum_{j=1}^{M} \left\vert Y_{j} \right\vert^{2} \right)^{\frac{1}{2}} \right] \\ 
   & \overset{(\ref{Est-ell2-Ym-g})}{\lesssim} & \, P^{2} \, a^{1-h} \, M^{\frac{1}{2}} \;    \left\Vert f \right\Vert_{\mathbb{H}^{-\frac{1}{2}}(\partial \Omega)} \, \left[ \left( \sum_{j=1}^{M} \left\vert  Y_{j}  -  \tilde{Y}_{j} \right\vert^{2} \right)^{\frac{1}{2}} \, + \, a^{\frac{4(1-h)}{3}} \; \left\Vert g \right\Vert_{\mathbb{H}^{-\frac{1}{2}}\left( \partial \Omega \right)} \right].
\end{eqnarray}
In addition, by subtracting $(\ref{0820})$ from $(\ref{1924})$, we derive the following algebraic system 
\begin{equation*}
\left( Y_{m} - \tilde{Y}_{m} \right) \, +  \, \sum_{j=1 \atop j \neq m}^{M} G(z_{m};z_{j}) \; P^{2} \, a^{1-h} \,  \frac{1}{\beta_{j}} \, \left(Y_{j} - \tilde{Y}_{j} \right)\, = \, \dfrac{\omega^{2} \, \rho_{1}}{k_{1}}  \;  \frac{Rest_{m}}{\alpha}.
\end{equation*}
Then, thanks to \textbf{Lemma \ref{MTR}}, the fact that $k_{1} \sim a^{2}$ and $\alpha \sim  a^{1-h}$, the following estimation holds,  
\begin{equation*}
    \left( \sum_{j=1}^{M} \left\vert  Y_{j}  -  \tilde{Y}_{j} \right\vert^{2} \right)^{\frac{1}{2}} \, \lesssim \, a^{h-3} \, \left( \sum_{j=1}^{M} \left\vert  Rest_{j} \right\vert^{2} \right)^{\frac{1}{2}} \, \overset{(\ref{ell2-norm-Rest})}{\lesssim} \, a^{3+h} \, \left\Vert g \right\Vert_{\mathbb{H}^{-\frac{1}{2}}(\partial \Omega)}.
\end{equation*}
Hence, by plugging the above estimation into $(\ref{EstQ})$ and using the fact that $M \sim a^{h-1}$, we obtain
\begin{equation*}
    \left\vert \bm{Q} \right\vert \,  \lesssim  \, P^{2} \,  a^{\frac{11(1-h)}{6}} \;   \left\Vert f \right\Vert_{\mathbb{H}^{-\frac{1}{2}}(\partial \Omega)} \,   \left\Vert g \right\Vert_{\mathbb{H}^{-\frac{1}{2}}\left( \partial \Omega \right)} .
\end{equation*}
Taking the modulus in both sides of $(\ref{DTBN})$, using the above estimation, we get
\begin{eqnarray*}
\left\vert \bm{J} \right\vert & \lesssim & \left\vert \alpha \right\vert \,  \left(\sum_{j=1}^{M} \left\vert \tilde{Y}_{j}  -  Y(z_{j}) \right\vert^{2} \right)^{\frac{1}{2}} \; \left( \sum_{j=1}^{M} \left\vert p^{f}(z_{j}) \right\vert^{2} \right)^{\frac{1}{2}} +  a^{\frac{(1-h)}{3}} \, P^{2} \, \left\Vert f \right\Vert_{\mathbb{H}^{-\frac{1}{2}}(\partial \Omega)}  \, \left\Vert g \right\Vert_{\mathbb{H}^{-\frac{1}{2}}(\partial \Omega)} \\
& \overset{(\ref{PfMf})}{\lesssim} & \left\vert \alpha \right\vert \, M^{\frac{1}{2}} \, \left(\sum_{j=1}^{M} \left\vert \tilde{Y}_{j}  -  Y(z_{j}) \right\vert^{2} \right)^{\frac{1}{2}} \; \left\Vert f \right\Vert_{\mathbb{H}^{-\frac{1}{2}}(\partial \Omega)} +  a^{\frac{(1-h)}{3}} \, P^{2} \, \left\Vert f \right\Vert_{\mathbb{H}^{-\frac{1}{2}}(\partial \Omega)}  \, \left\Vert g \right\Vert_{\mathbb{H}^{-\frac{1}{2}}(\partial \Omega)}.
\end{eqnarray*}
Noticing that $M \sim a^{h-1}$, using the fact that $ \alpha  \sim P^{2} \, a^{1-h}$, see  $(\ref{EstimationRealpha})$, and taking into account the estimation derived in $(\ref{maxY-Y})$, we obtain: 
\begin{eqnarray}\label{ZMEqua1030}
\nonumber
\left\vert \bm{J} \right\vert  & \lesssim &  a^{\frac{(1-h) (9 - 5 \delta)}{18(3-\delta)}} \, P^{6} \, \left\Vert f \right\Vert_{\mathbb{H}^{-\frac{1}{2}}(\partial \Omega)}  \, \left\Vert g \right\Vert_{\mathbb{H}^{-\frac{1}{2}}(\partial \Omega)} +  a^{\frac{(1-h)}{3}}  \, P^{2} \, \left\Vert f \right\Vert_{\mathbb{H}^{-\frac{1}{2}}(\partial \Omega)}  \, \left\Vert g \right\Vert_{\mathbb{H}^{-\frac{1}{2}}(\partial \Omega)} \\ 
& = & \qquad \mathcal{O}\left( a^{\frac{(1-h) (9 - 5 \delta)}{18 (3-\delta)}} \, P^{6} \, \left\Vert f \right\Vert_{\mathbb{H}^{-\frac{1}{2}}(\partial \Omega)}  \, \left\Vert g \right\Vert_{\mathbb{H}^{-\frac{1}{2}}(\partial \Omega)} \right).
\end{eqnarray}
Now, by gathering $(\ref{Lambda-d--Lambda-P}),(\ref{DefJ})$ and the estimation $(\ref{ZMEqua1030})$, we obtain:   
\begin{equation*}
\left\vert \langle \left( \Lambda_{D} - \Lambda_{P} \right)(f); g \rangle_{\mathbb{H}^{\frac{1}{2}}(\partial \Omega) \times \mathbb{H}^{-\frac{1}{2}}(\partial \Omega)} \right\vert \, = \, \left\vert \bm{J} \right\vert \, \overset{(\ref{ZMEqua1030})}{\lesssim} \, a^{\frac{(1-h) (9 - 5 \delta)}{18 (3-\delta)}} \, P^{6} \, \, \left\Vert f \right\Vert_{\mathbb{H}^{-\frac{1}{2}}(\partial \Omega)}  \, \left\Vert g \right\Vert_{\mathbb{H}^{-\frac{1}{2}}(\partial \Omega)}.
\end{equation*}
This suggest, 
\begin{equation*}
\left\Vert \Lambda_{D} - \Lambda_{P} \right\Vert_{\mathcal{L}(\mathbb{H}^{-\frac{1}{2}}(\partial \Omega);\mathbb{H}^{\frac{1}{2}}(\partial \Omega))} \lesssim a^{\frac{(1-h) (9 - 5 \delta)}{18 (3-\delta)}} \, P^{6} \,.
\end{equation*}
This proves $(\ref{energy-D})$ and ends the proof of \textbf{Theorem \ref{principal-Thm}}.
\addtocontents{toc}{\setcounter{tocdepth}{1}}
\section{Appendix - Proofs of auxiliary results}\label{Appendix}
This section is organized as follows. We start by proving \textbf{Lemma \ref{MTR}} related to the invertibility of the algebraic system $(\ref{0820})$. Then, we prove \textbf{Lemma \ref{LemmaNp}} related to the smallness of the Newtonian operator $N^{p}(\cdot)$ with respect to the parameter $P$. Next, it is important to first examine the proof of \textbf{Lemma \ref{LemmaG=phi+Remainder}}, on the analysis of the Green's kernel decomposition $G(\cdot;\cdot) = \Phi_{0}(\cdot;\cdot) \, + \, \mathcal{R}(\cdot;\cdot)$, before moving on to the proof of \textbf{Lemma \ref{ADZ-Lemma}}, giving us an a priori estimation satisfied by the acoustic field $v^{g}(\cdot)$.  
Later, we examine the proof of \textbf{Lemma \ref{EstimationRealphaImalpha}}, which gives us an estimation of the scattering coefficient $\alpha$.  Finally, we conclude this section by proving \textbf{Lemma \ref{Lemma51}}.
\subsection{Proof of Lemma \ref{MTR}}\label{SectionInjective}
This subsection's objective is to prove the invertibility of the algebraic system $(\ref{0820})$. To accomplish this, we link it to a continuous integral equation, for which we demonstrate its invertibility through variational formulation techniques. As a result, the algebraic system $(\ref{0820})$ can be inverted.
From $(\ref{0820})$, we have
\begin{equation*}\label{Equa1138}
Y_{m} \, + \, P^{2} \,  \sum_{j = 1 \atop j \neq m}^{M} G(z_{m};z_{j}) \; a^{1-h} \; \frac{1}{\beta_{j}} \, Y_{j} = S(z_{m})   + \dfrac{\omega^{2} \, \rho_{1}}{k_{1} \, \alpha}  \; Rest_{m}.
\end{equation*}
where $Y_{m}$ is defined by $(\ref{DefYj})$ and $Rest_{m}$ is given by $
(\ref{DefRestm})$. 
Then, by 
using the fact that $\left\vert \Omega_{j} \right\vert = a^{1-h}$, for $1 \leq j \leq M$, we 
rewrite the previous equation as 
\begin{eqnarray*}
Y_{m} \; + \; P^{2} \, \sum_{j = 1 \atop j \neq m}^{M} G(z_{m};z_{j}) \;  \left\vert \Omega_{j} \right\vert \; \frac{1}{\beta_{j}} \, Y_{j} \; &=&  S(z_{m}) + \dfrac{\omega^{2} \, \rho_{1}}{k_{1} \, \alpha}  \; Rest_{m} \\
Y_{m} \; + \;   P^{2} \; \sum_{j = 1 \atop j \neq m}^{M} \int_{\Omega} G(z_{m};z_{j}) \,  \chi_{\Omega_{j}}(x) \; \frac{1}{\beta_{j}} \,  Y_{j} \; dx &=& S(z_{m}) + \dfrac{\omega^{2} \, \rho_{1}}{k_{1} \, \alpha}  \; Rest_{m}.
\end{eqnarray*}
Multiplying the two sides of the previous equation with $\chi_{\Omega_{m}}(\cdot)$ and summing up with respect to the index $m$, we get:
\begin{eqnarray*}
\sum_{m=1}^{M} \chi_{\Omega_{m}}(\cdot) Y_{m} \; +  \, P^{2} \; \sum_{m=1}^{M} \, \chi_{\Omega_{m}}(\cdot) \sum_{j = 1 \atop j \neq m}^{M} \int_{\Omega} G(z_{m};z_{j}) \,  \chi_{\Omega_{j}}(x) \, \frac{1}{\beta_{j}} \; Y_{j} \; dx &=&  \sum_{m=1}^{M} \, \chi_{\Omega_{m}}(\cdot) S(z_{m}) \\
&+& \dfrac{\omega^{2} \, \rho_{1}}{k_{1} \, \alpha}  \; \sum_{m=1}^{M} \chi_{\Omega_{m}}(\cdot) Rest_{m}.
\end{eqnarray*}
which can be rewritten using the notations
\begin{equation}\label{DefYDefS}
    \boldsymbol{Y}(\cdot) := \sum_{m=1}^{M} \chi_{\Omega_{m}}(\cdot) Y_{m} \;\; \text{;} \;\; \boldsymbol{S}(\cdot) := \sum_{m=1}^{M} \chi_{\Omega_{m}}(\cdot) S(z_{m}) \;\; \text{and} \;\; \boldsymbol{R}(\cdot) := \sum_{m=1}^{M} \chi_{\Omega_{m}}(\cdot) Rest_{m},
\end{equation}
as
\begin{equation}\label{DHB}
\boldsymbol{Y}(\cdot) \; + \;  \, P^{2} \; \sum_{m=1}^{M} \, \chi_{\Omega_{m}}(\cdot) \sum_{j = 1 \atop j \neq m}^{M} \int_{\Omega} G(z_{m};z_{j}) \,  \chi_{\Omega_{j}}(x) \; \frac{1}{\beta_{j}} \,  Y_{j} \; dx = \boldsymbol{S}(\cdot) + \dfrac{\omega^{2} \, \rho_{1}}{k_{1} \, \alpha}  \; \boldsymbol{R}(\cdot). 
\end{equation}
The goal of the next lemma is to prove that the second term on the L.H.S converges, in $\mathbb{L}^{2}(\Omega)$, to a function which belongs to the range of the Newtonian operator $\mathscr{N}(\cdot)$, see $(\ref{DefNewG})$ for its definition. 
\begin{lemma}\label{Lemma51}
We set
\begin{equation}\label{DefT1}
    T_{1}(\cdot) \, := \, \mathscr{N}\left( \boldsymbol{Y} \right)(\cdot) - \sum_{m=1}^{M} \, \chi_{\Omega_{m}}(\cdot) \sum_{j = 1 \atop j \neq m}^{M} \int_{\Omega} G(z_{m};z_{j}) \,  \chi_{\Omega_{j}}(x) \; \frac{1}{\beta_{j}} \,  Y_{j} \; dx, \;\; \text{in} \;\; \Omega,
\end{equation}
where $\mathscr{N}(\cdot)$ is the Newtonian operator defined by $(\ref{DefNewG})$. Then, we have the following estimation
    \begin{equation}\label{LHSRHS}
      \left\Vert   T_{1}  \right\Vert_{\mathbb{L}^{2}(\Omega)} \lesssim a^{\frac{1}{6}(1-h)} \, \left\Vert \boldsymbol{Y} \right\Vert_{\mathbb{L}^{2}(\Omega)}.  
    \end{equation}
\end{lemma}
\begin{proof}
See \textbf{Subsection \ref{Prooflemma51}}. 
\end{proof}
Thanks to the previous lemma, we rewrite $(\ref{DHB})$ as    
\begin{equation}\label{IEI} 
\left( I  \; +  \,  P^{2} \, \mathscr{N} \right) \, \left( \boldsymbol{Y} \right) (\cdot) \, =  \, \boldsymbol{S}(\cdot) +  \boldsymbol{r}(\cdot), \quad \text{in} \; \Omega,
\end{equation}
where $\boldsymbol{S}(\cdot)$ is the function given by $(\ref{DefYDefS})$, and $\boldsymbol{r}(\cdot)$ is the function defined by 
\begin{equation}\label{Equa0543}
\boldsymbol{r}(\cdot) \, := \, \dfrac{\omega^{2} \, \rho_{1}}{k_{1} \, \alpha}  \; \boldsymbol{R}(\cdot)  +  \, P^{2} \, T_{1}(\cdot),
\end{equation}
with $\boldsymbol{R}(\cdot)$ as the function given by $(\ref{DefYDefS})$, and $T_{1}(\cdot)$ is the function defined by $(\ref{DefT1})$. 
Then, in the distributional sense, we have from $(\ref{IEI})$
\begin{equation}\label{MANEKKZM}
\left( \Delta \, + \, \omega^{2} \, n^{2} \, -  \, P^{2} \right)\left( \boldsymbol{Y} \right) \, = \, \left( \Delta \, + \, \omega^{2} \, n^{2}  \right)\left( \boldsymbol{S} \, + \, \boldsymbol{r} \right) \, =: \, \mathfrak{f}, \quad \text{in} \quad \Omega. 
\end{equation}
As by construction, see $(\ref{DefYDefS})$, we have $\boldsymbol{Y}(\cdot) \, = \, 0$ near $\partial \Omega$, then the equation $(\ref{MANEKKZM})$ can be stated in $\mathbb{R}^{3}$ by extending $\boldsymbol{Y}(\cdot)$ and $\mathfrak{f}(\cdot)$ by zero in $\mathbb{R}^{3} \setminus \Omega$. Keeping the same notations for $\boldsymbol{Y}(\cdot)$ and $\mathfrak{f}(\cdot)$ with their extensions to $\mathbb{R}^{3}$, we have    
\begin{equation*}
     \Delta \left( \boldsymbol{Y} \right) \, = \, \left( - \, \omega^{2} \, n^{2} \, +  \, P^{2} \right)\left( \boldsymbol{Y} \right) \, + \, \mathfrak{f}, \quad \text{in} \quad \mathbb{R}^{3}, 
\end{equation*}
with $\mathfrak{f} \in \mathbb{H}^{-2}_{comp}\left( \mathbb{R}^{3} \right)$. Therefore, 
\begin{equation}\label{CBEqua0410}
      \boldsymbol{Y}  \, + \, N_{\mathbb{R}^{3}}\left(\left(  P^{2} \, - \, \omega^{2} \, n^{2}  \right)\left( \boldsymbol{Y} \right)\right) \, = \, - \, N_{\mathbb{R}^{3}}\left(\mathfrak{f}\right), \quad \text{in} \quad \mathbb{L}^{2}\left(\mathbb{R}^{3}\right), 
\end{equation}
with $N_{\mathbb{R}^{3}}\left(\cdot\right)$ is the Newtonian operator defined by $(\ref{DefNPO})$. To study the existence and uniqueness of the solution corresponding to $(\ref{CBEqua0410})$, we start by multiplying its both sides by the function $\left( P^{2} \, - \, \omega^{2} \, n^{2}  \right) \, > \, 0$, for $P\gg 1$, to get 
\begin{equation}\label{CBEqua0410.}
      \left(  P^{2} \, - \, \omega^{2} \, n^{2}  \right) \, \boldsymbol{Y}  \, + \, \left(   P^{2} \, - \, \omega^{2} \, n^{2}  \right) \, N_{\mathbb{R}^{3}}\left(\left(  P^{2} \, - \, \omega^{2} \, n^{2}  \right)\left( \boldsymbol{Y} \right)\right) \, = \, - \, \left(  P^{2} \, - \, \omega^{2} \, n^{2}  \right) \, N_{\mathbb{R}^{3}}\left(\mathfrak{f}\right),
\end{equation}
in $\mathbb{L}^{2}\left(\mathbb{R}^{3}\right)$. Next, by taking the $\mathbb{L}^{2}(\mathbb{R}^{3})$-inner product on both sides of $(\ref{CBEqua0410.})$, we get 
\begin{equation}\label{gimel1=gimel2}
    \gimel_{1}\left(\boldsymbol{Y}; \boldsymbol{Z} \right) \, =  \,     \gimel_{2}\left(\boldsymbol{Z} \right),  
\end{equation}
where the bilinear form $    \gimel_{1}\left(\cdot; \cdot \right)$ is given by 
\begin{eqnarray*}
        \gimel_{1}\left(\boldsymbol{Y}; \boldsymbol{Z} \right) \, &:=& \, \langle \sqrt{P^{2} \, - \, \omega^{2} \, n^{2}} \boldsymbol{Y}; \sqrt{P^{2} \, - \, \omega^{2} \, n^{2}} \boldsymbol{Z} \rangle_{\mathbb{L}^{2}\left( \mathbb{R}^{3}\right)} \\ &+&
        \langle N_{\mathbb{R}^{3}}\left(\left(P^{2} \, - \, \omega^{2} \, n^{2} \right) \, \boldsymbol{Y} \right); \left( P^{2} \, - \, \omega^{2} \, n^{2} \right) \,  \boldsymbol{Z} \rangle_{\mathbb{L}^{2}\left( \mathbb{R}^{3}\right)},
\end{eqnarray*}
and the bilinear form $    \gimel_{2}\left(\cdot \right)$ is given by 
\begin{equation*}
     \gimel_{2}\left(\boldsymbol{Z} \right) \, := \, - \, 
        \langle N_{\mathbb{R}^{3}}\left( \mathfrak{f} \right); \left(P^{2} \, - \, \omega^{2} \, n^{2} \right) \,  \boldsymbol{Z} \rangle_{\mathbb{L}^{2}\left( \mathbb{R}^{3}\right)}.
\end{equation*}
We see that $\gimel_{1}\left(\cdot, \cdot \right)$ is continuous and admits the following estimation
\begin{equation*}
    \left\vert \gimel_{1}\left(\boldsymbol{Y}, \boldsymbol{Z} \right) \right\vert \, \leq \, \left\Vert \boldsymbol{Y} \right\Vert_{\mathbb{L}^{2}(\mathbb{R}^{3})} \, \left\Vert \boldsymbol{Z} \right\Vert_{\mathbb{L}^{2}(\mathbb{R}^{3})} \, \left\Vert  P^{2} - \omega^{2} \, n^{2} \right\Vert_{\mathbb{L}^{\infty}(\mathbb{R}^{3})} \, \left(1 \, + \left\Vert  P^{2} - \omega^{2} \, n^{2} \right\Vert_{\mathbb{L}^{\infty}(\mathbb{R}^{3})} \, \left\Vert N_{\mathbb{R}^{3}} \right\Vert_{\mathcal{L}}  \right).
\end{equation*}
Besides, $\gimel_{1}\left(\cdot, \cdot \right)$ is coercive satisfying 
\begin{eqnarray}\label{gimel1est}
\nonumber
    \gimel_{1}\left(\boldsymbol{Y}, \boldsymbol{Y} \right) \, &=& \, \left\Vert \sqrt{  P^{2} \, - \, \omega^{2} \, n^{2}} \, \boldsymbol{Y} \right\Vert^{2}_{\mathbb{L}^{2}(\mathbb{R}^{3})} \, + \, \langle N_{\mathbb{R}^{3}}\left( \left( P^{2} \, - \, \omega^{2} \, n^{2} \right) \, \boldsymbol{Y} \right); \left(P^{2} \, - \, \omega^{2} \, n^{2} \right) \,  \boldsymbol{Z} \rangle_{\mathbb{L}^{2}\left( \mathbb{R}^{3}\right)} \\
    & \geq & \left\Vert \sqrt{ P^{2} \, - \, \omega^{2} \, n^{2}} \, \boldsymbol{Y} \right\Vert^{2}_{\mathbb{L}^{2}(\mathbb{R}^{3})} \, \geq \, \underset{\mathbb{R}^{3}}{\text{Inf}}\left(P^{2} \, - \, \omega^{2} \, n^{2} \right) \, \left\Vert\boldsymbol{Y} \right\Vert^{2}_{\mathbb{L}^{2}(\mathbb{R}^{3})}, 
\end{eqnarray}
where the before last estimation is due to the positivity of the Newtonian operator. In addition, the linear form $\gimel_{2}\left( \cdot \right)$ is continuous and satisfy the following estimation, 
\begin{equation}\label{gimel2est}
   \left\vert \gimel_{2}\left( \boldsymbol{Z} \right) \right\vert \, \leq \, \left\Vert N_{\mathbb{R}^{3}}\left( \mathfrak{f} \right) \right\Vert_{\mathbb{L}^{2}(\mathbb{R}^{3})} \,  \left\Vert  P^{2} - \omega^{2} \, n^{2} \right\Vert_{\mathbb{L}^{\infty}(\mathbb{R}^{3})} \, \left\Vert \boldsymbol{Z} \right\Vert_{\mathbb{L}^{2}(\mathbb{R}^{3})}. 
\end{equation}
Hence, thanks to Lax-Milgram theorem, we deduce the existence and uniqueness of the solution corresponding to $(\ref{CBEqua0410})$. Furthermore, by gathering $(\ref{gimel1=gimel2}), (\ref{gimel1est})$ and $(\ref{gimel2est})$, we derive the following estimation
\begin{equation}\label{SAK}
    \underset{\mathbb{R}^{3}}{Inf}\left(   P^{2} \, - \, \omega^{2} \, n^{2} \right) \, \left\Vert\boldsymbol{Y} \right\Vert_{\mathbb{L}^{2}(\mathbb{R}^{3})} \, \leq \, \left\Vert N_{\mathbb{R}^{3}}\left( \mathfrak{f} \right) \right\Vert_{\mathbb{L}^{2}(\mathbb{R}^{3})} \,  \left\Vert   P^{2} - \omega^{2} \, n^{2} \right\Vert_{\mathbb{L}^{\infty}(\mathbb{R}^{3})}.
\end{equation}
Besides, as $P^{2} \, \gg \, 1$, we have  
\begin{equation*}
\underset{\mathbb{R}^{3}}{Inf}\left(P^{2} \, - \, \omega^{2} \, n^{2} \right) \, \sim \, P^{2} \quad \text{and} \quad \left\Vert P^{2} \, - \, \omega^{2} \, n^{2} \right\Vert_{\mathbb{L}^{\infty}(\mathbb{R}^{3})} \, \sim \, P^{2}, 
\end{equation*}
which, by plugging it into $(\ref{SAK})$, gives us 
\begin{equation*}
     \left\Vert\boldsymbol{Y} \right\Vert_{\mathbb{L}^{2}(\mathbb{R}^{3})} \, \leq \, \left\Vert N_{\mathbb{R}^{3}}\left( \mathfrak{f} \right) \right\Vert_{\mathbb{L}^{2}(\mathbb{R}^{3})} \, \lesssim \, \left\Vert  \mathfrak{f}  \right\Vert_{\mathbb{H}^{-2}(\mathbb{R}^{3})},
\end{equation*}
where we have used the continuity of the Newtonian operator to derive the last estimation. Knowing that $\boldsymbol{Y}(\cdot) \, = \, 0$, in $\mathbb{R}^{3} \setminus \overline{\Omega} $, and $\mathfrak{f}(\cdot) \, = \, 0$, in $\mathbb{R}^{3} \setminus \overline{\Omega} $, we deduce  
\begin{equation*}
     \left\Vert\boldsymbol{Y} \right\Vert_{\mathbb{L}^{2}(\Omega)} \,  \, \lesssim \, \left\Vert  \mathfrak{f}  \right\Vert_{\mathbb{H}^{-2}(\Omega)} \, \overset{(\ref{MANEKKZM})}{:=} \, \left\Vert \left(\Delta \, + \, \omega^{2} \, n^{2}  \right)\left( \boldsymbol{S}+\boldsymbol{r} \right) \right\Vert_{\mathbb{H}^{-2}(\Omega)}, 
\end{equation*}
which, by keeping the dominant part on the right hand side, can be reduced to  
\begin{equation*}
\left\Vert \boldsymbol{Y} \right\Vert_{\mathbb{L}^{2}(\Omega)}
        \, \lesssim  \, \left\Vert \Delta   \left( \boldsymbol{S}+\boldsymbol{r}  \right) \right\Vert_{\mathbb{H}^{-2}(\Omega)} \, \leq \, \left\Vert \Delta  \right\Vert_{\mathcal{L}\left(\mathbb{L}^{2}(\Omega);\mathbb{H}^{-2}(\Omega)\right)} \, \left\Vert  \boldsymbol{S}+\boldsymbol{r}   \right\Vert_{\mathbb{L}^{2}(\Omega)}.
\end{equation*}
Thus, the following inequalities hold, 
\begin{eqnarray*}
\left\Vert \boldsymbol{Y} \right\Vert_{\mathbb{L}^{2}(\Omega)}
        \, & \lesssim & \, \left\Vert  \boldsymbol{S}   \right\Vert_{\mathbb{L}^{2}(\Omega)} \, + \, \left\Vert \boldsymbol{r}   \right\Vert_{\mathbb{L}^{2}(\Omega)} \\ & \overset{(\ref{Equa0543})}{\lesssim} & \, \left\Vert  \boldsymbol{S}   \right\Vert_{\mathbb{L}^{2}(\Omega)} \, + \, \dfrac{1}{\left\vert k_{1} \, \alpha \right\vert} \, \left\Vert \boldsymbol{R}  \right\Vert_{\mathbb{L}^{2}(\Omega)} \, + \,  P^{2} \, \left\Vert T_{1} \right\Vert_{\mathbb{L}^{2}(\Omega)} \\
        &\overset{(\ref{LHSRHS})}{\lesssim}& \, \left\Vert  \boldsymbol{S}   \right\Vert_{\mathbb{L}^{2}(\Omega)} \, + \, \dfrac{1}{\left\vert k_{1} \, \alpha \right\vert} \, \left\Vert \boldsymbol{R}  \right\Vert_{\mathbb{L}^{2}(\Omega)} \, + \,  P^{2} \, a^{\frac{1}{6}(1-h)} \, \left\Vert \boldsymbol{Y} \right\Vert_{\mathbb{L}^{2}(\Omega)},
\end{eqnarray*}
which, under the fact that $P^{2} \, a^{\frac{1}{6}(1-h)} \, \ll \, 1$, as $a \ll 1$, which is satisfied because of (\ref{energy-D}) (or (\ref{M-P})), can be reduced to 
\begin{equation*}
    \left\Vert \boldsymbol{Y} \right\Vert_{\mathbb{L}^{2}(\Omega)} \, \lesssim \, \left\Vert  \boldsymbol{S}   \right\Vert_{\mathbb{L}^{2}(\Omega)} \, + \, \dfrac{1}{\left\vert k_{1} \, \alpha \right\vert} \, \left\Vert \boldsymbol{R}  \right\Vert_{\mathbb{L}^{2}(\Omega)}.
\end{equation*}
Now, using the fact that $k_{1} \sim a^{2}$, see $(\ref{ScaleBulk})$, and the estimation of $\alpha \sim a^{1-h}$, see $(\ref{EstimationRealpha})$, we deduce
\begin{equation*}
\left\Vert \boldsymbol{Y} \right\Vert_{\mathbb{L}^{2}(\Omega)}
        \,   \lesssim  \, \left\Vert  \boldsymbol{S}   \right\Vert_{\mathbb{L}^{2}(\Omega)} \, + \, a^{(h-3)} \, \left\Vert \boldsymbol{R}  \right\Vert_{\mathbb{L}^{2}(\Omega)} \, \overset{(\ref{SBMBR+})}{\lesssim}  \, \left\Vert  \boldsymbol{S}   \right\Vert_{\mathbb{L}^{2}(\Omega)} \, + \, a^{\frac{(1+h)}{2}} \, \left\Vert g \right\Vert_{\mathbb{H}^{-\frac{1}{2}}\left( \partial \Omega \right)}, 
\end{equation*}
which, by using $(\ref{DefYDefS})$ and $(\ref{NMP1})$, can be reduced to 
\begin{equation}\label{Eq1338}
      \left\Vert \boldsymbol{Y} \right\Vert_{\mathbb{L}^{2}(\Omega)} \, \lesssim  \, a^{\frac{(1-h)}{2}} \, \sqrt{\sum_{m=1}^{M}  \left\vert \mathscr{S}(g)(z_{m}) \right\vert^{2}} \, + \, a^{\frac{(1+h)}{2}} \, \left\Vert g \right\Vert_{\mathbb{H}^{-\frac{1}{2}}\left( \partial \Omega \right)}. 
\end{equation}
Let us estimate $\displaystyle \sum_{m=1}^{M}  \left\vert \mathscr{S}(g)(z_{m}) \right\vert^{2}$. To do this, we have 
\begin{eqnarray*}\label{Equa0810}
\nonumber
    \sum_{m=1}^{M} \left\vert \mathscr{S}(g)(z_{m}) \right\vert^{2} \, & \leq & \, \left\Vert g \right\Vert^{2}_{\mathbb{H}^{-\frac{1}{2}}(\partial \Omega)} \, \sum_{m=1}^{M} \left\Vert G(z_{m},\cdot) \right\Vert^{2}_{\mathbb{H}^{\frac{1}{2}}(\partial \Omega)} \\ \nonumber & \overset{(\ref{Equa1014IP})}{\leq} & 
    \, \left\Vert g \right\Vert^{2}_{\mathbb{H}^{-\frac{1}{2}}(\partial \Omega)} \, \sum_{m=1}^{M} \left\Vert G(z_{m},\cdot) \right\Vert^{2}_{\mathbb{H}^{1}(\Omega^{\diamond})} \\ \nonumber
    & \lesssim & 
    \, \left\Vert g \right\Vert^{2}_{\mathbb{H}^{-\frac{1}{2}}(\partial \Omega)} \, \sum_{m=1}^{M} \frac{1}{\dist^{4}(D_{m};\partial \Omega)}  \\ \nonumber
    & \overset{(\ref{SID+})}{\lesssim}  &  
    \, \left\Vert g \right\Vert^{2}_{\mathbb{H}^{-\frac{1}{2}}(\partial \Omega)} \, d^{-4} \\
    &\overset{(\ref{dmin})}{=}& \mathcal{O}\left( \left\Vert g \right\Vert^{2}_{\mathbb{H}^{-\frac{1}{2}}(\partial \Omega)} \, a^{\frac{-4(1-h)}{3}} \right).
\end{eqnarray*}
Then, plugging the above estimation into $(\ref{Eq1338})$,
\begin{equation*}
\left\Vert \boldsymbol{Y} \right\Vert_{\mathbb{L}^{2}(\Omega)}
        \,   \lesssim   \, a^{\frac{(h-1)}{6}} \, \left\Vert g \right\Vert_{\mathbb{H}^{-\frac{1}{2}}\left( \partial \Omega \right)}. 
\end{equation*}
In addition, as by construction, 
\begin{equation}\label{Equa0714}
\Vert \boldsymbol{Y} \Vert^2_{\mathbb{L} ^{2}(\Omega)} = \sum_{m=1}^{M} \vert Y_m\vert^2\vert \Omega_m\vert = \left\vert \Omega_{m_{0}} \right\vert \; \sum_{m=1}^{M} \left\vert Y_{m} \right\vert^{2},     
\end{equation}
we deduce the estimate
\begin{equation}\label{Equa0454}
    \left( \sum_{m=1}^{M} \left\vert Y_{m} \right\vert^{2} \right)^{\frac{1}{2}} \; \lesssim  \; a^{\frac{2(h-1)}{3}} \; \left\Vert g \right\Vert_{\mathbb{H}^{-\frac{1}{2}}\left( \partial \Omega \right)}.
\end{equation} 
This implies the injectivity of $(\ref{IEI})$. In addition, it is known that any injective linear map between two finite dimensional vector spaces of the same dimension is surjective. This proves the surjectivity and, consequently, the bijectivity of $(\ref{IEI})$. Hence, we have also the invertibility of the algebraic system $(\ref{0820})$.
This concludes the proof of \textbf{Lemma \ref{MTR}}. 
\subsection{Proof of Lemma \ref{LemmaNp}}\label{AS2331}
From the spectral theory, we have
\begin{equation*}
\left\Vert N^{p} \right\Vert_{\mathcal{L}\left(\mathbb{L}^{2}(\Omega);\mathbb{L}^{2}(\Omega)\right)} = \left\Vert \mathcal{R}(P^{2};\Delta) \right\Vert_{\mathcal{L}\left(\mathbb{L}^{2}(\Omega);\mathbb{L}^{2}(\Omega)\right)} \leq \frac{1}{\dist\left(P^{2};\sigma(\Delta) \right)},
\end{equation*} 
where $\sigma(\Delta)$ stands for the spectrum of the Neumann-Laplacian operator in $\mathbb{L}^{2}(\Omega)$. It is known that $\sigma(\Delta) := \left\{ \mu_{n} \right\}_{n \geq 1}$ such that $ 0 = \mu_{1} > \mu_{2} > \mu_{3} > \cdots \rightarrow - \infty$. Hence, we get $\dist\left(P^{2};\sigma(\Delta) \right) \, = \, P^{2}$. Consequently, 
\begin{equation*}
\left\Vert N^{p} \right\Vert_{\mathcal{L}\left(\mathbb{L}^{2}(\Omega);\mathbb{L}^{2}(\Omega)\right)}  \leq \frac{1}{P^{2}}.
\end{equation*}
This proves $(\ref{NormNewtonian})$. To prove $(\ref{TraceNormNewtonian})$, we start by remarking that 
for an arbitrary function $f \in \mathbb{L}^{2}(\Omega)$, the function $N^{p}(f)$ satisfies the problem 
\begin{align*}
\begin{cases}  
\left(\Delta  - P^{2} \, I \right) N^{p}(f)  = - \, f \quad \text{in} \quad  \Omega,  \\ 
\qquad \quad \quad \partial_{\nu} N^{p}(f) = \quad 0 \;\;\; \,  \text{on} \quad \partial \Omega. 
\end{cases}
\end{align*}
Multiplying both sides of the first equation by $N^{p}(f)$ and integrating in $\Omega$, we get
\begin{eqnarray*}
\left\Vert \nabla N^{p}(f) \right\Vert^{2}_{\mathbb{L}^{2}(\Omega)} & \leq & P^{2} \, \left\Vert N^{p}(f) \right\Vert^{2}_{\mathbb{L}^{2}(\Omega)} + \left\Vert f \right\Vert_{\mathbb{L}^{2}(\Omega)} \; \left\Vert N^{p}(f) \right\Vert_{\mathbb{L}^{2}(\Omega)} \\
& \leq & P^{2} \, \left\Vert N^{p} \right\Vert^{2}_{\mathcal{L}\left(\mathbb{L}^{2}(\Omega);\mathbb{L}^{2}(\Omega)\right)} \, \left\Vert f \right\Vert^{2}_{\mathbb{L}^{2}(\Omega)} + \left\Vert f \right\Vert^{2}_{\mathbb{L}^{2}(\Omega)} \; \left\Vert N^{p} \right\Vert_{\mathcal{L}\left(\mathbb{L}^{2}(\Omega);\mathbb{L}^{2}(\Omega)\right)}.
\end{eqnarray*}
Hence,
\begin{equation}\label{EstimationnablaN}
\left\Vert \nabla N^{p} \right\Vert^{2}_{\mathcal{L}\left(\mathbb{L}^{2}(\Omega);\mathbb{L}^{2}(\Omega)\right)} 
 \leq  P^{2} \, \left\Vert N^{p} \right\Vert^{2}_{\mathcal{L}\left(\mathbb{L}^{2}(\Omega);\mathbb{L}^{2}(\Omega)\right)}  +  \left\Vert N^{p} \right\Vert_{\mathcal{L}\left(\mathbb{L}^{2}(\Omega);\mathbb{L}^{2}(\Omega)\right)} \overset{(\ref{NormNewtonian})}{=} \mathcal{O}\left( \frac{1}{P^{2}} \right).
\end{equation}
Then, 
\begin{equation*}
\left\Vert N^{p} \right\Vert_{\mathcal{L}\left(\mathbb{L}^{2}(\Omega);\mathbb{H}^{1}(\Omega)\right)} := \left[ \left\Vert N^{p} \right\Vert^{2}_{\mathcal{L}\left(\mathbb{L}^{2}(\Omega);\mathbb{L}^{2}(\Omega)\right)} + \left\Vert \nabla N^{p} \right\Vert^{2}_{\mathcal{L}\left(\mathbb{L}^{2}(\Omega);\mathbb{L}^{2}(\Omega)\right)} \right]^{\frac{1}{2}}, 
\end{equation*}
which, using $(\ref{NormNewtonian})$ and $(\ref{EstimationnablaN})$, becomes $\left\Vert N^{p} \right\Vert_{\mathcal{L}\left(\mathbb{L}^{2}(\Omega);\mathbb{H}^{1}(\Omega)\right)} = \mathcal{O}\left( \dfrac{1}{P} \right)$ 
and, by taking the trace operator we end up with the following estimation
\begin{equation*}
\left\Vert \gamma N^{p} \right\Vert_{\mathcal{L}\left(\mathbb{L}^{2}(\Omega); \mathbb{H}^{\frac{1}{2}}(\partial \Omega) \right)} = \mathcal{O}\left( \frac{1}{P} \right).
\end{equation*}
This proves $(\ref{TraceNormNewtonian})$ and ends the proof of \textbf{Lemma \ref{LemmaNp}}. 
\subsection{Proof of Lemma \ref{LemmaG=phi+Remainder}}\label{SubsectionProofLemma2.3}
Multiplying both sides of $(\ref{RemainderPDE})$ by $\Phi_{0}(\cdot,\cdot)$, solution of $(\ref{LREq1110})$, integrating by parts over the domain $\Omega$, and using the fact that 
\begin{equation*}
\partial_{\nu_{x}}\left( \mathcal{R}(x,y) \right) \, = \, - \, \partial_{\nu_{x}} \left( \Phi_{0}(x,y)\right), \qquad x \in \partial \Omega \quad \text{and} \quad y \in \Omega, 
\end{equation*}
we obtain 
\begin{eqnarray}\label{SHEqua1127}
\nonumber
    \mathcal{R}(x,y) \, &=& \, - \, DL_{\partial \Omega}\left( \mathcal{R}(\cdot,y) \right)(x) \, + \, \omega^{2} \, N_{\Omega}\left(n^{2}(\cdot) \, \mathcal{R}(\cdot,y) \right)(x) \\
    &+& \omega^{2} \, N_{\Omega}(n^{2}(\cdot) \, \Phi_{0}(\cdot,y))(x) \, - \, SL_{\partial \Omega}\left( \partial_{\nu} \left( \Phi_{0}(\cdot,y)\right) \right)(x), \quad x \in \Omega, 
\end{eqnarray}
where $y \in \Omega$ is taken as a parameter, $N_{\Omega}(\cdot)$ is the Newtonian operator defined by $(\ref{DefNPO})$, $SL_{\partial \Omega}(\cdot)$ is the Single-Layer operator defined by
\begin{equation*}
    SL_{\partial \Omega}\left(f\right)(x) \, := \, \int_{\partial \Omega} \Phi_{0}(x,y) \, f(y) \, d\sigma(y), \quad x \in \Omega,  
\end{equation*}
and $DL_{\partial \Omega}(\cdot)$ is the Double-Layer operator defined by 
\begin{equation*}
    DL_{\partial \Omega}\left(f\right)(x) \, := \, \int_{\partial \Omega} \frac{\partial \Phi_{0}(x,y)}{\partial \nu(y)} \, f(y) \, d\sigma(y), \quad x \in \Omega,
\end{equation*}
Besides, thanks to \cite[Proposition 4.3]{MitreaTaylor}, we have the following singularity analysis 
\begin{equation*}
    \left\vert G(x,y) \right\vert \, \lesssim \, \frac{1}{\left\vert x \, - \, y \right\vert}, \quad x \neq y, 
\end{equation*}
which by plugging it into the right hand side of $(\ref{SHEqua1127})$, and up to an additive uniformly bounded part, gives us 
\begin{equation}\label{SHEqua1127red}
\nonumber
    \mathcal{R}(x,y) \, \simeq \, - \, SL_{\partial \Omega}\left( \partial_{\nu} \left( \Phi_{0}(\cdot,y)\right) \right)(x), \quad x \in \Omega. 
\end{equation}
As near the boundary $\partial \Omega$, i.e. $\dist(y,\partial \Omega) \, \simeq \, \kappa(a)$ and $\dist(x,\partial \Omega) \, \simeq \, \kappa(a)$, we have 
\begin{equation*}
\left\vert \mathcal{R}(x,y) \right\vert \,  \lesssim \,  \int_{\partial \Omega} \frac{1}{\left\vert t - x \right\vert} \, \frac{1}{\left\vert t - y \right\vert^{2}} \, d\sigma(t),
\end{equation*}
which, by using the Holder inequality, gives us 
\begin{equation}\label{MGI}
    \left\vert \mathcal{R}(x,y) \right\vert \,   \lesssim  \, \left( \int_{\partial \Omega} \frac{1}{\left\vert t - x \right\vert^{3}} \, d\sigma(t) \right)^{\frac{1}{3}} \, \left( \int_{\partial \Omega}  \frac{1}{\left\vert t - y \right\vert^{3}} \, d\sigma(t) \right)^{\frac{2}{3}} \, \lesssim  \, \left( \frac{1}{\dist(x,\partial \Omega)} \right)^{\frac{1}{3}} \, \left(   \frac{1}{\dist(y,\partial \Omega)} \right)^{\frac{2}{3}},
\end{equation}
see \cite[Lemma (4.6)]{Valdivia}.
This concludes the proof of \textbf{Lemma \ref{LemmaG=phi+Remainder}}.
\subsection{Proof of Lemma \ref{ADZ-Lemma}}\label{Proof-A-Priori-Estimate}
We start by recalling, from $(\ref{L.S.E.vf})$, that $v^{g}(\cdot)$ is solution of:
\begin{equation}\label{L.S.E.vf.Lemma}
    v^{g}(x) \, - \, \omega^{2} \, \int_{D} G(x,y) \, v^{g}(y) \, \left( \frac{\rho_{1}}{k_{1}} \, - \, n^{2}(y) \right) \, dy \, = \, S(x), \quad x \in D. 
\end{equation}
In the sequel, we divide the proof into two steps. 
\begin{enumerate}
\item[]
\item The case of one droplet. 
Using the decomposition $(\ref{G=phi+Remainder})$, of the Green's kernel $G(\cdot,\cdot)$, we rewrite $(\ref{L.S.E.vf.Lemma})$ as  
\begin{eqnarray*}
v^{g}(x) \, &-& \, \omega^{2} \, \frac{\rho_{1}}{k_{1}} \; \int_{D} \Phi_{0}(x,y) \, v^{g}(y) \, dy = S(x) \\ &+& \omega^{2} \, \frac{\rho_{1}}{k_{1}} \; \int_{D} \mathcal{R}(x,y) \, v^{g}(y) \, dy \, - \, \omega^{2}  \int_{D} G(x,y) \, v^{g}(y)  \, n^{2}(y)  \, dy.
\end{eqnarray*}
Next, we denote by $\left(\lambda_{n}^{D}; e_{n} \right)_{n \in \mathbb{N}}$ the eigensystem associated to the Newtonian operator $N_{D}(\cdot)$ in $\mathbb{L}^{2}(D)$. Then, after taking the inner product with respect to $e_{n}(\cdot)$ and the square modulus in both sides of the previous equation, we get
\begin{eqnarray*}
\left\vert \langle v^{g}; e_{n} \rangle_{\mathbb{L}^{2}(D)} \right\vert^{2} \, & \lesssim & \, \frac{\left\vert k_{1} \right\vert^{2}}{\left\vert k_{1} - \omega^{2} \; \rho_{1} \, \lambda^{D}_{n} \right\vert^{2}} \, \Bigg[ \left\vert \langle S; e_{n} \rangle_{\mathbb{L}^{2}(D)} \right\vert^{2} \, + \, \left\vert \langle \int_{D} G(\cdot,y) \, n^{2}(y) \, v^{g}(y) \, dy; e_{n} \rangle_{\mathbb{L}^{2}(D)} \right\vert^{2} \\ &+&  \,\left\vert k_{1}  \right\vert^{-2} \, \left\vert \langle \int_{D} \mathcal{R}(\cdot,y) \, v^{g}(y) \, dy; e_{n} \rangle_{\mathbb{L}^{2}(D)} \right\vert^{2} \Bigg].
\end{eqnarray*}
Then, by summing up with respect to the index $n$ and taking into account the relations $(\ref{closeres})$ and $(\ref{ScaleBulk})$ we obtain 
\begin{eqnarray}\label{AMEqua0613}
\nonumber
\left\Vert v^{g} \right\Vert^{2}_{\mathbb{L}^{2}(D)}  & \lesssim & a^{-2h} \; \Bigg[ \left\Vert  S \right\Vert^{2}_{\mathbb{L}^{2}(D)}  + \, a^{-4} \, \left\Vert  \int_{D} \mathcal{R}(\cdot,y) \, v^{g}(y) \, dy \right\Vert^{2}_{\mathbb{L}^{2}(D)} \\ && \qquad \qquad \qquad + \left\Vert \int_{D} G(\cdot,y) \, n^{2}(y) \, v^{g}(y) \, dy \right\Vert^{2}_{\mathbb{L}^{2}(D)} \, \Bigg].
\end{eqnarray}
Next, we estimate the second term and the third term on the R.H.S, as 
\begin{equation*}
\bm{R_1} := \left\Vert  \int_{D} \mathcal{R}(\cdot,y) \, v^{g}(y) \, dy \right\Vert^{2}_{\mathbb{L}^{2}(D)} 
 \leq  \int_{D} \int_{D} \left\vert \mathcal{R}(x,y) \right\vert^{2} \, dy  \; dx \; \left\Vert v^{g} \right\Vert^{2}_{\mathbb{L}^{2}(D)}.
\end{equation*}
Thanks to $(\ref{MGI})$, we have 
\begin{equation*}
\int_{D} \int_{D} \left\vert \mathcal{R}(x,y) \right\vert^{2} \, dy  \; dx \; \lesssim \; \int_{D} \frac{1}{\dist^{\frac{2}{3}}(x,\partial \Omega)} \, dx \, \int_{D} \frac{1}{\dist^{\frac{4}{3}}(y,\partial \Omega)} \, dy,
\end{equation*}
which, by using the fact that $\dist(x,\partial \Omega) \, \geq \, \kappa(a)$ and $\dist(y,\partial \Omega) \, \geq \, \kappa(a)$, gives us\footnote{If the droplet $D$ is away from the boundary $\partial \Omega$, the estimation $(\ref{IEq})$, will be reduced to 
\begin{equation*}
\int_{D} \int_{D} \left\vert \mathcal{R}(x,y) \right\vert^{2} \, dy  \; dx \; = \; \mathcal{O}\left( a^{6} \right).
\end{equation*}
Thus, the estimation $(\ref{IEq})$ corresponds to the worst case.} 
\begin{equation}\label{IEq}
\int_{D} \int_{D} \left\vert \mathcal{R}(x,y) \right\vert^{2} \, dy  \; dx \; \lesssim \; \left( \frac{\left\vert D \right\vert}{\kappa(a)} \right)^{2} \, \overset{(\ref{distto})}{=} \, \mathcal{O}\left( a^{\frac{2(8+h)}{3}} \right).
\end{equation}
Hence, 
\begin{equation}\label{Rvf}
\bm{R_1} := \left\Vert  \int_{D} \mathcal{R}(\cdot,y) \, v^{g}(y) \, dy \right\Vert^{2}_{\mathbb{L}^{2}(D)} \; \lesssim \; a^{\frac{2(8+h)}{3}} \; \left\Vert v^{g} \right\Vert^{2}_{\mathbb{L}^{2}(D)}.
\end{equation}
Furthermore, 
\begin{eqnarray*}
   \bm{R_2} \, &:=& \, \left\Vert \int_{D} G(\cdot,y) \, n^{2}(y) \, v^{g}(y) \, dy \right\Vert^{2}_{\mathbb{L}^{2}(D)} \\
   &\overset{(\ref{G=phi+Remainder})}{\lesssim}& \left\Vert N_{D}\left( n^{2} \, v^{g} \right) \right\Vert^{2}_{\mathbb{L}^{2}(D)} \, + \, \left\Vert \int_{D} \mathcal{R}(\cdot,y) \, n^{2}(y) \, v^{g}(y) \, dy \right\Vert^{2}_{\mathbb{L}^{2}(D)} \\
   &\leq& \left[  \left\Vert N_{D} \right\Vert^{2}_{\mathcal{L}\left(\mathbb{L}^{2}(D); \mathbb{L}^{2}(D)\right)} \,   + \, \int_{D} \int_{D} \left\vert \mathcal{R}(x,y) \right\vert^{2} \, dy \, dx  \, \right] \left\Vert  n^{2} \right\Vert^{2}_{\mathbb{L}^{\infty}(D)}  \, \left\Vert  v^{g}  \right\Vert^{2}_{\mathbb{L}^{2}(D)},
\end{eqnarray*}
which, by using the fact that $\left\Vert  n^{2} \right\Vert_{\mathbb{L}^{\infty}(D)} \, = \, \mathcal{O}\left( 1 \right)$, $\left\Vert N_{D} \right\Vert_{\mathcal{L}\left(\mathbb{L}^{2}(D); \mathbb{L}^{2}(D)\right)} \, = \, \mathcal{O}\left(a^{2} \right)$, and the estimation $(\ref{IEq})$, can be reduced to 
\begin{equation}\label{EstR2Equa0611}
   \bm{R_2} \,  \lesssim  \, a^{4}  \, \left\Vert  v^{g}  \right\Vert^{2}_{\mathbb{L}^{2}(D)}. 
\end{equation}
Then, by plugging $(\ref{Rvf})$ and $(\ref{EstR2Equa0611})$ into $(\ref{AMEqua0613})$, we obtain 
\begin{eqnarray}\label{LREqua1044}
\nonumber
\left\Vert v^{g} \right\Vert^{2}_{\mathbb{L}^{2}(D)}  \, & \lesssim & \, a^{-2h} \; \left\Vert  S \right\Vert^{2}_{\mathbb{L}^{2}(D)} + a^{\frac{4}{3}(1-h)} \, \left\Vert  v^{g} \right\Vert^{2}_{\mathbb{L}^{2}(D)} \\
& \lesssim & \, a^{-2h} \; \left\Vert  S \right\Vert^{2}_{\mathbb{L}^{2}(D)},
\end{eqnarray}
as $0 < h < 1$ and $a \ll 1$. Besides, thanks to $(\ref{NMP1})$ we know that $S \, = \, \mathscr{S}(g)$. Hence, 
\begin{equation}\label{LGEqua1028}
\left\Vert v^{g} \right\Vert^{2}_{\mathbb{L}^{2}(D)}  \,   \lesssim  \, a^{-2h} \; \int_{D} \left\vert \mathscr{S}(g)(x)   \right\vert^{2} \, dx.
\end{equation}
Now, by using the continuity of the Single-Layer operator from $\mathbb{H}^{-\frac{1}{2}}(\partial \Omega)$ to $\mathbb{H}^{1}(\Omega)$, the continuous embedding of $\mathbb{H}^{1}(\Omega)$ into $\mathbb{L}^{6}(\Omega)$, see \cite[Corollary 9.14]{brezis}, and the Holder inequality, we deduce that
\begin{eqnarray}\label{Equa1013}
\nonumber
\left\Vert v^{g} \right\Vert^{2}_{\mathbb{L}^{2}(D)}  \,  & \lesssim &  \, a^{-2h} \, \left( \int_{D} \left\vert \mathscr{S}(g)(x)   \right\vert^{6} \, dx \, \right)^{\frac{1}{3}} \, \left\vert D \right\vert^{\frac{2}{3}} \\ \nonumber
& = &  \, a^{2-2h} \, \left( \int_{D} \left\vert \int_{\partial \Omega} G(x,y) \, g(y) \, d\sigma(y)  \right\vert^{6} \, dx \, \right)^{\frac{1}{3}}  \\ \nonumber
& = &  \, a^{2-2h} \, \left( \int_{D} \left\vert \langle G(x,\cdot); g \rangle_{\mathbb{H}^{\frac{1}{2}}(\partial \Omega) \times \mathbb{H}^{-\frac{1}{2}}(\partial \Omega)}  \right\vert^{6} \, dx \, \right)^{\frac{1}{3}}  \\
& \leq &  \, a^{2-2h} \, \left\Vert g \right\Vert^{2}_{\mathbb{H}^{-\frac{1}{2}}(\partial \Omega)} \, \left( \int_{D} \left\Vert G(x,\cdot)   \right\Vert^{6}_{\mathbb{H}^{\frac{1}{2}}(\partial \Omega)} \, dx \, \right)^{\frac{1}{3}}.
\end{eqnarray}
Repeating the same computations done in $(\ref{InPa})-(\ref{ZMEqua0445})$, we derive the following estimation
\begin{eqnarray}\label{JPTSM}
\nonumber
  \left\Vert v^{g} \right\Vert^{2}_{\mathbb{L}^{2}(D)}  \,  & \lesssim & \, a^{2-2h} \, \left\Vert g \right\Vert^{2}_{\mathbb{H}^{-\frac{1}{2}}(\partial \Omega)} \, \left( \int_{D}  \int_{\Omega^{\diamond}} \frac{1}{\left\vert x - y \right\vert^{12}} \, dy  \, dx \, \right)^{\frac{1}{3}} \\
  &\overset{(\ref{distxykappaa})}{\lesssim}&  \, a^{2-2h} \, \left\Vert g \right\Vert^{2}_{\mathbb{H}^{-\frac{1}{2}}(\partial \Omega)} \, \left( \kappa(a)^{-12} \, \left\vert D \right\vert \right)^{\frac{1}{3}} \, \overset{(\ref{distto})}{=}  \, \mathcal{O}\left(a^{\frac{(5-2h)}{3}} \, \left\Vert g \right\Vert^{2}_{\mathbb{H}^{-\frac{1}{2}}(\partial \Omega)} \right).
\end{eqnarray}
Finally, 
\begin{equation*}\label{ASEqua1305}
    \left\Vert v^{g} \right\Vert_{\mathbb{L}^{2}(D)}  \, \lesssim \, a^{\frac{(5-2h)}{6}} \, \left\Vert g \right\Vert_{\mathbb{H}^{-\frac{1}{2}}(\partial \Omega)}.
\end{equation*}
\item The case of multiple droplets. 
From $(\ref{L.S.E.vf.Lemma})$, by taking $x \in D_{m}$, we get:
 \begin{eqnarray*}
\left( I - \frac{\omega^{2} \, \rho_{1}}{k_{1}}  \; N_{D_{m}} \right)\left(v^{g}_{m}\right)(x) \, &=& \, S_{m}(x) \, + \,  \frac{\omega^{2} \, \rho_{1}}{k_{1}} \;  \sum_{j=1 \atop j \neq m}^{M} \int_{D_{j}} G(x,y) \, v^{g}_{j}(y) \, dy \\ &+&  \frac{\omega^{2} \, \rho_{1}}{k_{1}} \;  \int_{D_{m}} \mathcal{R}(x,y) \, v^{g}_{m}(y) \, dy \,
- \, \omega^{2} \, \int_{D_{m}} G(x,y) \, n^{2}(y) \, v_{m}^{g}(y) \, dy \\
&-& \omega^{2} \, \sum_{j=1 \atop j \neq m}^{M} \int_{D_{j}} G(x,y) \, n^{2}(y) \, v_{j}^{g}(y) \, dy
\end{eqnarray*} 
where $N_{D_m}(\cdot)$ is the Newtonian operator given by $(\ref{DefNPO})$. Besides, taking a Taylor expansion for the functions $G(x,\cdot)$ and $G(x,\cdot) n^{2}(\cdot)$, and inverting the operator $\left( I - \frac{\omega^{2} \, \rho_{1}}{k_{1}}  \; N_{D_{m}} \right)$, we derive
 \begin{eqnarray*}
v^{g}_{m}(x) \, &=& \, \left( I - \frac{\omega^{2} \, \rho_{1}}{k_{1}}  \; N_{D_{m}} \right)^{-1}\Bigg[ S_{m} \, + \,  \frac{\omega^{2} \, \rho_{1}}{k_{1}} \;  \sum_{j=1 \atop j \neq m}^{M} G(\cdot,z_{j}) \, \int_{D_{j}} v^{g}_{j}(y) \, dy \\ &+&   \,  \frac{\omega^{2} \, \rho_{1}}{k_{1}} \;  \sum_{j=1 \atop j \neq m}^{M} \int_{D_{j}} \int_{0}^{1} \underset{y}{\nabla} G(\cdot,z_{j}+t(y-z_{j})) \cdot (y - z_{j}) \, dt  \, v^{g}_{j}(y) \, dy \\ &+& \, \frac{\omega^{2} \, \rho_{1}}{k_{1}} \;  \int_{D_{m}} \mathcal{R}(\cdot,y) \, v^{g}_{m}(y) \, dy \, - \omega^{2} \,  \int_{D_{m}} G(\cdot,y) \, n^{2}(y) \, v_{m}^{g}(y) \, dy \\ 
&-& \omega^{2} \, \sum_{j=1 \atop j \neq m}^{M} G(\cdot,z_{j}) n^{2}(z_{j}) \, \int_{D_{j}} v_{j}^{g}(y) \, dy \\
&-& \, \omega^{2} \, \sum_{j=1 \atop j \neq m}^{M} \int_{D_{j}} \int_{0}^{1} \nabla\left(G(\cdot,\cdot)n^{2}(\cdot) \right)(z_{j}+t(y-z_{j})) \cdot (y-z_{j}) \, dt \, v_{j}^{g}(y) \, dy \Bigg](x), 
\end{eqnarray*} 
with $x \in D_{m}$. Introducing the notation $(\ref{DefYj})$, the above equation can be rewritten as 
 \begin{eqnarray*}
v^{g}_{m}(x) \, &=& \, \left( I - \frac{\omega^{2} \, \rho_{1}}{k_{1}}  \; N_{D_{m}} \right)^{-1}\Bigg[ S_{m} \, + \, \alpha  \,  \sum_{j=1 \atop j \neq m}^{M} G(\cdot,z_{j}) \,Y_{j} \\ &+&   \,  \frac{\omega^{2} \, \rho_{1}}{k_{1}} \;  \sum_{j=1 \atop j \neq m}^{M} \int_{D_{j}} \int_{0}^{1} \underset{y}{\nabla} G(\cdot,z_{j}+t(y-z_{j})) \cdot (y - z_{j}) \, dt  \, v^{g}_{j}(y) \, dy \\ &+& \, \frac{\omega^{2} \, \rho_{1}}{k_{1}} \;  \int_{D_{m}} \mathcal{R}(\cdot,y) \, v^{g}_{m}(y) \, dy \, - \, \omega^{2} \, \int_{D_{m}} G(\cdot,z_{m}) \, n^{2}(y) \, v_{m}^{g}(y) \, dy  \\
&-& \frac{\alpha \, k_{1}}{\rho_{1}} \, \sum_{j=1 \atop j \neq m}^{M} G(\cdot,z_{j}) \, n^{2}(z_{j}) \, Y_{j} \\
&-& \, \omega^{2} \, \sum_{j=1 \atop j \neq m}^{M} \int_{D_{j}} \int_{0}^{1} \nabla\left(G(\cdot,\cdot)n^{2}(\cdot) \right)(z_{j}+t(y-z_{j})) \cdot (y-z_{j}) \, dt \, v_{j}^{g}(y) \, dy
\Bigg](x), 
\end{eqnarray*}
with $x \in D_{m}$. Taking the $\mathbb{L}^{2}(D_{m})$-norm, using $(\ref{Rvf})$, as well as $\rho_{1} \, \sim \, 1, \, k_{1} \, \sim \, a^{2}$, and 
\begin{equation*}
 \left\Vert \left( I - \dfrac{\omega^{2} \, \rho_{1}}{k_{1}} \; N_{D_{m}} \right)^{-1} \right\Vert_{\mathcal{L}(\mathbb{L}^{2}(D_{m});\mathbb{L}^{2}(D_{m}))} \lesssim a^{-h},   
\end{equation*}
proved for the case of one droplet, see  $(\ref{LREqua1044})$, we get: 
\begin{eqnarray*}
\left\Vert v^{g}_{m} \right\Vert_{\mathbb{L}^{2}(D_{m})}  & \lesssim & a^{-h} \; \left\Vert S_{m} \right\Vert_{\mathbb{L}^{2}(D_{m})} + a^{-h} \, \left\vert \alpha \right\vert \,   \sum_{j=1 \atop j \neq m}^{M} \left\Vert G(\cdot,z_{j})  \right\Vert_{\mathbb{L}^{2}(D_{m})} \, \left\vert Y_{j} \right\vert \\ &+& \, a^{-2-h} \;  \sum_{j=1 \atop j \neq m}^{M} \left\Vert \int_{D_{j}} \int_{0}^{1} \underset{y}{\nabla} G(\cdot,z_{j}+t(y-z_{j})) \cdot (y-z_{j}) \, dt \,   v^{g}_{j}(y) \, dy \right\Vert_{\mathbb{L}^{2}(D_{m})}  \\ &+& \, a^{\frac{2(1-h)}{3}} \, \left\Vert v^{g}_{m} \right\Vert_{\mathbb{L}^{2}(D_{m})} \, + \, a^{-h} \, \left\Vert  \int_{D_{m}} G(\cdot,y) n^{2}(y)  \, v_{m}^{g}(y) \, dy \right\Vert_{\mathbb{L}^{2}(D_{m})} \\ 
&+& a^{2-h} \, \left\vert \alpha \right\vert \,  \sum_{j=1 \atop j \neq m}^{M} \left\Vert G(\cdot,z_{j}) \right\Vert_{\mathbb{L}^{2}(D_{m})}  \, \left\vert Y_{j} \right\vert \\
&+& \, a^{-h} \, \sum_{j=1 \atop j \neq m}^{M} \left\Vert \int_{D_{j}} \int_{0}^{1} \nabla\left(G(\cdot,\cdot)n^{2}(\cdot) \right)(z_{j}+t(y-z_{j})) \cdot (y-z_{j}) \, dt \, v_{j}^{g}(y) \, dy \right\Vert_{\mathbb{L}^{2}(D_{m})}.
\end{eqnarray*}
Now, by estimating the terms containing the Green's kernel $G(\cdot,\cdot)$ appearing on the right hand side of the above inequality, and using the fact that $h \, < \, 1$, the previous inequality can be reduced to,
\begin{eqnarray*}
\left\Vert v^{g}_{m} \right\Vert_{\mathbb{L}^{2}(D_{m})}  & \overset{(\ref{SG})}{\lesssim} & a^{-h} \; \left\Vert S_{m} \right\Vert_{\mathbb{L}^{2}(D_{m})} + a^{\frac{3}{2}-h} \, \left\vert \alpha \right\vert \; \left( \sum_{j=1 \atop j \neq m}^{M} \frac{1}{\left\vert z_{m} \, - \, z_{j} \right\vert^{2}} \right)^{\frac{1}{2}} \, \left( \sum_{j=1 \atop j \neq m}^{M} \left\vert Y_{j} \right\vert^{2} \right)^{\frac{1}{2}} \\ &+& \, a^{2-h} \;  \left( \sum_{j=1 \atop j \neq m}^{M} \frac{1}{\left\vert z_{m} \, - \, z_{j} \right\vert^{4}} \right)^{\frac{1}{2}} \; \left( \sum_{j=1 \atop j \neq m}^{M} \left\Vert v^{g}_{j} \right\Vert^{2}_{\mathbb{L}^{2}(D_{j})} \right)^{\frac{1}{2}} \\ 
&+& a^{2-h} \, \left\Vert v_{m}^{g} \right\Vert_{\mathbb{L}^{2}(D_{m})} \, + \, a^{\frac{7}{2}-h} \, \left\vert \alpha \right\vert \, \left( \sum_{j=1 \atop j \neq m}^{M} \frac{1}{\left\vert z_{m} - z_{j} \right\vert^{2}} \right)^{\frac{1}{2}} \, \left( \sum_{j=1 \atop j \neq m}^{M} \left\vert Y_{j} \right\vert^{2} \right)^{\frac{1}{2}} \\
&+& \, a^{4-h} \;  \left( \sum_{j=1 \atop j \neq m}^{M} \frac{1}{\left\vert z_{m} \, - \, z_{j} \right\vert^{4}} \right)^{\frac{1}{2}} \; \left( \sum_{j=1 \atop j \neq m}^{M} \left\Vert v^{g}_{j} \right\Vert^{2}_{\mathbb{L}^{2}(D_{j})} \right)^{\frac{1}{2}},
\end{eqnarray*}
which, by using $(\ref{SID})$, can be reduced to
\begin{equation*}
\left\Vert v^{g}_{m} \right\Vert_{\mathbb{L}^{2}(D_{m})}   \lesssim  a^{-h} \; \left\Vert S_{m} \right\Vert_{\mathbb{L}^{2}(D_{m})} + a^{\frac{1}{2}-h} \, \left\vert \alpha \right\vert \, d^{-\frac{3}{2}} \, \left( \sum_{j=1}^{M} \left\vert \, Y_{j} \right\vert^{2} \right)^{\frac{1}{2}} + \, a^{2-h} \;  d^{-2} \; \left( \sum_{j=1 \atop j \neq m}^{M} \left\Vert v^{g}_{j} \right\Vert^{2}_{\mathbb{L}^{2}(D_{j})} \right)^{\frac{1}{2}}.
\end{equation*}
Besides, taking the square in both sides of the above equation, using the fact that $d \, \sim \, a^{\frac{(1-h)}{3}}$, see $(\ref{dmin})$, and the estimation $\alpha \, \sim \, a^{1-h}$, see $(\ref{EstimationRealpha})$, we deduce 
\begin{equation*}
\left\Vert v^{g}_{m} \right\Vert^{2}_{\mathbb{L}^{2}(D_{m})}   \lesssim  a^{-2h} \; \left\Vert S_{m} \right\Vert^{2}_{\mathbb{L}^{2}(D_{m})} + a^{4-3h} \, \sum_{j=1}^{M} \left\vert Y_{j} \right\vert^{2}  \, + \, a^{\frac{2(4-h)}{3}}  \;  \sum_{j=1}^{M} \left\Vert v^{g}_{j} \right\Vert^{2}_{\mathbb{L}^{2}(D_{j})},
\end{equation*}
which, by summing up with respect to the index $m$ gives us,
\begin{equation*}
\sum_{m=1}^{M} \left\Vert v^{g}_{m} \right\Vert^{2}_{\mathbb{L}^{2}(D_{m})}   \lesssim  a^{-2h} \; \sum_{m=1}^{M} \left\Vert S_{m} \right\Vert^{2}_{\mathbb{L}^{2}(D_{m})} + \, a^{4-3h} \, M \,  \sum_{j=1}^{M} \left\vert Y_{j} \right\vert^{2}  \, + \, a^{\frac{2(4-h)}{3}} \, M \;  \sum_{j=1}^{M} \left\Vert v^{g}_{j} \right\Vert^{2}_{\mathbb{L}^{2}(D_{j})}.
\end{equation*}
Since $M \, \sim \, a^{h-1}$, see $(\ref{M-})$, we obtain  
\begin{equation*}
\sum_{m=1}^{M} \left\Vert v^{g}_{m} \right\Vert^{2}_{\mathbb{L}^{2}(D_{m})}   \lesssim  a^{-2h} \; \sum_{m=1}^{M} \left\Vert S_{m} \right\Vert^{2}_{\mathbb{L}^{2}(D_{m})} + a^{3-2h} \,  \sum_{j=1}^{M} \left\vert Y_{j} \right\vert^{2} \, + \, a^{\frac{5}{3}(1-h)} \, \sum_{j=1}^{M} \left\Vert v^{g}_{j} \right\Vert^{2}_{\mathbb{L}^{2}(D_{j})},
\end{equation*}
which, by knowing that $h < 1$, can be reduced to
\begin{eqnarray}\label{Equa0812}
\nonumber
\sum_{m=1}^{M} \left\Vert v^{g}_{m} \right\Vert^{2}_{\mathbb{L}^{2}(D_{m})}  & \lesssim & a^{-2h} \; \sum_{m=1}^{M} \left\Vert S_{m} \right\Vert^{2}_{\mathbb{L}^{2}(D_{m})} + a^{3-2h} \,  \sum_{j=1}^{M} \left\vert Y_{j} \right\vert^{2} \\ 
\left\Vert v^{g} \right\Vert^{2}_{\mathbb{L}^{2}(D)} \,  & \overset{(\ref{Equa0454})}{\underset{(\ref{NMP1})}{\lesssim}} & \, a^{-2h} \; \left\Vert \mathscr{S}(g) \right\Vert^{2}_{\mathbb{L}^{2}(D)} \, + \,  a^{\frac{(5-2h)}{3}} \; \left\Vert g \right\Vert^{2}_{\mathbb{H}^{-\frac{1}{2}}\left( \partial \Omega \right)}.
\end{eqnarray}
Let us estimate $\left\Vert \mathscr{S}(g) \right\Vert^{2}_{\mathbb{L}^{2}(D)}$. To do this, as done for the case of a single droplet, see $(\ref{LGEqua1028})-(\ref{JPTSM})$, by using the Holder inequality, we can derive the following inequalities 
\begin{eqnarray*}\label{Equa0809}
\nonumber
    \left\Vert \mathscr{S}(g) \right\Vert^{2}_{\mathbb{L}^{2}(D)} & \leq & \sum_{m=1}^{M} \left\Vert \mathscr{S}(g) \right\Vert^{2}_{\mathbb{L}^{6}(D_{m})} \, \left\Vert 1 \right\Vert^{2}_{\mathbb{L}^{3}(D_{m})} \\ \nonumber
    &\lesssim& a^{2} \, \sum_{m=1}^{M}  \, \left[\int_{D_{m}} \left\vert \int_{\partial \Omega} G(x,y) \, g(y) \, dy \right\vert^{6} \; dx \right]^{\frac{1}{3}} \\ \nonumber
    & \leq & a^{2} \, \left\Vert g \right\Vert^{2}_{\mathbb{H}^{-\frac{1}{2}}(\partial \Omega)}  \; \sum_{m=1}^{M} \left[\int_{D_{m}} \left\Vert G(x,\cdot)  \right\Vert^{6}_{\mathbb{H}^{\frac{1}{2}}(\partial \Omega)} \; dx \right]^{\frac{1}{3}} \\ \nonumber
     & \underset{(\ref{JPTSM})}{\overset{(\ref{Equa1013})}{\leq}} & a^{2} \, \left\Vert g \right\Vert^{2}_{\mathbb{H}^{-\frac{1}{2}}(\partial \Omega)}  \; \sum_{m=1}^{M} \left( \int_{D_{m}} \int_{\Omega^{\diamond}} \frac{1}{\left\vert x - y \right\vert^{12}} \, dy \, dx \right)^{\frac{1}{3}} \\ \nonumber 
    &\lesssim & a^{3} \, \left\Vert g \right\Vert^{2}_{\mathbb{H}^{-\frac{1}{2}}(\partial \Omega)}  \; \sum_{m=1}^{M} \frac{1}{\dist^{4}(D_{m};\partial \Omega)} \\ \nonumber
    &\overset{(\ref{SID+})}{\lesssim}& a^{3} \, \left\Vert g \right\Vert^{2}_{\mathbb{H}^{-\frac{1}{2}}(\partial \Omega)}  \; d^{-4} \\
    &=& \left(a^{\frac{(5+4h)}{3}} \, \left\Vert g \right\Vert^{2}_{\mathbb{H}^{-\frac{1}{2}}(\partial \Omega)} \right).
\end{eqnarray*}
\end{enumerate} 
Thus, by plugging the above estimation into $(\ref{Equa0812})$, we obtain 
\begin{equation}\label{ASEqua1304}
\left\Vert v^{g} \right\Vert_{\mathbb{L}^{2}(D)}   \lesssim  a^{\frac{(5-2h)}{6}} \; \left\Vert g \right\Vert_{\mathbb{H}^{-\frac{1}{2}}(\partial \Omega)}. 
\end{equation}
This ends the proof of \textbf{Lemma \ref{ADZ-Lemma}}.
\subsection{Proof of Lemma $\ref{EstimationRealphaImalpha}$}\label{ESC}
We know that, for $m$ fixed, 
\begin{equation*}
\alpha := \int_{D_{m}} W_{m}(x) \, dx  =   \int_{D_{m}} \left( \dfrac{k_{1}}{\omega^{2} \, \rho_{1}} \; I -  N_{D_{m}} \right)^{-1}\left( 1 \right)(x) \; dx.
\end{equation*}
By expanding the constant function 1 over the basis of the Newtonian operator $N_{D_{m}}(\cdot)$ we obtain 
\begin{equation*}\label{DispersionEqua}
\alpha  =  \sum_{n} \; \langle 1 ; e_{n} \rangle_{\mathbb{L}^{2}(D_{m})} \; \int_{D_{m}} \left( \dfrac{k_{1}}{\omega^{2} \, \rho_{1}} \; I -  N_{D_{m}} \right)^{-1}\left( e_{n} \right)(x) \; dx =  \sum_{n} \; \left( \langle 1 ; e_{n} \rangle_{\mathbb{L}^{2}(D_{m})} \right)^{2} \;  \frac{\omega^{2} \, \rho_{1}}{\left( k_{1} - \omega^{2} \, \rho_{1} \, \lambda^{D_{m}}_{n}\right)}. 
\end{equation*}
We choose $\omega$ solution of the coming  dispersion equation 
\begin{equation*}
k_{1} - \omega^{2} \; \rho_{1} \, \lambda^{D_{m}}_{n_{0}} =  c_{n_0} \, a^{2+h}, \quad c_{n_{0}} \in \mathbb{R}.
\end{equation*}
By solving the previous quadratic equation we obtain 
\begin{equation*}\label{DispersionEquation}
\omega^{2} \, = \, \frac{k_{1} \, - \, c_{n_{0}} \; a^{2+h}}{\rho_{1} \, \lambda^{D_{m}}_{n_{0}}}.
\end{equation*}
Hence
\begin{equation}\label{closeres}
    \left\vert k_{1} - \omega^{2} \, \rho_{1} \, \lambda^{D_{m}}_{n} \right\vert = \begin{cases}
			a^{2+h}, & \text{if $n = n_{0}$}\\
            a^{2}, & \text{otherwise}
		 \end{cases}.
\end{equation}
Then,
\begin{equation*}\label{alpha=n0+Remainder}
\alpha  =  \left( \langle 1 ; e_{n_{0}} \rangle_{\mathbb{L}^{2}(D_{m})} \right)^{2} \;  \;  \frac{\omega^{2} \, \rho_{1}}{\left( k_{1} - \omega^{2} \, \rho_{1} \, \lambda^{D_{m}}_{n_{0}}\right)} + \sum_{n \neq n_{0}} \; \left( \langle 1 ; e_{n} \rangle_{\mathbb{L}^{2}(D_{m})} \right)^{2} \; \;  \frac{\omega^{2} \, \rho_{1}}{\left( k_{1} - \omega^{2}  \, \rho_{1} \, \lambda^{D_{m}}_{n}\right)}  
\end{equation*}
We estimate the second term on the R.H.S as
\begin{equation*}
\left\vert T_{R.H.S} \right\vert  \lesssim  \sum_{n \neq n_{0}} \frac{\left\vert \langle 1 ; e_{n} \rangle_{\mathbb{L}^{2}(D_{m})} \right\vert^{2}}{ \left\vert k_{1} - \omega^{2} \, \rho_{1} \, \lambda^{D_{m}}_{n} \right\vert} \overset{(\ref{closeres})}{\lesssim}  a^{-2} \; \left\Vert 1 \right\Vert^{2}_{\mathbb{L}^{2}(D_{m})}  = \mathcal{O}\left( a \right),
\end{equation*}
then 
\begin{equation*}\label{alphabtreim}
\alpha  = \left( \langle 1 ; e_{n_{0}} \rangle_{\mathbb{L}^{2}(D_{m})} \right)^{2} \; \frac{\omega^{2} \, \rho_{1}}{\left( k_{1} - \omega^{2} \; \rho_{1} \, \lambda^{D_{m}}_{n_{0}}  \right)}  + \mathcal{O}\left( a \right). 
\end{equation*}
Knowing that $\langle 1 ; e_{n_{0}} \rangle_{\mathbb{L}^{2}(D_{m})} = a^{\frac{3}{2}} \, \langle 1 ; \overline{e}_{n_{0}} \rangle_{\mathbb{L}^{2}(B)}$, and using the fact that \\ $k_{1} = a^{2} \, k_{0};\, \lambda^{D_{m}}_{n_{0}} = a^{2} \, \lambda_{n_{0}}^{B}$, we rewrite the previous equation like
\begin{equation}\label{Scattering-coefficient}
\alpha \, = \, \frac{\left( k_{0} - c_{n_{0}} \; a^{h} \right)}{\lambda_{n_{0}}^{B} \; c_{n_0} } \; \left( \langle 1 ; \overline{e}_{n_{0}} \rangle_{\mathbb{L}^{2}(B)} \right)^{2} \; a^{1-h} + \mathcal{O}\left( a \right) \, = \, \frac{k_{0}}{\lambda_{n_{0}}^{B} \; c_{n_0} } \; \left( \langle 1 ; \overline{e}_{n_{0}} \rangle_{\mathbb{L}^{2}(B)} \right)^{2} \; a^{1-h} + \mathcal{O}\left( a \right).
\end{equation}
We set $P^{2}$ to be the scaled dominant part of $\alpha$, i.e., 
\begin{equation*}\label{DefP2}
P^{2} := \frac{- \, k_{0} \; \left( \langle 1 ; \overline{e}_{n_{0}} \rangle_{\mathbb{L}^{2}(B)} \right)^{2}}{\lambda_{n_{0}}^{B} \; c_{n_0} }, 
\end{equation*}
and we end up with the following formula  
\begin{equation*}
\alpha  \, = \, - \, P^{2} \, a^{1-h} +\mathcal{O}\left(a \right). 
\end{equation*}
To estimate $\left\Vert W_{m} \right\Vert_{\mathbb{L}^{2}(D_{m})}$, we use the same above arguments to derive
\begin{equation*}
    \left\Vert W_{m} \right\Vert^{2}_{\mathbb{L}^{2}(D_{m})} \,  =  \, \sum_{n} \frac{\left\vert \omega^{2} \, \rho_{1} \right\vert^{2}}{\left\vert k_{1} \, - \, \omega^{2} \, \rho_{1} \, \lambda^{D_{m}}_{n} \right\vert^{2}} \, \left\vert \langle 1 ; e_{n} \rangle_{\mathbb{L}^{2}(D_{m})} \right\vert^{2} \,
     \overset{(\ref{closeres})}{\lesssim} \, a^{-2(2+h)} \, \left\Vert 1 \right\Vert^{2}_{\mathbb{L}^{2}(D_{m})}.
\end{equation*}
Hence, 
\begin{equation*}
    \left\Vert W_{m} \right\Vert_{\mathbb{L}^{2}(D_{m})} \,  \lesssim \, a^{-(2+h)} \, \left\Vert 1 \right\Vert_{\mathbb{L}^{2}(D_{m})} \, = \, \mathcal{O}\left( a^{- (\frac{1}{2}+h)} \right).
\end{equation*}    
This concludes the proof of \textbf{Lemma $\ref{EstimationRealphaImalpha}$}.
\subsection{Proof of Lemma \ref{Lemma51}}\label{Prooflemma51}
We compute the $\mathbb{L}^{2}(\Omega)$-norm of the term $T_{1}(\cdot)$ defined by $(\ref{DefT1})$. 
\begin{equation*}
     \left\Vert T_{1} \right\Vert^{2}_{\mathbb{L}^{2}(\Omega)}  :=  \int_{\Omega} \left\vert \int_{\Omega} G(y,x) \, \boldsymbol{Y}(x) \; dx - \sum_{m=1}^{M} \, \chi_{\Omega_{m}}(y) \sum_{j = 1 \atop j \neq m}^{M} \int_{\Omega} G(z_{m};z_{j}) \,  \chi_{\Omega_{j}}(x) \; \frac{1}{\beta_{j}} \, Y_{j} \; dx  \right\vert^{2} \, dy.
\end{equation*}
In contrary to \textbf{Subsection \ref{SubSection43}}, where the cutting of $\Omega$ onto $\underset{j=1}{\overset{M}{\cup}} \Omega_{j}$ and $\underset{j=1}{\overset{\aleph}{\cup}} \Omega_{j}^{\star}$ was of capital importance to derive the exact dominant term related to $\int_{\Omega_{j}} u^{g}(x) \, dx$, for $1 \leq j \leq M$. Here, we need only to estimate functions  (not to extract dominant term) defined in $\Omega$, thus involving both $\underset{j=1}{\overset{M}{\cup}} \Omega_{j}$ and $\underset{j=1}{\overset{\aleph}{\cup}} \Omega_{j}^{\star}$.  Because, for every $1 \leq j \leq M$ and $1 \leq k \leq \aleph$, we have $\left\vert \Omega_{j} \right\vert \sim a^{1-h} \sim \left\vert \Omega_{k}^{\star} \right\vert$, we do not need to specify, in our comping computations, if we are dealing with $\left\{ \Omega_{j} \right\}_{j=1}^{M}$ or $\left\{ \Omega_{j}^{\star} \right\}_{j=1}^{\aleph}$. Moreover, to write short, we use the notation $\Omega_{j}$ for the domains  $\Omega_{j}^{\star}$. Then, 
    \begin{eqnarray*}
    \left\Vert T_{1} \right\Vert^{2}_{\mathbb{L}^{2}(\Omega)}    & = & \int_{\Omega} \left\vert \sum_{m=1}^{M} \, \left( \int_{\Omega} G(y,x) \, \boldsymbol{Y}(x) \; dx -   \sum_{j = 1 \atop j \neq m}^{M} \int_{\Omega} G(z_{m};z_{j}) \,  \chi_{\Omega_{j}}(x) \; \frac{1}{\beta_{j}} \, Y_{j} \; dx \right) \chi_{\Omega_{m}}(y) \right\vert^{2} \, dy \\
        & = & \sum_{m=1}^{M} \, \int_{\Omega_{m}} \left\vert  \int_{\Omega} G(y,x) \, \boldsymbol{Y}(x) \; dx -   \sum_{j = 1 \atop j \neq m}^{M} \int_{\Omega} G(z_{m};z_{j}) \,  \chi_{\Omega_{j}}(x) \; \frac{1}{\beta_{j}} \, Y_{j} \; dx  \right\vert^{2} \, dy.
    \end{eqnarray*}
Using the definition of $\boldsymbol{Y}(\cdot)$, see  $(\ref{DefYDefS})$, and the triangular inequality we rewrite the previous equation as 
    \begin{eqnarray}\label{DTEqua1404}
    \nonumber
        \left\Vert T_{1} \right\Vert^{2}_{\mathbb{L}^{2}(\Omega)} 
        & \lesssim & \sum_{m=1}^{M} \left\vert Y_{m} \right\vert^{2} \, \left\vert \Omega_{m} \right\vert \, \int_{\Omega_{m}} \int_{\Omega_{m}} \left\vert G(y,x) \right\vert^{2} \; dx \, dy \\ &+& \sum_{m=1}^{M} \,     \sum_{j = 1 \atop j \neq m}^{M} \left\vert Y_{j} \right\vert^{2} \;\sum_{j = 1 \atop j \neq m}^{M} \left\vert \Omega_{j} \right\vert \int_{\Omega_{m}} \, \int_{\Omega_{j}} \left\vert G(y,x)  -   G(z_{m};z_{j}) \, \frac{1}{\beta_{j}}\,  \right\vert^{2} \; dx \; dy. 
    \end{eqnarray}
    Furthermore, by using $(\ref{DTBeta_m})$, we have 
    \begin{equation*}
        \left\vert G(y,x)  -   G(z_{m};z_{j}) \, \frac{1}{\beta_{j}}\,  \right\vert^{2} \, \lesssim \, \left\vert G(y,x)  -   G(z_{m};z_{j}) \,  \right\vert^{2} \, + \, a^{\frac{4(1-h)}{3}} \, \left\vert G(z_{m};z_{j}) \,  \right\vert^{2}.
    \end{equation*}
    Hence, by plugging the above estimation into $(\ref{DTEqua1404})$, we obtain 
        \begin{eqnarray*}
        \left\Vert T_{1} \right\Vert^{2}_{\mathbb{L}^{2}(\Omega)} 
        & \leq & \sum_{m=1}^{M} \left\vert Y_{m} \right\vert^{2} \, \Bigg[ \underset{1 \leq m \leq M}{\max} \left( \left\vert \Omega_{m} \right\vert \, \int_{\Omega_{m}} \int_{\Omega_{m}} \left\vert G(y,x) \right\vert^{2} \; dx \, dy \right) \\ &+& \sum_{m=1}^{M}  \sum_{j = 1 \atop j \neq m}^{M} \left\vert \Omega_{j} \right\vert \int_{\Omega_{m}} \int_{\Omega_{j}} \left\vert G(y,x)  -   G(z_{m};z_{j}) \right\vert^{2}  dx  dy\Bigg] 
        +  a^{\frac{10(1-h)}{3}} \sum_{j=1}^{M} \left\vert Y_{j} \right\vert^{2} \sum_{j=1 \atop j \neq m}^{M} \left\vert G(z_{m};z_{j}) \right\vert^{2}, 
    \end{eqnarray*}
Using Taylor expansion for the function $G(\cdot;\cdot)$, near the centers, we get 
\begin{equation*}
    G(y,x)-G(z_{m};z_{j}) \, = \, \int_{0}^{1} \nabla G(z_{m};z_{j}+t(x-z_{j})) \cdot (x-z_{j}) \; dt \, + \,  \int_{0}^{1} \nabla G(z_{m}+t(y-z_{m});x) \cdot (y-z_{m}) \; dt. 
\end{equation*}
We plug the previous expansion into the previous estimation and we use $(\ref{SG})$ to reduce the previous estimation to:

        \begin{eqnarray*}
        \left\Vert T_{1} \right\Vert^{2}_{\mathbb{L}^{2}(\Omega)} 
        & \lesssim & \sum_{m=1}^{M} \left\vert Y_{m} \right\vert^{2} \, \Bigg[ \underset{1 \leq m \leq M}{\max} \left( \left\vert \Omega_{m} \right\vert \, \int_{\Omega_{m}} \int_{\Omega_{m}} \frac{1}{\left\vert y \, - \, x \right\vert^{2}}  \; dx \, dy \right) \\ &+& \sum_{m=1}^{M}  \;\sum_{j = 1 \atop j \neq m}^{M} \left\vert \Omega_{j} \right\vert \int_{\Omega_{m}} \, \int_{\Omega_{j}} \frac{\left\vert x - z_{j} \right\vert^{2}}{\left\vert y - z_{j} \right\vert^{4}}  \; dx \; dy\Bigg] + a^{\frac{10(1-h)}{3}} \, \sum_{j=1}^{M} \left\vert Y_{j} \right\vert^{2} \, \sum_{j=1 \atop j \neq m}^{M} \frac{1}{\left\vert z_{m} - z_{j} \right\vert^{2}}. 
    \end{eqnarray*}
    Besides, by knowing that $\left\vert \Omega_{m} \right\vert \, = \, a^{1-h}$, with $1 \leq m \leq M$, we deduce the following estimation  
    \begin{equation*}
        \underset{1 \leq m \leq M}{\max} \left( \left\vert \Omega_{m} \right\vert \, \int_{\Omega_{m}} \int_{\Omega_{m}} \frac{1}{\left\vert y \, - \, x \right\vert^{2}}  \; dx \, dy \right) \, \lesssim \, a^{\frac{7}{3}(1-h)}.
    \end{equation*}
    Then, 
    \begin{eqnarray*}
        \left\Vert T_{1} \right\Vert^{2}_{\mathbb{L}^{2}(\Omega)} 
         & \lesssim & \sum_{m=1}^{M} \left\vert Y_{m} \right\vert^{2} \, \left[ a^{\frac{7}{3}(1-h)} \, + \, \underset{1 \leq j \leq M}{\max} \left( \left\vert \Omega_{j} \right\vert  \right) \, \sum_{m=1}^{M}  \;\sum_{j = 1 \atop j \neq m}^{M} \, \int_{\Omega_{m}} \frac{1}{\left\vert y - z_{j} \right\vert^{4}} \, dy \,   \int_{\Omega_{j}} \left\vert x - z_{j} \right\vert^{2}  \; dx \right] \\
         &+& a^{\frac{10(1-h)}{3}} \, \sum_{j=1}^{M} \left\vert Y_{j} \right\vert^{2} \, \sum_{j=1 \atop j \neq m}^{M} \frac{1}{\left\vert z_{m} - z_{j} \right\vert^{2}}.
    \end{eqnarray*}
    In addition, by Taylor expansion, we have 
    \begin{equation*}
        \int_{\Omega_{m}} \frac{1}{\left\vert y - z_{j} \right\vert^{4}} \, dy \, = \, \frac{1}{\left\vert z_{m} - z_{j} \right\vert^{4}} \, \left\vert \Omega_{m} \right\vert \,+ \, \int_{\Omega_{m}} \int_{0}^{1} \nabla\left( \left\vert \cdot - z_{j} \right\vert^{-4} \right)\left(z_{m}+t(y-z_{m})\right) \cdot (y-z_{m}) \, dt \, dy,
    \end{equation*}
    hence 
        \begin{eqnarray*}
        \left\Vert T_{1} \right\Vert^{2}_{\mathbb{L}^{2}(\Omega)} 
         & \lesssim & \sum_{m=1}^{M} \left\vert Y_{m} \right\vert^{2} \, \left[ a^{\frac{7}{3}(1-h)} \, + \, a^{(1-h)} \, \underset{1 \leq j \leq M}{\max} \left(  \int_{\Omega_{j}} \left\vert x - z_{j} \right\vert^{2}  \; dx \right) \, \sum_{m=1}^{M} \left\vert \Omega_{m} \right\vert  \, \sum_{j = 1 \atop j \neq m}^{M} \,  \frac{1}{\left\vert z_{m} - z_{j} \right\vert^{4}} \,  \right] \\
         &+& a^{\frac{10(1-h)}{3}} \, \sum_{j=1}^{M} \left\vert Y_{j} \right\vert^{2} \, \sum_{j=1 \atop j \neq m}^{M} \frac{1}{\left\vert z_{m} - z_{j} \right\vert^{2}}
    \end{eqnarray*}
    The following estimations hold
    \begin{eqnarray*}
         \underset{1 \leq j \leq M}{\max} \left(  \int_{\Omega_{j}} \left\vert x - z_{j} \right\vert^{2}  \; dx \right) \, &=& \, \mathcal{O}\left( a^{\frac{5}{3}(1-h)} \right) \\
         \sum_{m=1}^{M} \left\vert \Omega_{m} \right\vert  \, \sum_{j = 1 \atop j \neq m}^{M} \,  \frac{1}{\left\vert z_{m} - z_{j} \right\vert^{4}} \, & \overset{(\ref{SID})}{\lesssim} & \, d^{-4} \, \sum_{m=1}^{M} \left\vert \Omega_{m} \right\vert \, = \, \mathcal{O}\left( d^{-4} \right) \\
         \sum_{j=1 \atop j \neq m}^{M} \frac{1}{\left\vert z_{m} - z_{j} \right\vert^{2}} & \overset{(\ref{SID})}{\lesssim} & \mathcal{O}\left( d^{-3} \right),
    \end{eqnarray*}
    then, by recalling that $d \, \sim \, a^{\frac{1}{3}(1-h)}$ and using the fact that $\left\vert \Omega_{m_{0}} \right\vert \, =  \, a^{1-h}$, we obtain 
    \begin{equation*}
        \left\Vert T_{1} \right\Vert^{2}_{\mathbb{L}^{2}(\Omega)} 
          \lesssim  \sum_{m=1}^{M} \left\vert Y_{m} \right\vert^{2} \, \left[ a^{\frac{7}{3}(1-h)} \, + \, a^{\frac{4}{3}(1-h)}  \,  \right] \, = \, \mathcal{O}\left(a^{\frac{4}{3}(1-h)}  \, \sum_{m=1}^{M} \left\vert Y_{m} \right\vert^{2} \right) \, \overset{(\ref{Equa0714})}{=} \, \mathcal{O}\left(a^{\frac{(1-h)}{3}}  \, \left\Vert \boldsymbol{Y} \right\Vert^{2}_{\mathbb{L}^{2}(\Omega)}  \right).
    \end{equation*}
    Finally, 
     \begin{equation*}
        \left\Vert T_{1} \right\Vert_{\mathbb{L}^{2}(\Omega)} 
         \,  = \, \mathcal{O}\left(a^{\frac{(1-h)}{6}}  \, \left\Vert \boldsymbol{Y} \right\Vert_{\mathbb{L}^{2}(\Omega)}  \right).  
    \end{equation*}
This concludes the proof of \textbf{Lemma \ref{Lemma51}}. 
\subsection{Proof of \textbf{Lemma \ref{distLemma}}}\label{ProofCounitingLemma}
To prove $(\ref{SID})$ we refer the readers to \cite[Section 3.3]{Ammari_2019}. To justify $(\ref{SID+})$, we set $\bm{\Omega}_{n}$ to be 
\begin{equation*}
    \bm{\Omega}_{n} \, := \, \left\{ x \in \Omega, \quad (n-1) \, d \, \leq \, \dist\left(x, \partial \Omega \right) \, \leq \, n \, d \right\}, \quad \text{for} \quad n = 1, \cdots , \left[d^{-1} \right],
\end{equation*}
where $d$ is the minimum distance given by $(\ref{dmin})$. Then, $\Omega \, \subset \, \underset{n=1}{\overset{\left[d^{-1} \right]}{\cup}}   \bm{\Omega}_{n}$, and 
\begin{equation*}
    \left\vert \bm{\Omega}_{n} \right\vert \, \lesssim \, \left( n \, d \right)^{2} \, d \, = \, \mathcal{O}\left(n^{2} \, d^{3} \right).
\end{equation*}
Then, the number of droplets in $\bm{\Omega}_{n}$ is of order $n^{2}$. Hence, 
\begin{equation*}
    \sum_{j=1}^{M} \frac{1}{\dist^{k}\left(D_{j}; \partial \Omega \right)} \,  =  \, \sum_{n=1}^{\left[d^{-1} \right]} \sum_{D_j \subset \Omega_n} \frac{1}{\dist^{k}\left(D_{j}; \partial \Omega \right)} \, \lesssim \, \sum_{n=1}^{\left[d^{-1} \right]} n^{2} \, \frac{1}{\left((n-1) \, d\right)^{k}} \, \simeq \, \frac{1}{d^{k}} \sum_{n=1}^{\left[d^{-1} \right]}\frac{1}{n^{k-2}}. 
\end{equation*}
Therefore, 
\begin{equation*}
       \sum_{j=1}^{M} \frac{1}{\dist^{k}\left(D_{j}; \partial \Omega \right)}  \, = \, \begin{cases}
			\mathcal{O}\left(d^{-3}\right), & \text{for $k < 3$,}\\
            & \\
            \mathcal{O}\left(d^{-k}\right), & \text{for $k > 3$.}
		 \end{cases}.
\end{equation*}
This ends the proof of \textbf{Lemma \ref{distLemma}}.
\subsection{Normal derivative of $SL^{p}(\cdot)$.}\label{JSLp}
We recall, from $(\ref{Equa0848})$, the following definition of the Single-Layer operator $SL^{p}(\cdot)$ 
\begin{equation}\label{EqSL1225}
SL^{p}\left( f \right)(x) \, := \, \int_{\partial \Omega} G_{p}(x,y) \, f(y) \, d\sigma(y), \quad x \in \Omega,
\end{equation}
where $G_{p}(\cdot,\cdot)$ is solution of $(\ref{Green's-Kernel-with-P})$. The goal of this subsection is to compute the jumping coefficient of the normal derivative related to $(\ref{EqSL1225})$. To do this, we let $\Phi_{ip}(\cdot,\cdot)$ to be the fundamental solution of $\left( \Delta - P^{2} \right) \, \Phi_{ip}(\cdot,\cdot) \, = \, - \, \boldsymbol{\delta}_{\cdot}(\cdot)$, in $\mathbb{R}^{3}$. Multiplying $(\ref{Green's-Kernel-with-P})$ by $\Phi_{ip}(\cdot,\cdot)$ and integrating over $\Omega$, allows us to deduce that 
\begin{equation}\label{ASEq1232}
    G_{p}(\cdot,y) \, + \, D^{ip}\left( G_{p}(\cdot,y)\right)(\cdot) \, = \, \Phi_{ip}(\cdot,y), \quad \text{in} \;\; \Omega, 
\end{equation}
where $y \in \Omega$ is taken as a parameter, and $D^{ip}(\cdot)$ is the Double-Layer operator associated to $\Phi_{ip}(\cdot,\cdot)$. Besides, from $(\ref{ASEq1232})$, we deduce that 
\begin{equation}\label{ASEq1233}
    G_{p}(\cdot,y)  \, = \, \left(\frac{1}{2} I \, + \, K^{ip} \right)^{-1}\left( \Phi_{ip}(\cdot,y)\right), \quad \text{on} \;\;  \partial \Omega, 
\end{equation}
where $y \in \Omega$ is taken as a parameter, and $K^{ip}(\cdot)$ is the Neumann-Poincare operator associated to $\Phi_{ip}(\cdot,\cdot)$. Furthermore, by denoting $S^{ip}(\cdot)$ the Single-Layer operator associated to $\Phi_{ip}(\cdot,\cdot)$ and using $(\ref{ASEq1233})$, we deduce that 
\begin{eqnarray*}
    S^{ip}(f)(x) &:=& \int_{\partial \Omega} \Phi_{ip}(x,y) \, f(y) \, d\sigma(y), \quad x \in \Omega, \\
    &\overset{(\ref{ASEq1233})}{=}& \int_{\partial \Omega} \left(\frac{1}{2} I \, + \, K^{ip} \right)\left(G_{p}(\cdot,y)\right)(x) \, f(y) \, d\sigma(y), \quad x \in \Omega, \\
    &=& \int_{\partial \Omega} G_{p}(x,y) \, \left(\frac{1}{2} I \, + \, K^{ip} \right)^{\star}\left(f\right)(y) \, d\sigma(y), \quad x \in \Omega, \\
    &=& \int_{\partial \Omega} G_{p}(x,y) \, \left(\frac{1}{2} I \, + \, K^{-ip} \right)\left(f\right)(y) \, d\sigma(y), \quad x \in \Omega, \\
    &\overset{(\ref{EqSL1225})}{=}& SL^{p}\left( \left(\frac{1}{2} I \, + \, K^{-ip} \right)\left(f\right)\right)(x), \quad x \in \Omega.
\end{eqnarray*}
As $f(\cdot)$ is an arbitrary function, we deduce the coming relation 
\begin{equation*}
    SL^{p}\left( f \right) \, = \, S^{ip}\left(\left(\frac{1}{2} I \, + \, K^{-ip} \right)^{-1}(f)\right), \quad  \text{in} \;\; \Omega,
\end{equation*}
hence by taking the normal derivative on both sides of the above equation we obtain 
\begin{equation*}
    \frac{\partial}{\partial \nu}\left[  SL^{p}\left( f \right)\right] \, = \, \frac{\partial}{\partial \nu}\left[ S^{ip}\left(\left(\frac{1}{2} I \, + \, K^{-ip} \right)^{-1}(f)\right)\right],
\end{equation*}
which, on the right hand side, by using the jumping properties for the Single-Layer operator $S^{ip}(\cdot)$, we deduce 
\begin{eqnarray*}
    \frac{\partial}{\partial \nu}\left[  SL^{p}\left( f \right)\right] \, &=& \, \left( \frac{1}{2} \, +  \, K^{ip} \right)^{\star} \left(\left(\frac{1}{2} I \, + \, K^{-ip} \right)^{-1}(f)\right) \\
    &=& \, \left( \frac{1}{2} \, +  \, K^{-ip} \right) \, \left(\left(\frac{1}{2} I \, + \, K^{-ip} \right)^{-1}(f)\right) \; = \; f, \quad \text{on} \;\; \partial \Omega.
\end{eqnarray*}
\textcolor{black}{
\subsection{Estimating $\left\Vert \textbf{R}  \right\Vert_{\mathbb{L}^{2}(\Omega)}$}
We recall from $(\ref{DefYDefS})$ that
\begin{equation*}
       \boldsymbol{R}(\cdot) \; := \; \sum_{m=1}^{M} \chi_{\Omega_m}(\cdot) Rest_{m},
\end{equation*}
with $Rest_{m}$ is given by $
(\ref{DefRestm})$. Then, using the fact that $\Omega_{m}$'s are disjoint sets, we obtain 
\begin{equation}\label{SBMBR}
      \left\Vert \boldsymbol{R} \right\Vert^{2}_{\mathbb{L}^{2}(\Omega)}  =  \sum_{m=1}^{M} \left\vert \Omega_{m} \right\vert  \left\vert Rest_{m} \right\vert^{2}  =  \left\vert \Omega_{m_{0}} \right\vert  \sum_{m=1}^{M}   \left\vert Rest_{m} \right\vert^{2}  =  a^{(1-h)}  \sum_{m=1}^{M}  \left\vert Rest_{m} \right\vert^{2}.
\end{equation}
Besides, by taking the absolute value in both sides of $(\ref{DefRestm})$, we obtain 
\begin{eqnarray}\label{HAEqua0757}
\nonumber
\left\vert Rest_{m} \right\vert & \lesssim &  \, \sum_{j = 1 \atop j \neq m}^{M} \left\vert \int_{D_{m}} W_{m}(x)  \int_{0}^{1}   \nabla G(z_{m}+t(x-z_{m});z_{j}) \cdot (x-z_{m}) \; dt \; dx \right\vert \, \left\vert \int_{D_{j}} v_{j}^{g}(y) \, dy \right\vert  \\ \nonumber 
&+&  \sum_{j = 1 \atop j \neq m}^{M} \left\vert \int_{D_{m}} W_{m}(x) \int_{D_{j}} \int_{0}^{1}   \nabla G(x;z_{j}+t(y-z_{j})) \cdot (y-z_{j}) \; dt \, v_{j}^{g}(y) \, dy \; dx \right\vert \\ \nonumber 
&+& a^{2} \, \left\vert \int_{D_{m}} W_{m}(x) \, \int_{0}^{1}   \nabla S_{m}(z_{m}+t(x-z_{m})) \cdot (x-z_{m}) dt \, dx \right\vert \\ \nonumber
&+& \left\vert \int_{D_{m}} W_{m}(x) \, \int_{D_{m}} \int_{0}^{1} \underset{y}{\nabla} \mathcal{R}(x,z_{m}+t(y-z_{m})) \cdot (y-z_{m}) \, dt \, v^{g}_{m}(y) \, dy \, dx\right\vert \\
&+&  a^{2}  \, \left\vert \int_{D_{m}} W_{m}(x) \int_{D} G(x,y) \, v^{g}(y) \, n^{2}(y) \, dy \, dx \right\vert.
\end{eqnarray}
In addition, for the third term on the right hand side, we have 
\begin{eqnarray*}
    \zeta_{3} \, &:=& \, a^{2} \, \left\vert \int_{D_{m}} W_{m}(x) \, \int_{0}^{1}   \nabla S_{m}(z_{m}+t(x-z_{m})) \cdot (x-z_{m}) dt \, dx \right\vert \\
    & \leq & \, a^{2} \, \left\Vert  W_{m} \right\Vert_{\mathbb{L}^{2}(D_{m})} \; \left\Vert  \int_{0}^{1}   \nabla S_{m}(z_{m}+t(\cdot-z_{m})) \cdot (\cdot-z_{m}) dt  \right\Vert_{\mathbb{L}^{2}(D_{m})} \\
    & \leq & \, a^{2} \, \left\Vert  W_{m} \right\Vert_{\mathbb{L}^{2}(D_{m})} \; \left[\int_{0}^{1} \frac{1}{t} \int_{B(z_{m},ta)} \left\vert \nabla S_{m}(y) \right\vert^{2} \, \left\vert y - z_{m} \right\vert^{2} \, dy \, dt \right]^{\frac{1}{2}},
\end{eqnarray*}
where $B(z_{m},ta)$ is a ball of center $z_{m}$ and radius $t \, a$. Then, 
\begin{eqnarray*}
    \zeta_{3} \, 
    & \leq & \, a^{3} \, \left\Vert  W_{m} \right\Vert_{\mathbb{L}^{2}(D_{m})} \; \left[\int_{0}^{1}  \int_{B(z_{m},ta)} \left\vert \nabla S_{m}(y) \right\vert^{2}  \, dy \, dt \right]^{\frac{1}{2}} \\
        & \leq & \, a^{3} \, \left\Vert  W_{m} \right\Vert_{\mathbb{L}^{2}(D_{m})} \; \left[\int_{0}^{1}  \int_{D_{m}} \left\vert \nabla S_{m}(y) \right\vert^{2}  \, dy \, dt \right]^{\frac{1}{2}},
\end{eqnarray*}
as $B(z_{m},t \, a) \subset D_{m}$. Then, 
\begin{equation}\label{HAEqua0756}
    \zeta_{3} \, 
     =  \, \mathcal{O}\left(  a^{3} \, \left\Vert  W_{m} \right\Vert_{\mathbb{L}^{2}(D_{m})} \; \left\Vert  \nabla S_{m}(y) \right\Vert_{\mathbb{L}^{2}(D_{m})} \right).
\end{equation}
Now, by using $(\ref{NMP1}), (\ref{SG})$ and $(\ref{HAEqua0756})$, we derive from the inequality $(\ref{HAEqua0757})$ the coming one, 
\begin{eqnarray*}
\left\vert Rest_{m} \right\vert & \lesssim & \,  \left\Vert W_{m} \right\Vert_{\mathbb{L}^{2}(D_{m})} \, \Bigg[ a^{4}  \,  \sum_{j = 1 \atop j \neq m}^{M} \frac{1}{\left\vert z_{m} - z_{j} \right\vert^{2}} \, \left\Vert v_{j}^{g}  \right\Vert_{\mathbb{L}^{2}(D_{j})} \, + a^{3} \, \left\Vert \nabla \mathscr{S}(g) \right\Vert_{\mathbb{L}^{2}(D_{m})} \\ &+& \, a \, \left[ \int_{D_{m}} \int_{D_{m}} \left\vert \nabla \mathcal{R}(x,y) \right\vert^{2} \, dy dx  \right]^{\frac{1}{2}} \, \left\Vert v^{g}_{m} \right\Vert_{\mathbb{L}^{2}(D_{m})} + \, a^{4} \, \left\Vert v^{g}_{m} \right\Vert_{\mathbb{L}^{2}(D_{m})} \, \Bigg] \\
&+&  \, a^{(\frac{9}{2}-h)} \, \sum_{j=1 \atop j \neq m}^{M} \frac{1}{\left\vert z_{m} - z_{j} \right\vert} \, \left\Vert v^{g}_{j} \right\Vert_{\mathbb{L}^{2}(D_{j})}. 
\end{eqnarray*}
Using Cauchy–Schwarz inequality and the estimation given by $(\ref{SID})$ we deduce
\begin{eqnarray}\label{Equa1040}
\nonumber
\left\vert Rest_{m} \right\vert \, & \lesssim & \,  \left\Vert W_{m} \right\Vert_{\mathbb{L}^{2}(D_{m})} \, \left[  a^{\frac{(10+2h)}{3}}   \,  \left\Vert v^{g}  \right\Vert_{\mathbb{L}^{2}(D)} \, + \, a^{3} \, \left\Vert \nabla \mathscr{S}(g) \right\Vert_{\mathbb{L}^{2}(D_{m})}  \right] \\ \nonumber
 & + & \,  \left\Vert W_{m} \right\Vert_{\mathbb{L}^{2}(D_{m})} \, \left[ a \, \left[ \int_{D_{m}} \int_{D_{m}} \left\vert \nabla \mathcal{R}(x,y) \right\vert^{2} \, dy dx  \right]^{\frac{1}{2}}  \, + \, a^{4}  \right] \, \left\Vert v^{g}_{m} \right\Vert_{\mathbb{L}^{2}(D_{m})} \\ &+& \, a^{\frac{(8-h)}{2}} \, \left\Vert v^{g} \right\Vert_{\mathbb{L}^{2}(D)}.
\end{eqnarray}
The following estimation holds, 
\begin{eqnarray}\label{nablaS(g)Dm} 
\nonumber
     \left\Vert \nabla \mathscr{S}(g) \right\Vert_{\mathbb{L}^{2}(D_{m})} & \leq & \, \left\Vert g \right\Vert_{\mathbb{H}^{-\frac{1}{2}}(\partial \Omega)} \, \left[ \int_{D_{m}} \, \left\Vert \nabla G(x,\cdot)  \right\Vert^{2}_{\mathbb{H}^{\frac{1}{2}}(\partial \Omega)} \, dx \right]^{\frac{1}{2}} \\ \nonumber
     & \overset{\ref{Equa1014IP}}{\leq} & \, \left\Vert g \right\Vert_{\mathbb{H}^{-\frac{1}{2}}(\partial \Omega)} \, \left[ \int_{D_{m}} \, \left\Vert \nabla G(x,\cdot)  \right\Vert^{2}_{\mathbb{H}^{1}( \Omega^{\diamond})} \, dx \right]^{\frac{1}{2}} \\ \nonumber
      & \lesssim & \, \left\Vert g \right\Vert_{\mathbb{H}^{-\frac{1}{2}}(\partial \Omega)} \, \left[ \int_{D_{m}} \, \frac{1}{\dist^{6}(x,\partial \Omega)}  dx \right]^{\frac{1}{2}} \\
      & \overset{(\ref{distto})}{\lesssim} & \, \left\Vert g \right\Vert_{\mathbb{H}^{-\frac{1}{2}}(\partial \Omega)} \, \frac{a^{\frac{3}{2}}}{\dist^{3}(D_{m};\partial \Omega)}.
\end{eqnarray}
Besides, similarly to $(\ref{FphOOEqua0645})$, for the $\nabla \mathcal{R}(\cdot,\cdot)$, we can prove that  
\begin{equation*}
    \left\vert \underset{y}{\nabla} \mathcal{R}(x,y) \right\vert \,   \lesssim  \, \left( \frac{1}{\dist(x,\partial \Omega)} \right)^{\frac{2}{3}} \, \left(   \frac{1}{\dist(y,\partial \Omega)} \right)^{\frac{4}{3}}, \quad \text{for} \quad x \neq y.
\end{equation*}
Then, 
\begin{equation}\label{2xintDmnablaR}
    \int_{D_{m}} \int_{D_{m}} \left\vert \nabla \mathcal{R}(x,y) \right\vert^{2} \, dy dx \,   \lesssim  \, \int_{D_{m}} \int_{D_{m}} \frac{1}{\dist^{\frac{4}{3}}(x, \partial \Omega)} \, \frac{1}{\dist^{\frac{8}{3}}(y, \partial \Omega)}  \, dy \, dx  
      \lesssim   \, \frac{\left\vert D_{m} \right\vert^{2}}{\dist^{4}(D_{m}, \partial \Omega)}.
\end{equation}
Hence, by returning to $(\ref{Equa1040})$, and using $(\ref{nablaS(g)Dm})$ and $(\ref{2xintDmnablaR})$, we derive the following estimation,    
\begin{eqnarray*}
    \left\vert Rest_{m} \right\vert \, & \lesssim &  \,  \left\Vert W_{m} \right\Vert_{\mathbb{L}^{2}(D_{m})} \, \Bigg[  a^{\frac{(10+2h)}{3}} \, \left\Vert v^{g} \right\Vert_{\mathbb{L}^{2}(D)} \ + \, \left\Vert g \right\Vert_{\mathbb{H}^{-\frac{1}{2}}(\partial \Omega)} \, \frac{a^{\frac{9}{2}}}{\dist^{3}(D_{m};\partial \Omega)}  \Bigg] \\
    & + &  \,  \left\Vert W_{m} \right\Vert_{\mathbb{L}^{2}(D_{m})} \, a^{4} \, \frac{1}{\dist^{2}(D_{m};\partial \Omega)}  \, \left\Vert v^{g}_{m} \right\Vert_{\mathbb{L}^{2}(D_{m})} + a^{\frac{(8-h)}{2}} \, \left\Vert v^{g} \right\Vert_{\mathbb{L}^{2}(D)} \\
    &\overset{(\ref{Equa1020})}{\lesssim}& \,  \left\Vert 1 \right\Vert_{\mathbb{L}^{2}(D_{m})} \, \Bigg[  a^{\frac{(4-h)}{3}} \, \left\Vert v^{g} \right\Vert_{\mathbb{L}^{2}(D)} \ + \, \left\Vert g \right\Vert_{\mathbb{H}^{-\frac{1}{2}}(\partial \Omega)} \, \frac{a^{\frac{(5-2h)}{2}}}{\dist^{3}(D_{m};\partial \Omega)}  \Bigg] \\
    & + &  \,  \left\Vert 1 \right\Vert_{\mathbb{L}^{2}(D_{m})}  \, \frac{a^{(2-h)}}{\dist^{2}(D_{m};\partial \Omega)}  \, \left\Vert v^{g}_{m} \right\Vert_{\mathbb{L}^{2}(D_{m})} + a^{\frac{(8-h)}{2}} \, \left\Vert v^{g} \right\Vert_{\mathbb{L}^{2}(D)}.
\end{eqnarray*}
Thus, by using $(\ref{SID+})$, we obtain
\begin{equation}\label{ell2-norm-Rest}
    \sum_{m=1}^{M} \left\vert Rest_{m} \right\vert^{2} \, \lesssim \, a^{\frac{(14+h)}{3}} \, \left\Vert v^{g} \right\Vert^{2}_{\mathbb{L}^{2}(D)} \, + \, a^{6} \, \left\Vert g \right\Vert^{2}_{\mathbb{H}^{-\frac{1}{2}}(\partial \Omega)} \, \overset{(\ref{ASEqua1304})}{\lesssim} \, a^{6} \, \left\Vert g \right\Vert^{2}_{\mathbb{H}^{-\frac{1}{2}}(\partial \Omega)}. 
\end{equation}
Then, by plugging the above estimation into $(\ref{SBMBR})$, we obtain
\begin{equation}\label{SBMBR+}
      \left\Vert \boldsymbol{R} \right\Vert_{\mathbb{L}^{2}(\Omega)}  \, = \, \mathcal{O}\left( a^{\frac{(7-h)}{2}} \, \left\Vert g \right\Vert_{\mathbb{H}^{-\frac{1}{2}}\left(\partial \Omega \right)} \right).
\end{equation}}
\end{document}